\documentclass[12pt]{article}
\usepackage{amsmath,amsfonts,amssymb,amsthm}
\usepackage{mathrsfs}
\usepackage{latexsym}
\usepackage[english]{babel}
\usepackage[usenames,dvipsnames]{xcolor}
\usepackage{comment}
\usepackage{appendix}

\usepackage{graphicx}

\renewenvironment{proof}[1][\proofname]{{\bfseries #1.} }{\qed}

\def\Cov{{\rm Cov\,}}

\newcommand{\Var}{{\rm Var}}
\newcommand{\Corr}{{\rm Corr}}

\newcommand{\e}{{\rm e}}

\def\authors#1{{ \begin{center} #1 \vspace{0pt} \end{center} } \smallskip}
\def\institution#1{{\sl \begin{center} #1 \vspace{0pt} \end{center} } }
\def\inst#1{\unskip $^{#1}$}
\def\title#1{{\huge\bf  \begin{center} #1 \vspace{0pt} \end{center}  } \smallskip}

\def\E{{\mathbb{ E}}}

\def\paref#1{(\ref{#1})}

\newtheorem{theorem}{Theorem}[section]
\newtheorem{proposition}[theorem]{Proposition}
\newtheorem{lemma}[theorem]{Lemma}
\newtheorem{conventions}[theorem]{Conventions}
\newtheorem{corollary}[theorem]{Corollary}
\newtheorem{definition}[theorem]{Definition}

\newtheorem{remark}[theorem]{Remark}
\newtheorem{example}[theorem]{Example}
\newtheorem{assu}[theorem]{Assumption}

\def\tr{\mathop{\rm tr}}

\begin{document}

\date{Jul 2021}

\title{\sc {Asymptotic distribution \\ of Nodal Intersections for ARW\\  against a Surface}}
\authors{\large Riccardo W. Maffucci\inst{\star,}\footnote{E-mail address: \texttt{riccardo.maffucci@epfl.ch} (Corresponding author).}, Maurizia Rossi\inst{\diamond,}\footnote{E-mail address: \texttt{maurizia.rossi@unimib.it}.}
}
\institution{\inst{\star}Institute of Mathematics, \'Ecole Polytechnique F\'ed\'erale de Lausanne\\
\inst{\diamond}Dipartimento di Matematica e Applicazioni, Universit\`a di Milano-Bicocca 
}

\begin{abstract}

We investigate Gaussian Laplacian eigenfunctions (Arithmetic Random Waves) on the three-dimensional standard flat torus, in particular the asymptotic distribution of the nodal intersection length against a fixed regular reference surface. Expectation and variance have been addressed by Maffucci (2019) who found that the expected length is proportional to the square root of the eigenvalue times the area of the surface, while the asymptotic variance only depends on the geometry of the surface, the projected lattice points being equidistributed on the two-dimensional unit sphere in the high-energy limit.  He also noticed that there are ``special'' surfaces, so-called static, for which the variance is of smaller order; however he did not prescribe the precise asymptotic law in this case.
In this paper, we study second order fluctuations of the nodal intersection length. Our first main result is a Central Limit Theorem for ``generic'' surfaces, while for static ones, a sphere or a hemisphere e.g., our main results are a non-Central Limit Theorem and a precise asymptotic law for the variance of the nodal intersection length, conditioned on the existence of so-called well-separated sequences of Laplacian eigenvalues. It turns out that, in this regime, the nodal area investigated by Cammarota (2019) is asymptotically fully correlated with the length of the nodal intersections against any sphere. The main ingredients for our proofs are the Kac-Rice formula for moments, the chaotic decomposition for square integrable functionals of Gaussian fields, and some arithmetic estimates that may be of independent interest. 

\smallskip


\smallskip


\end{abstract}

\section{Introduction}

The Laplace eigenvalues (or energy levels) of the three-dimensional standard flat torus $\mathbb T^3:= \mathbb R^3/ \mathbb Z^3$ are of the form $4\pi^2 m$, $m\in S$, where 
$$
S := \lbrace m : m = a_1^2 + a_2^2 + a_3^2, a_i \in \mathbb Z\rbrace. 
$$
Recall that a natural number $m$ is representable as a sum of three squares if and only if $m\ne 4^l (8k+7)$, for $k,l$ non-negative integers \cite{harwri}. 
For $m\in S$, let  
$$
\mathcal E = \mathcal E_m := \lbrace \mu=(\mu_1, \mu_2,\mu_3)\in \mathbb Z^3 : |\mu|^2 := \mu_1^2 + \mu_2^2 + \mu_3^2 = m\rbrace
$$
be the set of lattice points on the two-dimensional sphere of radius $\sqrt{m}$. It is invariant w.r.t. to the group
\begin{equation}
\label{group}
\mathcal{W}=\mathcal{W}_3, \qquad |\mathcal{W}|=48
\end{equation}
of \textit{signed permutations} (cf. \cite{rudwi2}), consisting of coordinate permutations and sign change of any coordinate.
The cardinality of $\mathcal E$ will be denoted by 
$$
N = N_m := |\mathcal E_m | = r_3(m),
$$
where $r_3(m)$ is the number of ways to express $m$ as a sum of three squares. 
For $m\ne 0,4,7 \ (\text{\rm mod\ } 8)$, one has 
\begin{equation}\label{Nbehavior}
(\sqrt m)^{1-\epsilon} \ll N \ll (\sqrt{m})^{1+\epsilon}
\end{equation}
for all $\epsilon >0$, see \cite[\S 1]{bosaru}. The condition $m\ne 0,4,7 \ (\text{\rm mod\ } 8)$, ensuring $N\to +\infty$, is natural, see the discussion in \cite[\S 1.1]{Maf18}.

It is well-known that the projected lattice points $\mathcal E_m/\sqrt{m}$ become equidistributed on the unit two-dimensional sphere $\mathcal S^2$ as  $m\to +\infty$ such that $m\ne 0,4,7 \ (\text{\rm mod\ } 8)$. More precisely, let $g$ be a $C^2$-smooth function on $\mathcal{S}^2$. For $m\to\infty$, $m\not\equiv 0,4,7 \pmod 8$, we have (\cite[Lemma 4.3]{Maf18}; cf. {\cite[Lemma 8]{pascbo}})
\begin{equation}
\label{equidi}
\frac{1}{N}\sum_{\mu\in\mathcal{E}}g\left(\frac{\mu}{|\mu|}\right)
=
\int_{z\in\mathcal{S}^2}g(z)\frac{dz}{4\pi}
+
O\left(\frac{1}{m^{1/28-\epsilon}}\right),
\end{equation}
where the constant involved in the $O$-notation only depends on $g$.

The complex-valued Laplace eigenfunctions of eigenvalue $4\pi^2 m$ may be written as 
\begin{equation}\label{def_aw}
G(x) = G_m(x) = \sum_{\mu\in \mathcal E_m} c_\mu e^{i 2\pi\langle \mu, x\rangle }, \qquad x\in \mathbb T^3,
\end{equation}
where $c_\mu\in \mathbb C$ are Fourier coefficients; the condition 
$ \overline{c_\mu} = c_{-\mu} $ ensures $G$ in (\ref{def_aw}) to be real-valued. 

\subsection{Nodal intersections}

The nodal set of $G$ in (\ref{def_aw}) is its zero-locus
$$
G^{-1}(0) = \lbrace x\in \mathbb T^3 : G(x) = 0 \rbrace.
$$
It is well-known that $G^{-1}(0)$ is a union of smooth surfaces, possibly together with a set of lower dimension (curves and points). Let $\Sigma\subset \mathbb T^3$ be a fixed smooth surface, we are interested in the intersection 
$$
G^{-1}(0) \cap \Sigma
$$
in the high-energy limit ($m\to +\infty$). Assume that $\Sigma$ is real-analytic, with nowhere zero Gauss-Kronecker curvature, then \cite{brgafa, brnoda} there exists $m_\Sigma$ such that for every $m\ge m_\Sigma$ 
\begin{itemize}
    \item the surface $\Sigma$ is not contained in the nodal set of any eigenfunction $G$;
    \item the $1$-dimensional Hausdorff measure of the intersection has the upper bound 
    $$
    h_1(G^{-1}(0)\cap \Sigma) \le C_\Sigma \sqrt{m}
    $$
   for some constant $C_\Sigma >0$ depending on the surface;
   \item for every eigenfunction $G$ it holds that
   $$
   G^{-1}(0) \cap \Sigma \ne \emptyset.
   $$
\end{itemize}

\subsection{Arithmetic random waves} 

We work with random Laplace eigenfunctions: it suffices to choose the Fourier coefficients $c_\mu$ in (\ref{def_aw}) to be random. Let us fix once for all a probability space $(\Omega, \mathcal F, \mathbb P)$; we denote by $\mathbb E$ expectation under $\mathbb P$. 
\begin{definition}\label{def_arw}
Let $m\in S$. The $m$-th \textbf{arithmetic random wave} on $\mathbb T^3$ is 
$$
F(x)=F_m(x):= \frac{1}{\sqrt{N_m}} \sum_{\mu\in \mathcal E_m}
a_\mu e^{i 2\pi\langle \mu, x\rangle }, \qquad x\in \mathbb T^3, $$
where $a_\mu$ are standard complex Gaussian random variables\footnote{i.e., $a_\mu = b_\mu + i c_\mu$, where $b_\mu$ and $c_\mu$ are i.i.d. real Gaussian random variables with zero mean and variance $1/2$.}, and independent save for the relations $\overline{a_\mu} = a_{-\mu}.$
\end{definition}
It is straightforward that $F$ is a stationary centered Gaussian random field on $\mathbb T^3$ whose covariance kernel is 
\begin{equation}\label{cov}
\mathbb E[F(x) F(y)] = \frac{1}{N} \sum_{\mu\in \mathcal E_m} e^{i 2\pi \langle \mu, x-y\rangle}=: r_m(x-y)=r(x-y), \qquad x,y\in \mathbb T^3.
\end{equation}
Arithmetic random waves on the three-dimensional flat torus have recently received some attention, see e.g. \cite{RWY16, benmaf, Maf18, Cam19, Maf20, Not20}.
\begin{definition}\label{def_Sigma} Let $\Sigma\subset \mathbb T^3$ be a fixed compact regular\footnote{i.e., a smooth two-dimensional submanifold of $\mathbb T^3$, possibly with boundary (see \cite{docarm, serne2}).} toral surface, of (finite) area $A:=|\Sigma|$.  Assume that $\Sigma$ admits a smooth normal vector locally. The \textbf{nodal intersection length} of the $m$-th arithmetic random wave $F$ against $\Sigma$ is the random variable
$$
\mathcal L = \mathcal L_m := h_1(F^{-1}(0) \cap \Sigma),
$$
where $h_1$ denotes the $1$-dimensional Hausdorff measure.
\end{definition} 
We are interested in the asymptotic behaviour of $\mathcal L$ as $m\to +\infty$.

\subsubsection{Prior work} 

In \cite{Maf18} expectation and variance of $\mathcal L$ in Definition \ref{def_Sigma} have been investigated. It holds that \cite[Proposition 1.2]{Maf18}
\begin{equation}\label{mean}
\mathbb E[\mathcal L] = \frac{\pi}{\sqrt 3} \sqrt{m} \cdot A
\end{equation}
for every $m\in S$. In addition assume that $\Sigma$ has nowhere zero Gauss-Kronecker curvature and define 
\begin{equation}\label{def_I}
 \mathcal I=\mathcal I_{2,\Sigma} := \iint_{\Sigma^2} \langle n(\sigma), n(\sigma')\rangle^2\,d\sigma d\sigma',
\end{equation}
where $n(\sigma)$ is the unit normal vector to $\Sigma$ at the point $\sigma$. 
As $m\to\infty$ s.t. $m\not\equiv 0,4,7 \pmod 8$, one has \cite[Theorem 1.3]{Maf18}
\begin{equation}
\label{vargen1}
\text{Var}(\mathcal{L})=\frac{\pi^2}{60}\frac{m}{N}\left[3\mathcal{I}-A^2+O\left(\frac{1}{m^{1/28-o(1)}}\right)\right];
\end{equation}
furthermore \cite[Proposition 1.4]{Maf18}
\begin{equation}\label{bounds}
\frac{A^2}{3}\le \mathcal I \le A^2
\end{equation}
ensuring the leading coefficient of (\ref{vargen1}) to be always non-negative and bounded. From (\ref{mean}) and (\ref{vargen1}) we deduce that the nodal intersection length is asymptotically concentrated at the mean value, i.e., for all $\epsilon >0$, as $m\to +\infty$ s.t. $m\not\equiv 0,4,7 \pmod 8$
$$
\mathbb P \left ( \left | \frac{\mathcal L}{\pi \sqrt{m} A/\sqrt 3} - 1\right | > \epsilon \right )  \mathop{\to} 0.
$$

There exist ``special" surfaces for which
\begin{equation}\label{eq_s}
\mathcal I = \frac{A^2}{3};
\end{equation}
for these (\ref{vargen1}) only gives that the true asymptotic behaviour of $\text{Var}(\mathcal L)$ is of lower order of magnitude than $m/N$ \cite[Section 1.6]{Maf18}. 
For corresponding results in the two-dimensional case see \cite{rudwig}.

\subsection{Statement of the main results}

Our principal results concern the limiting laws for the nodal intersection length in Definition \ref{def_Sigma}. Theorem \ref{th1} asserts the Central Limit Theorem for $\mathcal L$ in the high-energy limit under the assumption that the leading term $3\mathcal I - A^2$ of the variance in (\ref{vargen1}) is strictly positive. Theorem \ref{th2} investigates the alternative situation (\ref{eq_s}) with a non-Gaussian limiting distribution for so-called well separated sequences of energy levels (see Definition \ref{welsep} and Assumption \ref{theassu}), and among other things determines the true asymptotic behaviour of $\text{Var}(\mathcal L)$ in the latter case, in particular complementing (\ref{vargen1}). For corresponding results in the two dimensional case see \cite{roswig}.
\begin{conventions}\rm
In this manuscript we will denote by $\mathop{\to}^d$ (resp. $\mathop{=}^d$) convergence (resp. equality) in distribution, $\mathcal N(\lambda, \sigma^2)$ will stand for a standard Gaussian random variable with mean $\lambda\in \mathbb R$ and variance $\sigma^2$. 
\end{conventions}
\begin{theorem}\label{th1}
Let $\Sigma\subset \mathbb T^3$ be a surface as in Definition \ref{def_Sigma}, with nowhere zero Gauss-Kronecker curvature. Assume moreover that $3 \mathcal I > A^2$. Then the limiting distribution of the nodal intersection length is Gaussian, i.e., as $m\to +\infty$ s.t. $m\not\equiv 0,4,7 \pmod 8$,
\begin{equation}
     \frac{\mathcal L - \mathbb E[\mathcal L]}{\sqrt{\text{\rm Var}(\mathcal L)}} \mathop{\longrightarrow}^d Z,
\end{equation}
where $Z\sim \mathcal N(0,1)$. 
\end{theorem}
Let us now investigate the alternative situation (\ref{eq_s}). For our purposes it is convenient to rewrite \eqref{vargen1} as \cite[Proof of Theorem 1.3]{Maf18}
\begin{equation}
\label{eq:varh}
\text{Var}(\mathcal{L})=\frac{\pi^2}{24}\frac{m}{N}\left(9\mathcal H(\nu) - A^2\right ) + O\left (\frac{1}{m^{3/16-o(1)}}\right),
\end{equation}
where 
$\nu$ denotes the probability measure on the unit two-dimensional sphere $\mathcal S^2$ induced by the set of frequencies $\mathcal E$, namely 
\begin{equation}\label{nu}
    \nu =\nu_m := \frac{1}{N}\sum_{\mu\in \mathcal E}
 \delta_{\mu/\sqrt m},
 \end{equation}
and 
\begin{eqnarray}\label{H}
  \mathcal H(\nu) &=& \mathcal H_\Sigma(\nu)  := \frac{1}{N } \sum_{\mu\in \mathcal E}\int_\Sigma \int_\Sigma \left \langle \frac{\mu}{|\mu|}, n(\sigma)\right \rangle^2 \left \langle \frac{\mu}{|\mu|}, n(\sigma')\right \rangle^2\, d\sigma  d\sigma'\cr
  &=& \int_{\mathcal S^2} \left ( \int_\Sigma \left \langle \theta, n(\sigma)\right \rangle^2\, d\sigma \right )^2\,\nu(d\theta)
\end{eqnarray}
where $n(\sigma)$ is the unit normal vector to $\Sigma$ at the point $\sigma$. Indeed, in light of lattice point equidistribution \eqref{equidi}, as $m\to\infty$, $m\not\equiv 0,4,7 \pmod 8$, one has the estimate (\cite[Lemma 5.3]{Maf18})
\begin{equation}
\label{IH}
\mathcal{H}(\nu)=\frac{1}{15}(A^2+2\mathcal{I})+O\left(\frac{1}{m^{1/28-\epsilon}}\right).
\end{equation}

One may also define $\mathcal H(\eta)=\mathcal H_\Sigma(\eta)$ for any probability measure $\eta$ on $\mathcal S^2$ invariant w.r.t. the group \eqref{group}, with $\eta$ in place of $\nu$ on the r.h.s. of (\ref{H}).
\begin{definition}\label{def_static}
Let $\Sigma\subset \mathbb T^3$ be a surface as in Definition \ref{def_Sigma}, with nowhere zero Gauss-Kronecker curvature. We say that $\Sigma$ is a \textbf{static surface} if 
\begin{equation*}
\mathcal{H}_\Sigma(\eta)=\frac{A^2}{9}
\end{equation*}
for any probability measure $\eta$ on $\mathcal S^2$ invariant w.r.t. the group \eqref{group}.
\end{definition}
For static surfaces, in particular $\mathcal{I}=A^2/3$.
%
For instance any sphere or hemisphere are static. For static surfaces $\Sigma\subset\mathbb{T}^3$ of area $A$, as $m\to\infty$, $m\not\equiv 0,4,7 \pmod 8$, the bound
\begin{equation*}
\text{\rm Var}(\mathcal{L})=O\left(\frac{m}{N^2}N^{5/8+\epsilon}\right)
\end{equation*}
is straightforward from \eqref{eq:varh}. The variance is of smaller order than for non-static surfaces. 

Let $\ell\ge 2$ be an integer, the set of $\ell$-th lattice point correlations, or $\ell$-th correlations for short, is 
\begin{equation}\label{l-corr}
    \mathcal C(\ell) = \mathcal C_m(\ell) := \left \lbrace (\mu^{(1)}, \mu^{(2)}, \dots, \mu^{(\ell)})\in \mathcal E^\ell : \sum_{i=1}^\ell \mu^{(i)} = 0 \right \rbrace.
\end{equation}
\begin{definition}
\label{welsep}
We call a sequence of eigenvalues \textbf{well separated} if for $\ell=2,4,6$ we have the bounds
\begin{equation}\label{frakS}
\mathfrak{S}_\ell:=\frac{1}{N^{\ell}}\sum_{\mathcal{E}^\ell\setminus\mathcal{C}(\ell)}\frac{1}{\left|\sum_{i=1}^{\ell}\mu_i\right|^2}=o\left(\frac{1}{N^{2}}\right)
\end{equation}
along the sequence, where $\mathcal C(\ell)$ is the set of $\ell$-th correlations in (\ref{l-corr}).
\end{definition}
This is the critical regime for these arithmetic sums: on one hand, a result by Bourgain, Sarnak, and Rudnick \cite[Theorem 1.1]{bosaru} implies (\cite[Section 5.2]{Maf18})
\begin{equation*}
\mathfrak{S}_2=O\left(\frac{N^\epsilon}{N^2}\right).
\end{equation*}
On the other hand, the statement
\begin{equation*}
\mathfrak{S}_2=O\left(\frac{1}{N^{2}}\right)
\end{equation*}
is false \cite{prirud} -- see Section \ref{future} for further details.

\begin{assu}
\label{theassu}
There exists a well separated sequence of eigenvalues.
\end{assu}

In order to state our second main result we need to introduce some more notation. For $k$ a natural number we define 
\begin{equation}
\label{Ik}
\mathcal I_k=\mathcal{I}_{k,\Sigma}:=\iint_{\Sigma^2}\langle {n}(\sigma), {n}(\sigma')\rangle^k\,d\sigma d\sigma',
\end{equation}
consistently with (\ref{def_I}). Plainly, $\mathcal I_2 = \mathcal I$.

\begin{theorem}\label{prop1}
\label{varthm}
Let $\Sigma\subset\mathbb{T}^3$ be a static surface as in Definition \ref{def_static} of area $A$. Then we have, as $m\to\infty$, $m\not\equiv 0,4,7 \pmod 8$ along a well separated sequence of eigenvalues,
\begin{equation}
\label{var}
\text{\rm Var}(\mathcal{L})=\frac{\pi^2}{9600}\frac{m}{N^2}\left[81 \mathcal{I}_{4}+35 A^2+o(1)\right].
\end{equation}
\end{theorem}
Note that $81\mathcal I_4 + 35A^2>0$. Let us now define the following random variable 
\begin{equation}\label{M}
 {\mathcal M} := \frac{1}{\sqrt{\frac{8}{225}(81 \mathcal I_4 + 35A^2)}}\bigg (\sum_{i\le j} c_{ij} (X_{ij}^2-1) + \sum_{\substack{i\le j, l\le k\\ (i,j)\ne (l,k)}} c_{ijlk} X_{ij}X_{lk}\bigg ),
\end{equation}
where $X_{ij}$, $i,j\in \lbrace 1,2,3\rbrace, i\le j$ are i.i.d. standard Gaussian random variables, and ($\int = \int_{\Sigma}$)
\begin{eqnarray*}
&&c_{11}=0,\quad 
c_{12}=-\frac{1}{5}\int(3( n_1^2-n_3^2)^2+4n_2^2)d\sigma,\quad c_{13}=-\int(\frac{11}{15}-2n_2^2+\frac{9}{5}n_2^4)d\sigma,\cr
&&c_{22}=-\frac{4}{5}\int (n_1^2+3n_2^2n_3^2)d\sigma,\quad 
c_{23}=-\frac{4}{5}\int (n_2^2+3n_1^2n_3^2)d\sigma,\quad c_{33}=-\frac{4}{5}\int (n_3^2+3n_1^2n_2^2)d\sigma,\cr
&&c_{1112}=-\frac{16\sqrt{15}}{15}\int( n_1^2-n_3^2)d\sigma,\quad 
c_{1113}=-\frac{16\sqrt{5}}{15}\int(1-3n_2^2)d\sigma,\quad c_{1122}=\frac{32\sqrt{15}}{15}\int n_2n_3d\sigma,\cr
&&c_{1123}=\frac{32\sqrt{15}}{5}\int n_1n_3d\sigma, \quad 
c_{1133}=\frac{32\sqrt{15}}{5}\int n_1n_2d\sigma,\quad
c_{1213}=\frac{2\sqrt{3}}{15}\int (n_1^2-n_3^2)(1+9n_2^2)d\sigma,\cr
&&c_{1222}=\frac{4}{5}\int n_2n_3(2+3(n_1^2-n_3^2))d\sigma,\quad 
c_{1223}=\frac{12}{5}\int n_1n_3(n_1^2-n_3^2)d\sigma,\cr &&c_{1233}=\frac{4}{5}\int n_1n_2(-2+3(n_1^2-n_3^2))d\sigma,\quad c_{1322}=\frac{4\sqrt{3}}{15}\int n_2n_3(5-9n_2^2)d\sigma, \cr
&&c_{1323}=\frac{4\sqrt{3}}{15}\int n_1n_3(-1-9n_2^2)d\sigma, \quad 
c_{1333}=\frac{4\sqrt{3}}{15}\int n_1n_2(5-9n_2^2)d\sigma,\cr
&&c_{2223}=\frac{8}{5}\int n_1n_2(1-3n_3^2)d\sigma, \quad c_{2233}=\frac{8}{5}\int n_1n_3(1-3n_2^2)d\sigma,\quad c_{2333}=\frac{8}{5}\int n_2n_3(1-3n_1^2)d\sigma.
\end{eqnarray*}
\begin{theorem}\label{th2}
Let $\Sigma\subset\mathbb{T}^3$ be a static surface as in Definition \ref{def_static} of area $A$. Then we have, as $m\to\infty$, $m\not\equiv 0,4,7 \pmod 8$ along a well separated sequence of eigenvalues,
\begin{equation}
   \frac{\mathcal L - \mathbb E[\mathcal L]}{\sqrt{\text{\rm Var}(\mathcal L)}} \mathop{\longrightarrow}^d \mathcal M,
    \end{equation}
    where $\mathcal M$ is defined as in (\ref{M}).
\end{theorem}
\begin{example}\rm For $\Sigma$ the unit sphere, $A=4\pi$ and $\mathcal{I}_{4,\Sigma}=16\pi^2/5$ so that as $m\to\infty$, $m\not\equiv 0,4,7 \pmod 8$ along a well separated sequence of eigenvalues,
\begin{equation*}
\text{Var}(\mathcal{L})\sim \frac{32\pi^4}{375}\frac{m}{N^2}.
\end{equation*}
Furthermore, the only non-zero coefficients in \eqref{M} are
$$c_{12}=c_{13}=c_{22}=c_{23}=c_{33} 
= -\frac{128}{75}\pi,$$ 
hence still along a well separated sequence of eigenvalues 
\begin{equation}\label{limitesfera}
   \frac{\mathcal L - \mathbb E[\mathcal L]}{\sqrt{\text{\rm Var}(\mathcal L)}} \mathop{\longrightarrow}^d \frac{5 -\chi^2(5)}{\sqrt{10}},
\end{equation}
where $\chi^2(5)$ is a chi-square random variable with $5$ degrees of freedom. 
\end{example}
\begin{remark}\label{remark1}\rm
It is worth noticing that the limiting distribution in (\ref{limitesfera}) of the nodal intersection length for arithmetic random waves in dimension $3$ against the unit sphere along a well separated sequence of eigenvalues coincides with the limiting distribution of their nodal area \cite[Theorem 1]{Cam19} as $m\to\infty$, $m\not\equiv 0,4,7 \pmod 8$. We observe the same phenomenon in dimension two: the nodal length \cite[Theorem 1.1]{MPRW} and the number of nodal intersection for arithmetic random waves against the unit circle \cite[Example 1.4]{roswig} along a ``generic'' sequence of eigenvalues share the same limiting distribution (a chi-square with $2$ degrees of freedom). 
\end{remark}
Actually, the connection pointed out in Remark \ref{remark1} is deeper \cite{primar}.  
\begin{corollary}\label{corD}
Let $\Sigma\subset\mathbb{T}^3$ be the two-dimensional unit sphere, and $\mathcal A=\mathcal A_m$ denote the area of the nodal set $F^{-1}_m(0)$. Then we have, as $m\to\infty$, $m\not\equiv 0,4,7 \pmod 8$ along a well separated sequence of eigenvalues, 
\begin{equation*}
    \Corr(\mathcal L, \mathcal A) \longrightarrow 1. 
\end{equation*}
\end{corollary}
We do believe that the same phenomenon holds true in dimension two: for ``generic'' sequences of eigenvalues, there is asymptotic \emph{full} correlation between the nodal length \cite{MPRW} and the number of nodal intersections against the unit circle \cite{roswig} for arithmetic random waves on the two-dimensional standard flat torus. 
%

\subsection{Future directions}\label{future}

To prove Assumption \ref{theassu}, we would like to construct, or even show existence of, well separated sequences. This seems to be a difficult problem. In fact, there are arbitrarily large values of $m$ satisfying \cite{prirud}
\begin{equation*}
\mathfrak{S}_2=\Omega\left(\frac{\log(m)}{N^2}\right),
\end{equation*}
say. Indeed, for $m$ such that $m-1=x^2+y^2$ is the sum of two squares, the lattice points $(x,y,\pm 1)$ have distance two, hence
\begin{equation*}
\sum_{\mu\neq\mu'}\frac{1}{\left|\mu-\mu'\right|^2}\geq\sum_{\substack{(x,y)\in\mathbb{Z}^2\\x^2+y^2=m-1}}\frac{1}{2^2}=\frac{r_2(m-1)}{4},
\end{equation*}
where $r_2$ is the arithmetic function counting the number of ways to express a number as sum of two integer squares. It is well-known that $r_2$ is not bounded by any power of $\log$ \cite[Theorems 337 and 338]{harwri}. One can draw similar conclusions for $\mathfrak{S}_4, \mathfrak{S}_6$.

\subsection*{Acknowledgements}

The authors would like to thank Igor Wigman for suggesting this problem and this collaboration. Many thanks to Jacques Benatar, Matthew De Courcy-Ireland, Domenico Marinucci, Zeev Rudnick, Maryna Viazovska, Richard Wade, Nadav Yesha for helpful discussions.
The research leading to this work received fundings from the LMS Grant \emph{Research in Pairs}, scheme 4 (Mar-Sep 2018). 

The research of R.M. has been partially supported by the Engineering \& Physical Sciences Research Council (EPSRC) Fellowship EP/M002896/1 held by Dmitry Belyaev, and by Swiss National Science Foundation project 200021\_184927. The research of M.R. has been partially supported by the FSMP and the ANR-17-CE40-0008 Project \emph{Unirandom}.

\section{Outline of the paper}

\subsection{On the proofs of the main results}

The proofs of Theorem \ref{th1} and Theorem \ref{th2} rely on the \emph{chaotic expansion} for the nodal intersection length (see \cite[\S 2]{nopebook} for a complete discussion on Wiener chaos and related topics). This technique has been exploited for studying a wide variety of functionals of Gaussian processes, see e.g. \cite{kraleo, marwig11, EL16, MPRW, roswig, CM18, Cam19, Not20, Not21} and the references therein, however we shall encounter some additional difficulties, see below for a preview. The starting point is the formal expression
\begin{equation}\label{formal}
\mathcal L = \int_{\Sigma} \delta_0(F(\sigma)) | \nabla_{\Sigma} F(\sigma)|\, d\sigma, 
\end{equation}
 where $\delta_0$ denotes the Dirac mass at zero. Here $\nabla_{\Sigma} F$ stands for the surface gradient of $F$, i.e.
\begin{equation}\label{surface_gradient}
\nabla_\Sigma F(\sigma) := \nabla F(\sigma) - \langle \nabla F(\sigma), n(\sigma)\rangle n(\sigma), \qquad \sigma\in \Sigma,
\end{equation}
and $| \cdot |$ denotes the Euclidean norm in $\mathbb R^3$. 
One of the above mentioned additional difficulties is due to the fact that the components of $\nabla_\Sigma F(\sigma)$ are \emph{not} independent random variables, as explained in \S \ref{sec_chaos}. 
The finite-variance random variable $\mathcal L$ in (\ref{formal}) admits the Wiener-\^Ito chaotic expansion (see \S \ref{sec_chaos}) of the form 
\begin{equation}\label{wi}
\mathcal L = \sum_{q=0}^{+\infty} \mathcal L[2q],
\end{equation}
the above series converging in $L^2(\mathbb P)$. The core of this expansion is roughly the fact that the family of Hermite polynomials \cite[\S 1.4]{nopebook}
\begin{equation}\label{defH}
   H_q(t) = (-1)^q \text{e}^{t^2/2} \frac{d^q}{dt^q} \text{e}^{-t^2/2},\qquad q\in \mathbb N,
\end{equation}
is an orthonormal basis for the space of square integrable functions on the real line with respect to the Gaussian density.  
In particular, $\mathcal L[2q]$ is the orthogonal projection of $\mathcal L$ onto the so-called $2q$-th Wiener chaos, $\mathcal L[2q]$ and $\mathcal L[2q']$ are uncorrelated random variables whenever $q\ne q'$, and $\mathcal L[0] = \mathbb E[\mathcal L]$. 

Under the assumptions in Theorem \ref{th1}, evaluating the second chaotic projection $\mathcal L[2]$ yields that the variance of the nodal intersection length (\ref{vargen1}) is asymptotic, in the high-energy limit, to the variance of $\mathcal L[2]$. The latter result and the orthogonality of the Wiener chaoses imply that the distribution of $\mathcal L[2]$ asymptotically dominates the series on the right hand side of (\ref{wi}), and a standard Central Limit Theorem result for $\mathcal L[2]$ allows to infer the statement of Theorem \ref{th1}. 

Assume now that the surface is static according to Definition \ref{def_static}. In order to obtain the true asymptotic behaviour of $\text{\rm Var}(\mathcal L)$, we inspect the proof of the approximate Kac-Rice formula \cite[\S 3]{Maf18}, and obtain one extra term in the Taylor expansion of the two-point correlation function, see \S \ref{sec_KR}. One of the main difficulties is how to control ``off-diagonal'' terms coming from the fourth moment (and the sixth one) of the covariance kernel $r$ in (\ref{cov}), and specifically the quantity
\begin{equation}\label{quantity}
\frac{1}{N^2}\sum_{(\mu,\mu',\mu'',\mu''')\in\mathcal E^4\setminus \mathcal C(4)} \left | \int_{\Sigma} e^{i 2\pi \langle \mu+\mu'+\mu''+\mu''', \sigma\rangle}\, d\sigma \right |^2,
\end{equation}
where $\mathcal C(4)$ is the set of fourth lattice point correlations, see (\ref{l-corr}). 
Under Assumption \ref{theassu}, the quantity (\ref{quantity}) is $o(1)$ at least for one subsequence of energy levels, so that we can establish Theorem \ref{prop1} for such eigenvalues.
It turns out that the leading term in the chaotic expansion (\ref{wi}) for the nodal intersection length is no longer the projection onto the second chaos, but $\mathcal L[4]$, the projection  onto the fourth one. A precise analysis of the latter allows to get its asymptotic non-Gaussian distribution in (\ref{M}), thus concluding the proof of Theorem \ref{th2}. 

\subsection{Kac-Rice formulae}

Given a random field, call $\mathcal{V}$ the measure of its nodal set. When certain assumptions are met, the moments of $\mathcal{V}$ may be computed via Kac-Rice formulas \cite[Theorems 6.8 and 6.9]{azawsc}. In \cite{Maf18}, Kac-Rice formulas for a random field defined on a {\em surface} were developed. For the second moment, the {\bf approximate Kac-Rice formula} \cite[Proposition 1.7]{Maf18} was found for surfaces of nowhere vanishing Gauss-Kronecker curvature. The problem of computing the nodal intersection length variance \eqref{vargen1} for arithmetic waves was thus reduced to estimating the second moment of the covariance function $r$ \eqref{cov} and of its various first and second order derivatives. The error term in \eqref{vargen1} comes from bounding the fourth moment of $r$ and of its derivatives.

In the present paper, we prove a more precise approximate Kac-Rice, where the main term comes from the \textit{fourth} moment of $r$ and its derivatives, and the error is controlled via the \textit{sixth} moment.

\begin{proposition}[Kac-Rice approximate formula]
\label{KR}
For a static 
surface $\Sigma$, 
we have
\begin{multline*}
\Var(\mathcal{L})=M
\iint_{\Sigma^2}d\sigma d\sigma'
\left\{
\frac{1}{8}r^2
+\frac{1}{16}\tr(X)
+\frac{1}{16}\tr(X')
+\frac{1}{32}\tr(Y'Y)
+\frac{3}{32}r^4
\right.
\\
\left.
+\frac{1}{2^{11}}
\left[32\tr(X)\tr(X')-16\tr(XYY')-16\tr(X'Y'Y)-24(\tr X)^2-24(\tr X')^2
+2\tr(Y'YY'Y)
\right.\right.
\\
\left.\left.
+(\tr Y'Y)^2-8\tr(X)\tr(Y'Y)-8\tr(X')\tr(Y'Y)
\right]
+\frac{1}{64}
\left[2r^2\tr(X)+2r^2\tr(X')+r^2\tr(Y'Y)\right]
\right\}
\\+m\cdot o\left(\frac{1}{N^2}\right)+m\cdot O_\Sigma(\mathfrak{S}_6).
\end{multline*}
\end{proposition}
Proposition \ref{KR} will be proven in Section \ref{sec_KR}.

\subsection{The arithmetic formula}\label{sec_arithmetic}

After the application of the Kac-Rice approximate formula, we arrive at the arithmetic part of our argument. Here the results of \cite{benmaf} come into play. Recall the definition of $\ell$-correlations $\mathcal C_m(\ell)$ in (\ref{l-corr}). 
The {\bf non-degenerate} $\ell$-correlations are those such that no proper subsum vanishes,
\begin{equation}\label{defxl}
\mathcal X(\ell) = \mathcal{X}_m(\ell):=\Big\{(\mu^{(1)},\dots,\mu^{(\ell)})\in\mathcal{C}_m(\ell) : \forall I\subsetneq\{1,\dots,\ell\}, \sum_{i\in I}\mu^{(i)}\neq 0\Big\}.
\end{equation}

Now for $\ell$ even, $\ell$-correlations are related to the $\ell$-th moment of the covariance function \eqref{cov},
\begin{equation}
\label{rk}
\mathcal{R}(\ell)=\mathcal{R}_m(\ell):=\int_{\mathbb{T}^d}|r^\ell(x)|dx=\frac{|\mathcal{C}_m(\ell)|}{N^\ell}.
\end{equation}

For the purposes of this paper, one requires non-trivial upper bounds for the cardinalities $|\mathcal{X}(4)|$ and $|\mathcal{C}(6)|$. A geometric argument yields readily \cite[(4.8)]{Maf18}
\begin{equation}
\label{fourcorrbis}
|\mathcal{X}_m(4)|\leq|\mathcal{C}_m(4)|=O(N^{2+o(1)})
\end{equation}
and similarly
\begin{equation*}
|\mathcal{C}_m(6)|=O(N^{4+o(1)}).
\end{equation*}
To compare with lower bounds, the complement of $\mathcal{X}(4)$ is the set of the symmetric $4$-spectral correlations $\mathcal{D}'(4)$, of order
\begin{equation*}
|\mathcal{D}'(4)|\gg N^{2}.
\end{equation*}
Note that these are the correlations that cancel out in pairs.
\\
Similarly,
\begin{equation*}
|\mathcal{C}(6)|\gg N^{3}.
\end{equation*}

In \cite{Maf18}, the upper bound \eqref{fourcorrbis} on $|\mathcal{C}_m(4)|$ was sufficient. This work is markedly different in this aspect, as we actually need
\begin{equation*}
|\mathcal{X}_m(4)|=o(N^{2}) \qquad\text{ and }\qquad
|\mathcal{C}_m(6)|=o(N^{4}).
\end{equation*}
In \cite[Theorems 1.6 and 1.7]{benmaf}, it was proven that, as $m\to\infty$,
\begin{equation}
\label{fourcorr}
|\mathcal{X}_m(4)|=O(N^{7/4+o(1)})
\end{equation}
and
\begin{equation}\label{sixcorr}
|\mathcal{C}_m(6)|=O(N^{11/3+o(1)}).
\end{equation}
These will be applied in Section \ref{secarith}, to prove the following.

\begin{proposition}[The arithmetic formula]
\label{arith}
For a static surface $\Sigma$, we have as $m\to\infty$, $m\not\equiv 0,4,7 \pmod 8$,
\begin{equation}
\label{var}
\text{\rm Var}(\mathcal{L})=\frac{\pi^2}{9600}\frac{m}{N^2}\left[81 \mathcal{I}_{4}+35 A^2+o(1)\right]+m\cdot O(\mathfrak{S}_2+\mathfrak{S}_4+\mathfrak{S}_6).
\end{equation}
\end{proposition}

\begin{proof}[Proof of Theorem \ref{varthm}]
Following the arithmetic formula, the variance asymptotic is now established for well separated sequences.
\end{proof}


\subsection{Chaos expansion}\label{sec_chaos}

In this part we compute the chaotic expansion (\ref{wi}) for the nodal intersection length $\mathcal L$, see e.g. \cite[\S 2]{nopebook} for details on Wiener chaos. 
Recall the formal representation (\ref{formal}), and define for $\varepsilon >0$ the $\varepsilon$-approximating nodal intersection length as
\begin{equation}\label{approx}
    \mathcal L^\varepsilon = \mathcal L_m^\varepsilon := \int_{\Sigma} \delta_0^\varepsilon(F(\sigma)) | \nabla_{\Sigma} F(\sigma)|\, d\sigma,
\end{equation}
where 
\begin{equation*}
    \delta_0^\varepsilon := \frac{1}{2\varepsilon} \mathbf{1}_{[-\varepsilon, \varepsilon]}
\end{equation*}
is an approximation of $\delta_0$ at level $\varepsilon$. From \cite[Proposition 2.3]{Maf18} and the (uniform over $\varepsilon$) upper bound
\begin{equation*}
    \mathcal L^\varepsilon \le 18 \sqrt{m} 
\end{equation*}
showed in \cite[Lemma 2.5]{Maf18} we immediately have the following (thus justifying (\ref{formal})). 
\begin{lemma}\label{lem1}
For $m\in S$ we have 
\begin{equation*}
   \lim_{\varepsilon \to 0} \mathcal L^\varepsilon =\mathcal L,
\end{equation*}
both a.s. and in $L^2(\mathbb P)$.
\end{lemma}
A straightforward differentiation of $F$ in Definition \ref{def_arw} leads to 
\begin{equation*}
    \nabla F(x) = \frac{i2\pi}{\sqrt{N}} \sum_{\mu\in \mathcal E} \mu\, a_\mu e^{i2\pi\langle \mu, x\rangle},\qquad x\in \mathbb T^3.
\end{equation*}
It is easy to check that for each $x\in \mathbb T^3$, the random variables $F(x)$ and $\nabla F(x)$ are (Gaussian) independent, and, for each $\sigma\in \Sigma$, the same holds true for $F(\sigma)$ and $\nabla_{\Sigma} F(\sigma)$ (see (\ref{surface_gradient})). The latter property is very useful since it simplifies computations a lot. However the components of $\nabla_{\Sigma} F(\sigma)$
are \emph{not} independent random variables for $\sigma\in \Sigma$ (on the contrary the components of $\nabla F(x)$ are independent for each $x\in \mathbb T^3$). In order to overcome this additional difficulty, in what follows we will first write the (random) quantity  $| \nabla_{\Sigma} F(\sigma)|$ in terms of the Euclidean norm of a standard Gaussian vector thus reducing the question to a two-dimensional problem, and then derive the chaotic expansion of $\mathcal L^\varepsilon$ taking advantage of \cite[Proposition 3.2]{MPRW}. Letting $\varepsilon$ tend to zero, thanks to Lemma \ref{lem1}, we will find the chaotic expansion of $\mathcal L$. 

\subsubsection{Covariance matrices}

Let us first introduce some more notation. We denote by $\partial^\Sigma_i$ the components of the surface gradient, i.e. 
$$
\nabla_\Sigma F = (\partial_1^\Sigma F, \partial_2^\Sigma  F, \partial_3^\Sigma F)^t,
$$
where, from \paref{surface_gradient}, for $i=1,2,3$
$$
\partial_i^\Sigma F = \partial_i F - n_i \sum_{j=1}^3 n_j \partial_j F,
$$
with the convention $\partial_i := \partial / \partial x_i$ (recall that $n(\sigma)$ is the unit normal vector to $\Sigma$ at the point $\sigma$). 
Since for every $\sigma\in \Sigma$ it holds that 
$
\langle \nabla_\Sigma F(\sigma), n(\sigma) \rangle = 0,
$
assuming w.l.o.g. that $n_3$ is non-zero, we have
$$
\partial_3^\Sigma F = -\frac{n_1}{n_3} \partial_1^\Sigma F -\frac{n_2}{n_3} \partial_2^\Sigma F. 
$$
We can write (see also the proof of Lemma 3.5 in \cite{Maf18}) 
\begin{equation}
| \nabla_{\Sigma} F|^2 = \left (\partial_1^\Sigma F, \partial_2^\Sigma  F \right ) \overline{\Omega}^{-1} \left (\partial_1^\Sigma F, \partial_2^\Sigma  F \right )^t,
\end{equation}
where 
$$
\overline{\Omega}:= \begin{pmatrix} n_2^2 + n_3^2 &-n_1 n_2\\
-n_1n_2 &n_1^2 + n_3^2
\end{pmatrix}.
$$
Note that $\overline{\Omega}$ is the covariance matrix of the rescaled vector $\left (\widetilde \partial_1^\Sigma F, \widetilde \partial_2^\Sigma  F \right )$, where 
$$
\widetilde \partial_i^\Sigma F := \frac{1}{\sqrt M} \partial_i^\Sigma F;
$$
here $M:= 4\pi^2m/3$ is the variance of $\partial_i F(x)$.
It is immediate to check that $\det(\overline{\Omega}) = n_3^2$ and the eigenvalues are $1$ and $n_3^2$. Note that if $n_3^2 = 1$, then $\overline{\Omega} = \text{\rm Id}$ the identity matrix, and the vector $\left (\widetilde \partial_1^\Sigma F, \widetilde \partial_2^\Sigma  F \right )$ is a bivariate standard Gaussian. 
Let us hence assume $n_3^2 \ne 1$. We are going to diagonalize $\overline{\Omega}$ through orthogonal matrices. Let us consider the orthogonal matrix 
$$
O = \frac{1}{\sqrt{n_1^2 + n_2^2}}\begin{pmatrix}
n_1 &-n_2\\
n_2 &n_1
\end{pmatrix}
$$
whose columns form an orthonormal basis of eigenvectors for $\overline{\Omega}$. It holds 
$$
\overline{\Omega} = O^t \Delta O,\qquad \qquad 
\Delta = \begin{pmatrix}
n_3^2 &0\\
0 &1
\end{pmatrix}. 
$$
Let us set
$$
\overline{\Omega}^{1/2} := O^t \Delta^{1/2} O,\qquad \qquad 
\Delta^{1/2} = \begin{pmatrix}
n_3 &0\\
0 &1
\end{pmatrix}. 
$$
The matrix $\overline{\Omega}^{1/2}$ is a symmetric square root of $\overline{\Omega}$. Let us consider now the standard Gaussian vector $(Z_1, Z_2)$ such that
\begin{equation}\label{defZ}
\left (\widetilde \partial_1^\Sigma F, \widetilde \partial_2^\Sigma  F \right )^t = \overline{\Omega}^{1/2} (Z_1, Z_2)^t 
\end{equation}
then 
\begin{equation}\label{Z_F}
 | \nabla_{\Sigma} F|^2 = \left (\partial_1^\Sigma F, \partial_2^\Sigma  F \right ) \overline{\Omega}^{-1} \left (\partial_1^\Sigma F, \partial_2^\Sigma  F \right )^t = M | (Z_1, Z_2)|^2. 
\end{equation}
For each $\sigma\in\Sigma$ we just wrote $| \nabla_{\Sigma} F(\sigma)|$ as the norm (up to a factor) of a bivariate \emph{standard} Gaussian vector $(Z_1(\sigma), Z_2(\sigma))$. 

\subsubsection{Chaotic components}

Recall that 
\begin{equation*}
    M = \frac{4\pi^2m}{3}.
\end{equation*}
Substituting (\ref{Z_F}) into (\ref{approx}) we can write for every 
$\varepsilon>0$
\begin{equation}\label{newapprox}
\mathcal L^\varepsilon = \sqrt M \int_{\Sigma} \delta^\varepsilon_0(F(\sigma)) | (Z_1(\sigma), Z_2(\sigma))|\, d\sigma;
\end{equation}
in particular $F(\sigma), Z_1(\sigma), Z_2(\sigma)$ are \emph{independent} random variables for each $\sigma\in \Sigma$. 
In view of (\ref{newapprox}) we can apply previous results in \cite[\S 3.2.2]{MPRW} to derive the chaotic expansion for $\mathcal L^\varepsilon$, and then immediately establish the following result bearing in mind Lemma \ref{lem1} while letting $\varepsilon$ tend to zero. Let us first introduce some more notation: for $k,n,l$ non-negative integers 
\begin{equation}\label{beta}
    \beta_{2k} := \frac{1}{\sqrt{2\pi}} H_{2k}(0),
\end{equation}
$H_{2k}$ denoting the $2k$-th Hermite polynomial (\ref{defH}) and
\begin{equation}\label{alpha}
    \alpha_{2n,2l} := \sqrt{\frac{\pi}{2}} \frac{(2n)!(2l)!}{n! l!} \frac{1}{2^{n+l}} p_{n+l}\left ( \frac{1}{4} \right ),
\end{equation}
where for $i\in \mathbb N$ and $t\in \mathbb R$
\begin{equation*}
    p_{i}(t) =  \sum_{j=0}^i (-1)^j (-1)^i { i \choose j} \frac{(2j+1)!}{(j!)^2} t^j.
\end{equation*}
\begin{lemma}\label{lem2} For $m\in S$ we have 
\begin{equation}\label{chaos_exp}
\mathcal L = \sum_{q=0}^{+\infty} \mathcal L[2q],
\end{equation}
where the convergence of the series is in $L^2(\mathbb P)$, and for $q\ge 0$
\begin{equation}\label{chaos_comp}
\mathcal L[2q] = \sqrt M \sum_{k+n+l=q} \frac{\beta_{2k} \alpha_{2n, 2l}}{(2k)! (2n)! (2l)!}\int_{\Sigma} H_{2k}(F(\sigma)) H_{2n}(Z_1(\sigma))H_{2l}(Z_2(\sigma))\, d\sigma,
\end{equation}
the vector $Z$ being defined as in (\ref{defZ}).
\end{lemma} 
\begin{remark}\label{remchaos}\rm
We will show that $\mathcal L[2]$ (resp. $\mathcal L[4]$) can be reduced to the integral over the surface of a bivariate polynomial of degree two (resp. of degree four) evaluated at $F$ and $|(Z_1, Z_2)|^2$ (hence at $F$ and $| \nabla_\Sigma F|^2$ thanks to (\ref{Z_F})), see the proof of Lemma \ref{lem_2} (resp. the proof of Lemma \ref{4dec}). This simplifies computations a lot and is of some independent interest. We do believe it holds true for \emph{each} chaotic component, namely for every $q$ the random variable $\mathcal L[2q]$ can be written as the integral over $\Sigma$ of a bivariate polynomial of degree $2q$ evaluated at $F$ and $| \nabla_\Sigma F|^2$. We leave it as a topic for future research. 
\end{remark}

\section{Proofs of the main results}\label{sec_proofs_main}

Let us bear in mind the chaotic expansion for $\mathcal L$ in Lemma \ref{lem2}. 
The projection onto the zeroth Wiener chaos, i.e., onto $\mathbb R$ is the mean of the random variable, cf. (\ref{mean}).
\begin{lemma}
For every $m\in S$,
$$
\mathcal L[0] =  \pi \sqrt{\frac{m}{3}} \cdot A.
$$
\end{lemma}
\begin{proof} From (\ref{chaos_comp}), (\ref{beta}) and (\ref{alpha})
$$
\mathcal L[0]=\sqrt M \beta_{0} \alpha_{0,0} | \Sigma| = \sqrt{\frac{4\pi^2 m}{3}} \frac{1}{\sqrt{2\pi}} \sqrt{\frac{\pi}{2}} | \Sigma| = \pi \sqrt{\frac{m}{3}} \cdot A,
$$
thus concluding the proof. 
\end{proof}

\subsection{Proof of Theorem \ref{th1}}

We first need to study the variance of the second chaotic component $\mathcal L[2]$ in (\ref{chaos_exp}). It turns out that it is asymptotic to the variance of the nodal intersection length $\mathcal L$ (Proposition \ref{prop_var2}), hence the orthogonality of Wiener chaoses and a Central Limit Theorem for $\mathcal L[2]$ (Proposition \ref{prop_law2}) allow to deduce the same for $\mathcal L$ thus establishing Theorem \ref{th1}. 
\begin{proposition}\label{prop_var2}
Let $\Sigma\subset \mathbb T^3$ be a surface as in Definition \ref{def_Sigma}, with nowhere zero Gauss-Kronecker curvature. Assume moreover that $3\mathcal I_{2, \Sigma} > A^2$. Then, as $m\to +\infty$ s.t. $m\not\equiv 0,4,7 \pmod 8$,
$$
\text{\rm Var}(\mathcal L[2]) \sim \text{\rm Var}(\mathcal L).
$$
\end{proposition}
Proposition \ref{prop_var2} will be proven in \S \ref{sec_prop2}. Let us now study the asymptotic distribution of the second chaotic component.
\begin{proposition}\label{prop_law2}
Let $\Sigma\subset \mathbb T^3$ be a surface as in Definition \ref{def_Sigma}, with nowhere zero Gauss-Kronecker curvature. Assume moreover that $3\mathcal I_{2, \Sigma} > A^2$. Then, as $m\to +\infty$ s.t. $m\not\equiv 0,4,7 \pmod 8$,
$$
\frac{\mathcal L[2]}{\sqrt{\text{\rm Var}(\mathcal L[2])}} \mathop{\to}^d Z,
$$
where $Z\sim \mathcal N(0,1)$. 
\end{proposition}
We will prove Proposition \ref{prop_law2} in \S \ref{sec_prop2}. We are now in a position to establish our first main result.

\begin{proof}[Proof of Theorem \ref{th1} assuming Proposition \ref{prop_var2} and Proposition \ref{prop_law2}]
From (\ref{chaos_exp}) we can write 
$$
\frac{\mathcal L - \mathbb E[\mathcal L]}{\sqrt{\text{\rm Var}(\mathcal L)}} = \sum_{q=1}^{+\infty} \frac{\mathcal L[2q]}{\sqrt{\text{\rm Var}(\mathcal L)}},
$$
where the series converges in $L^2(\mathbb P)$. Thanks to Proposition \ref{prop_var2} and the orthogonality of Wiener chaoses we have 
$$
\frac{\mathcal L - \mathbb E[\mathcal L]}{\sqrt{\text{\rm Var}(\mathcal L)}} = \frac{\mathcal L[2]}{\sqrt{\text{\rm Var}(\mathcal L[2])}} + o_{\mathbb P}(1),
$$
where $o_{\mathbb P}(1)$ denotes a sequence of random variables converging to zero in probability. 
Proposition \ref{prop_law2} then allows to conclude the proof of Theorem \ref{th1}.
\end{proof}

\subsection{Proof of Theorem \ref{th2}} 

For a static surface $\Sigma$, equation \eqref{vargen1} only gives that the true asymptotic behaviour of the variance of the nodal intersection length is of lower order of magnitude than $m/N$. In \S \ref{sec_KR} we will prove Theorem \ref{prop1} establishing the fine asymptotic law of $\text{\rm Var}(\mathcal L)$ for well separated sequences of eigenvalues through Kac-Rice formula. In this circumstances it turns out that the second chaotic projection $\mathcal L[2]$ is negligible, and evaluating the fourth chaotic component yields (Proposition \ref{prop_var4}) that its variance  is asymptotic to $\text{\rm Var}(\mathcal L)$. As for Theorem \ref{th2}, the orthogonality of Wiener chaoses and a (non-Central) Limit Theorem for $\mathcal L[4]$ (Proposition \ref{prop_law4}) allow to deduce the same for $\mathcal L$ thus establishing Theorem \ref{th2}. 
\begin{proposition}\label{prop_var4}
Let $\Sigma\subset \mathbb T^3$ be a static surface as in Definition \ref{def_static}. Then, as $m\to +\infty$ s.t. $m\not\equiv 0,4,7 \pmod 8$ along well separated sequences of eigenvalues, we have 
$$
\text{\rm Var}(\mathcal L[4]) \sim \text{\rm Var}(\mathcal L)
$$
where the asymptotics for $\text{\rm Var}(\mathcal L)$ are established in Theorem \ref{prop1}.
\end{proposition}
Proposition \ref{prop_var4} will be proven in \S \ref{sec_4chaos}. Let us now investigate the asymptotic distribution of the fourth chaotic projection.
\begin{proposition}\label{prop_law4}
Let $\Sigma\subset \mathbb T^3$ be a static surface as in Definition \ref{def_static}. Then, as $m\to +\infty$ s.t. $m\not\equiv 0,4,7 \pmod 8$ along well separated sequences of eigenvalues,
$$
\frac{\mathcal L[4]}{\sqrt{\text{\rm Var}(\mathcal L[4])}} \mathop{\to}^d \mathcal M,
$$
where $\mathcal M$ is defined in (\ref{M}). 
\end{proposition}
We will prove Proposition \ref{prop_law4} in \S \ref{sec_4chaos}. We can now establish our second main result.

\begin{proof}[Proof of Theorem \ref{th2} assuming Theorem \ref{prop1}, Proposition \ref{prop_var2} and Proposition \ref{prop_law2}]
From (\ref{chaos_exp}) we can write 
$$
\frac{\mathcal L - \mathbb E[\mathcal L]}{\sqrt{\text{\rm Var}(\mathcal L)}} = \sum_{q=1}^{+\infty} \frac{\mathcal L[2q]}{\sqrt{\text{\rm Var}(\mathcal L)}},
$$
where the series converges in $L^2(\mathbb P)$. Thanks to Proposition \ref{prop_var2} and the orthogonality of Wiener chaoses we have 
$$
\frac{\mathcal L - \mathbb E[\mathcal L]}{\sqrt{\text{\rm Var}(\mathcal L)}} = \frac{\mathcal L[4]}{\sqrt{\text{\rm Var}(\mathcal L[4])}} + o_{\mathbb P}(1),
$$
where $o_{\mathbb P}(1)$ denotes a sequence of random variables converging to zero in probability. 
Proposition \ref{prop_law4} then allows to conclude the proof of Theorem \ref{th2}.
\end{proof}

In order to prove Corollary \ref{corD}, we will need the following: let $\Sigma\subset \mathbb T^3$ be the two-dimensional unit sphere, then as $m\to +\infty$ s.t. $m\not\equiv 0,4,7 \pmod 8$ along well separated sequences of eigenvalues,
\begin{equation}\label{eqD}
    \Corr(\mathcal L[4], \mathcal A[4]) \longrightarrow 1,
\end{equation}
where $\mathcal A[4]=\mathcal A_m[4]$ denotes the fourth chaotic component of the area $\mathcal A=\mathcal A_m$ of the nodal set $F_m^{-1}(0)$ \cite[(4.14)]{Cam19}. Actually,  the dominant terms in $\mathcal A[4]$ and $\mathcal L[4]$ \emph{coincide}, up to a factor depending on $m$, see \S \ref{sec_4chaos} for details.

\begin{proof}[Proof of Corollary \ref{corD} assuming Proposition \ref{prop_var4} and (\ref{eqD})]
Thanks to the orthogonality of Wiener chaoses, 
\begin{eqnarray*}
\Corr(\mathcal L, \mathcal A) = \frac{\Cov(\mathcal L[4], \mathcal A[4])}{\sqrt{\Var(\mathcal L) \Var(\mathcal A)}}.
\end{eqnarray*}
The result then follows from (\ref{eqD}), bearing in mind Proposition \ref{prop_var4} and both \cite[Theorem 1.2]{benmaf}, \cite[Theorem 1]{Cam19} ensuring that
\begin{equation*}
    \Var(\mathcal L) \sim \Var(\mathcal L[4]),\qquad 
    \Var(\mathcal A) \sim \Var(\mathcal A[4])
\end{equation*}
respectively. 
\end{proof}

\section{Second chaotic component: analytic formulas}

\subsection{Proofs of Proposition \ref{prop_var2} and Proposition \ref{prop_law2}}\label{sec_prop2}

From (\ref{chaos_comp}) for $q=1$, recalling that the coefficients $\alpha_{n,m}$ in (\ref{alpha}) are symmetric, the second chaotic component of the nodal intersection length can be written as  
\begin{equation}\label{exp_2}
\mathcal L[2] = \sqrt M \left ( \frac{\beta_{2}\alpha_{0,0}}{2!} \int_{\Sigma} H_2(F(\sigma))\,d\sigma + \frac{\beta_{0}\alpha_{2,0}}{2!} \int_{\Sigma} \left ( H_2(Z_1(\sigma))+H_2(Z_2(\sigma))\right )\,d\sigma     \right ),
\end{equation}
where as before $M=4\pi^2m/3$. Let us introduce some more notation: 
$$\mathcal E^+=\mathcal E_m^+:=\lbrace \mu \in \mathcal E : \mu_1 >0\rbrace$$
if $m$ is not a sum of two squares, otherwise if $m$ is not a square
$\mathcal E^+=\mathcal E_m^+:=\lbrace \mu \in \mathcal E : \mu_1 >0\rbrace \cup \lbrace \mu\in \mathcal E : \mu_1=0, \mu_2 >0 \rbrace$ else $\mathcal E^+=\mathcal E_m^+:=\lbrace \mu \in \mathcal E : \mu_1 >0\rbrace \cup \lbrace \mu\in \mathcal E : \mu_1=0, \mu_2 >0 \rbrace \cup \lbrace (0,0,\sqrt m)\rbrace$. 
\begin{remark}\label{rem1}\rm
The random variables $a_\mu, \mu\in \mathcal E^+$ in Definition \ref{def_arw} are independent.
\end{remark}
A careful investigation of the r.h.s. of (\ref{exp_2}) leads to the following key result whose proof is given in Appendix \ref{sec_lem2}. 
\begin{lemma}\label{lem_2}
For every $m\in S$ we have the following decomposition
\begin{equation}\label{2dec}
\mathcal L[2] = \mathcal L^a[2] + \mathcal L^b[2],
\end{equation}
where the first term on the r.h.s. of (\ref{2dec}) is
\begin{eqnarray}\label{2a}
   \mathcal L^a[2]=\mathcal L_m^a[2] :=  \frac{\sqrt M}{8 \sqrt{\frac{N}{2} }}\cdot   \frac{1}{\sqrt{\frac{N}{2} }} \sum_{\mu\in \mathcal E^+} (|a_\mu |^2  -1)   \left ( |\Sigma | 
-   3\int_\Sigma \left \langle \frac{\mu}{|\mu|}, n(\sigma)\right \rangle^2 \, d\sigma  \right ),
\end{eqnarray}
and the second one is 
\begin{eqnarray}\label{2b}
\mathcal L^b[2] = \mathcal L^b_m[2]
     &:=& \frac{\sqrt M}{8N} \sum_{\mu\ne \mu'} a_\mu \overline{a_{\mu'}}  \int_{\Sigma}e_{\mu}(\sigma) e_{-\mu'}(\sigma)\cr
     &&\times  \left ( -2 + 3  \left \langle \frac{\mu}{|\mu|}, \frac{\mu'}{|\mu'|} \right \rangle - 3 \left \langle \frac{\mu}{|\mu|}, n(\sigma) \right \rangle \left \langle \frac{\mu'}{|\mu'|}, n(\sigma) \right \rangle \right )d\sigma.
\end{eqnarray}
\end{lemma}
\begin{proof}[Proof of Proposition \ref{prop_var2} assuming Lemma \ref{lem_2}]
From (\ref{2dec}) we have 
\begin{equation}\label{ovvio}
    \Var (\mathcal L[2]) =  \Var (\mathcal L^a[2]) + \Var (\mathcal L^b[2]) + 2 \Cov (\mathcal L^a[2], \mathcal L^b[2]).
\end{equation}
First we compute the variance of $\mathcal L^a[2]$. Recalling Remark \ref{rem1} and the fact that $\Var(|a_\mu |^2  -1)=1$  for every $\mu\in \mathcal E$,
we can write 
\begin{eqnarray}\label{hola6}
\Var \left (  \mathcal L^a[2] \right ) &=&  \frac{M}{64 \cdot \frac{N}{2} }   \frac{1}{\frac{N}{2} } \sum_{\mu\in \mathcal E^+} \Var(|a_\mu |^2  -1) \left ( |\Sigma | 
-   3\int_\Sigma \left \langle \frac{\mu}{|\mu|}, n(\sigma)\right \rangle^2 \, d\sigma  \right )^2\cr
&=& \frac{M}{64 \cdot \frac{N}{2} }   \frac{1}{N } \sum_{\mu\in \mathcal E} \left ( |\Sigma |^2 -   6|\Sigma |\int_\Sigma \left \langle \frac{\mu}{|\mu|}, n(\sigma)\right \rangle^2 \, d\sigma  +9\mathcal H(\mathcal E) \right )\cr
&=& \frac{M}{64\cdot \frac{N}{2}} \left (- A^2 + 9 \mathcal H(\mathcal E) \right ),
\end{eqnarray}
where $\mathcal H(\mathcal E)$ is defined as in (\ref{H}) and we used the fact that 
\begin{equation*}
    \frac{1}{N } \sum_{\mu\in \mathcal E} \int_\Sigma \left \langle \frac{\mu}{|\mu|}, n(\sigma)\right \rangle^2 \, d\sigma = \frac13 |\Sigma|^2.
\end{equation*}
From \cite[Lemma 5.3]{Maf18} applied to the r.h.s. of (\ref{hola6}) we have that, as $m\to +\infty$ s.t. $m\not\equiv 0,4,7 \pmod 8$,
\begin{eqnarray}\label{var2a}
    \text{\rm Var}(\mathcal L^a[2]) &=& \frac{4\pi^2m/3}{64 \cdot \frac{N}{2} }  \left ( -|\Sigma |^2  + \frac{9}{15}(|\Sigma |^2 + 2 \mathcal I) + O\left (m^{-1/28 +o(1)} \right ) \right )\cr
    &=& \frac{\pi^2 m}{60 N} \left (3\mathcal I - |\Sigma |^2 + O\left (m^{-1/28 +o(1)} \right ) \right ),
\end{eqnarray}
hence from (\ref{vargen1}) the variance of $\mathcal L^a[2]$ is asymptotic to the variance of nodal intersection length, i.e., as $m\to +\infty$ s.t. $m\not\equiv 0,4,7 \pmod 8$
\begin{equation}\label{bip1}
    \text{\rm Var}(\mathcal L) \sim \text{\rm Var}(\mathcal L^a[2]).
\end{equation}
Let us now investigate the variance of $\mathcal L^b[2]$ in (\ref{2b}). We have 
\begin{eqnarray*}
\text{\rm Var}(\mathcal L^b[2]) &=& \frac{M}{64 N^2}\sum_{\mu\ne \mu'}\sum_{\mu''\ne \mu'''} \mathbb E[a_\mu \overline{a_{\mu'}}a_{\mu''} \overline{a_{\mu'''}}  ]\int_{\Sigma^2}e_{\mu}(\sigma) e_{-\mu'}(\sigma)e_{\mu''}(\sigma') e_{-\mu'''}(\sigma')\cr
&& \times   \left ( -2 + 3  \left \langle \frac{\mu}{|\mu|}, \frac{\mu'}{|\mu'|} \right \rangle - 3 \left \langle \frac{\mu}{|\mu|}, n(\sigma) \right \rangle \left \langle \frac{\mu'}{|\mu'|}, n(\sigma) \right \rangle \right ) \cr
&& \times   \left ( -2 + 3  \left \langle \frac{\mu''}{|\mu''|}, \frac{\mu'''}{|\mu'''|} \right \rangle - 3 \left \langle \frac{\mu''}{|\mu|''}, n(\sigma') \right \rangle \left \langle \frac{\mu'''}{|\mu'''|}, n(\sigma') \right \rangle \right )d\sigma d\sigma' \cr
&\le& C \frac{m}{N^2}\sum_{\mu\ne \mu'} \left | \int_{\Sigma}e_{\mu}(\sigma) e_{-\mu'}(\sigma)\right. \cr
&& \left. \times \left ( -2 + 3  \left \langle \frac{\mu}{|\mu|}, \frac{\mu'}{|\mu'|} \right \rangle - 3 \left \langle \frac{\mu}{|\mu|}, n(\sigma) \right \rangle \left \langle \frac{\mu'}{|\mu'|}, n(\sigma) \right \rangle \right )d\sigma \right |^2,
\end{eqnarray*}
for some absolute constant $C>0$. 
Thanks to Proposition 5.4 in \cite{Maf18} we can bound the r.h.s. of the previous inequality and get  
\begin{equation}\label{aa}
    \text{\rm Var}(\mathcal L^b[2]) \le C \frac{m}{N^2} \sum_{\mu\ne \mu'} \frac{1}{|\mu-\mu'|^2}.
\end{equation}
Applying \cite[Theorem 1.1]{bosaru} we obtain 
\begin{equation}\label{aaa}
    \sum_{\mu\ne \mu'} \frac{1}{|\mu-\mu'|^2}\ll \frac{N^2}{(\sqrt m)^{2-o(1)}},
\end{equation}
and substituting (\ref{aaa}) into (\ref{aa}) (also recalling (\ref{Nbehavior})) we deduce that
as $m\to +\infty$ s.t. $m\not\equiv 0,4,7 \pmod 8$,
\begin{equation}\label{var2b}
    \text{\rm Var}(\mathcal L^b[2]) = o\left ( \text{\rm Var}(\mathcal L^a[2]) \right ).
\end{equation}
By Cauchy-Schwartz inequality, thanks to (\ref{var2a}) and (\ref{var2b}) we have 
$$
 \text{\rm Cov}(\mathcal L^a[2], \mathcal L^b[2]) = o\left ( \text{\rm Var}(\mathcal L^a[2]) \right )
$$
that together with (\ref{ovvio}), (\ref{var2a}) and (\ref{var2b}) gives, as $m\to +\infty$ s.t. $m\not\equiv 0,4,7 \pmod 8$,
\begin{equation}\label{bip2}
     \text{\rm Var}(\mathcal L[2]) \sim \text{\rm Var}(\mathcal L^a[2]).
\end{equation}
Finally (\ref{bip1}) and (\ref{bip2}) 
allow to conclude the proof. 
\end{proof}

\begin{proof}[Proof of Proposition \ref{prop_law2} assuming Lemma \ref{lem_2}] From Lemma \ref{lem_2} and the proof of Proposition \ref{prop_var2} we immediately deduce that as $m\to +\infty$ s.t. $m\not\equiv 0,4,7 \pmod 8$
\begin{equation}\label{cip1}
    \frac{\mathcal L[2]}{\sqrt{\text{\rm Var}(\mathcal L[2]}} =  \frac{\mathcal L^a[2]}{\sqrt{\text{\rm Var}(\mathcal L^a[2])}} + o_{\mathbb P}(1),
\end{equation}
hence it suffices to study the asymptotic distribution of $\mathcal L^a[2]$. We have 
\begin{eqnarray}
&&\frac{\mathcal L^a[2]}{\sqrt{\text{\rm Var}(\mathcal L^a[2])}} =    \frac{1}{\sqrt{\frac{N}{2} }} \sum_{\mu\in \mathcal E^+} (|a_\mu |^2  -1) \frac{|\Sigma | 
-   3\int_\Sigma \left \langle \frac{\mu}{|\mu|}, n(\sigma)\right \rangle^2 \, d\sigma   }{  \sqrt{9\mathcal H(\mathcal E) - A^2}}.
\end{eqnarray}
Now we can apply Lindeberg's criterion (see e.g. \cite[Remark 11.1.2]{nopebook}): since 
$$
\max_{\mu\in \mathcal E^+}  \frac{1}{\sqrt{\frac{N}{2}}}\left | \frac{|\Sigma | 
-   3\int_\Sigma \left \langle \frac{\mu}{|\mu|}, n(\sigma)\right \rangle^2 \, d\sigma   }{  \sqrt{9\mathcal H(\mathcal E) - A^2}}\right | \to 0 
$$
we have that as $m\to +\infty$ s.t. $m\not\equiv 0,4,7 \pmod 8$
$$
\frac{\mathcal L^a[2]}{\sqrt{\text{\rm Var}(\mathcal L^a[2])}}\to Z,
$$
where $Z\sim \mathcal N(0,1)$ that together with (\ref{cip1}) allows to conclude the proof. 
\end{proof}

\section{Proof of the Kac-Rice approximate formula Proposition \ref{KR}}
\label{sec_KR}
In \cite{Maf18}, it is shown that the conditions of Kac-Rice for the second moment of the nodal intersection length $\mathcal{L}$ are verified \cite[Section 3]{Maf18}. Therefore,
\begin{equation}
\label{snsprepre}
\mathbb{E}[\mathcal{L}^2]
=M\iint_{\Sigma^2}K_{2;\Sigma}(\sigma,\sigma')d\sigma d\sigma',
\end{equation}
where $K_{2;\Sigma}$ is the (scaled) 2-point function
\begin{equation*}
K_{2;\Sigma}=\frac{1}{2\pi\sqrt{1-r^2}}\mathbb{E}[|(\hat{W_1},\hat{W_2})|\cdot|(\hat{W_3},\hat{W_4})|],
\end{equation*}
\begin{equation}
\label{hattheta}\hat{W}=(\hat{W_1},\hat{W_2},\hat{W_3},\hat{W_4})\sim\mathcal{N}(0,\hat{\Theta}),
\qquad\qquad
\hat{\Theta}=I_4+
\begin{pmatrix}
X & Y \\ Y' & X'
\end{pmatrix},
\end{equation}
with $X,Y,X',Y'$ appropriate $2\times 2$ matrices. These are defined as follows. Recall that $\Omega$ is the covariance matrix of $\nabla_\Sigma F$. We denote
\begin{equation}
	\label{matL}
	L:=\begin{pmatrix}
	1 & 0
	\\ 0 & 1
	\\ 0 & 0
	\end{pmatrix}
\end{equation}
and
\begin{equation}
\label{Q}
Q(\sigma)=
\frac{1}{n_3^2+n_3}\cdot
\begin{pmatrix}
n_1^2+n_3^2+n_3 & n_1n_2 \\ n_1n_2 & n_2^2+n_3^2+n_3
\end{pmatrix}=Q^T,
\end{equation}
a square root of $(L^T\Omega L)^{-1}$. Moreover, $Q':=Q(\sigma')$. Let $D$ be the row vector of partial derivatives of $r$,
\begin{equation*}
	D=D(\sigma,\sigma'):=
	\left(\frac{\partial r}{\partial \sigma_1},
	\frac{\partial r}{\partial \sigma_2},
	\frac{\partial r}{\partial \sigma_3}\right)
	=\frac{2\pi i}{N}
	\sum_{\mu\in\mathcal{E}} e^{2\pi i\langle\mu,\sigma-\sigma'\rangle}\cdot\mu,
	\end{equation*}
and $H$ the Hessian matrix,
\begin{equation*}
	H=H(\sigma,\sigma'):=H_{r}
	=-\frac{4\pi^2}{N}
		\sum_{\mu\in\mathcal{E}} e^{2\pi i\langle\mu,\sigma-\sigma'\rangle}\cdot\mu^T\mu.
	\end{equation*}
Then we set
\begin{align*}
& X:=X(\sigma,\sigma')=-
\frac{1}{(1-r^2)M}QL^T\Omega D^TD\Omega LQ,
\\
& X':=X(\sigma',\sigma),
\\
& Y:=Y(\sigma,\sigma')=-\frac{1}{M}\left[QL^T\Omega\left(H+\frac{r}{1-r^2}D^TD\right)\Omega'LQ'\right],
\\
& Y':=Y(\sigma',\sigma).
\end{align*}

In \cite[Lemma 4.5]{Maf18} we found an asymptotic expression for the scaled two-point function, via a Taylor expansion of a perturbed $4\times 4$ standard Gaussian matrix
\begin{equation*}
I_4+
\begin{pmatrix}
X & Y \\ Y' & X'
\end{pmatrix},
\end{equation*}
to order two; in the next lemma, we Taylor expand it to order four. The method employed was used first by Berry \cite{berry2}. The case where $X'=X$ and $Y'=Y$ was treated in \cite[Lemma 5.1]{krkuwi} ($4\times 4$ Gaussian matrix) and \cite[Lemma 5.8]{benmaf} ($6\times 6$).
\begin{lemma}
\label{lemma5.1}
Suppose that
\begin{equation*} \hat{W}=(\hat{W_1},\hat{W_2},\hat{W_3},\hat{W_4})\sim\mathcal{N}(0,\hat{\Theta}),
\end{equation*}
where
\begin{equation*}
\hat{\Theta}=I_4+
\begin{pmatrix}
X & Y \\ Y' & X'
\end{pmatrix}
\end{equation*}
is positive definite with real entries, and the $2\times 2$ blocks $X,X',Y,Y'$ are symmetric. Assume further that $rk(X)=rk(X')=1$. Then one has
\begin{multline*}
\mathbb{E}[\|(\hat{W_1},\hat{W_2})\|\cdot\|(\hat{W_3},\hat{W_4})\|]
=\frac{\pi}{2}+\frac{\pi}{2}\cdot\frac{1}{8}\left[2\tr(X)+2\tr(X')+\tr(Y'Y)\right]
\\+\frac{\pi}{2}\cdot\frac{1}{512}\left[32\tr(X)\tr(X')-16\tr(XYY')-16\tr(X'Y'Y)-24(\tr X)^2-24(\tr X')^2
\right.
\\
\left.
+2\tr(Y'YY'Y)+(\tr Y'Y)^2-8\tr(X)\tr(Y'Y)-8\tr(X')\tr(Y'Y)
\right]
\\+O(X^3+X'^3+Y^6+Y'^6).
\end{multline*}
\end{lemma}
The proof of Lemma \ref{lemma5.1} is omitted here, as the ideas and computations involved are very similar to those of \cite[Lemma 5.1]{krkuwi} and \cite[Lemma 5.8]{benmaf}.

For $\sigma,\sigma'$ such that $r(\sigma,\sigma')$ is bounded away from $\pm 1$, we have $\frac{1}{\sqrt{1-r^2}}=1+\frac{1}{2}r^2+O(r^4)$, so that we establish an asymptotic for the (scaled) two point function (a more precise version of \cite[Proposition 3.7]{Maf18})
\begin{multline}
\label{prop4.5}
K_{2;\Sigma}(\sigma,\sigma')\sim\frac{1}{4}
\left[
1+\frac{1}{2}r^2+\frac{tr(X)}{4}+\frac{tr(X')}{4}+\frac{tr(Y'Y)}{8}
\right]
\\
+\frac{3}{16}r^4
+\frac{1}{2^{10}}
\left[32\tr(X)\tr(X')-16\tr(XYY')-16\tr(X'Y'Y)-24(\tr X)^2-24(\tr X')^2
\right.
\\
\left.
+2\tr(Y'YY'Y)+(\tr Y'Y)^2-8\tr(X)\tr(Y'Y)-8\tr(X')\tr(Y'Y)
\right]
\\
+\frac{1}{32}
\left[2r^2\tr(X)+2r^2\tr(X')+r^2\tr(Y'Y)\right],
\end{multline}
up to an error
\begin{equation*}
\mathfrak{E}:=r^6+X^3+X'^3+Y^6+Y'^6+r^4(tr(X)+tr(X')+tr(Y'Y))+r^2(tr(X^2)+tr(X'^2)+tr(Y'Y)^2).
\end{equation*}

The rest of the proof of Proposition \ref{KR} follows as in \cite[Section 3.4]{Maf18}: the two point function asymptotic \eqref{prop4.5} holds `almost everywhere' in the sense that $r$ is small outside of a small subset of $\Sigma^2$, called the \textbf{singular set} $S$ \cite[Definition 3.10]{Maf18}. The RHS of \eqref{snsprepre} is separated into integrals over $S$ and its complement. The measure of the singular set satisfies the bound \cite[Lemma 3.11]{Maf18}
\begin{equation}
\label{measS}
\text{meas}(S)\ll\mathcal{R}_6(m).
\end{equation}
Together with \cite[Lemma 3.12]{Maf18}, this implies in particular
\begin{equation*}
\iint_{S}K_{2;\Sigma}(\sigma,\sigma')d\sigma d\sigma'\ll\mathcal{R}_6(m).
\end{equation*}
In the integral over $\Sigma^2\setminus S$, after applying \eqref{prop4.5}, we note that the entries of $X,X',Y,Y'$ are uniformly bounded with respect to $\sigma,\sigma'$ \cite[Lemma 3.5]{Maf18}, so that changing the domain of integration back to $\Sigma^2$ carries an error of $\text{meas}(S)$.

The approximate Kac-Rice formula Proposition \ref{KR} is thus established, leaving the following lemma to be proven in Section \ref{secarith}.
\begin{lemma}
\label{err}
We have the bound
\begin{equation*}
\iint_{\Sigma^2}\mathfrak{E}(\sigma,\sigma')d\sigma d\sigma'
=o\left(\frac{1}{N^2}\right)+O(\mathfrak{S}_6).
\end{equation*}
\end{lemma}

\section{The arithmetic part: proof of Proposition \ref{arith}}
\label{secarith}
We begin with a Lemma, estimating the second and fourth moments of the covariance function \eqref{cov} and of its derivatives. It is a more precise version of \cite[Lemmas 5.1, 5.2, 5.5]{Maf18}.
\allowdisplaybreaks[0]
\begin{lemma}
\label{5.1}
Let $\Sigma$ be a regular compact surface  with nowhere vanishing Gauss-Kronecker curvature.
Then we have
\begin{align}
\label{R2}
&\iint_{\Sigma^2}r^2d\sigma d\sigma'
=\frac{A^2}{N}+O(\mathfrak{S}_2),
\\
\label{R4}
&\iint_{\Sigma^2}r^4d\sigma d\sigma'
=\frac{3A^2}{N^2}+o\left(\frac{1}{N^2}\right)+O(\mathfrak{S}_4),
\end{align}
and the following estimates, up to an error $o(1/N^2)+O(\mathfrak{S}_2)+O(\mathfrak{S}_4)+O(\mathfrak{S}_6)$:
\begin{align}
\label{trX}
&\iint_{\Sigma^2}tr(X)d\sigma d\sigma'
\sim-\frac{2A^2}{N}-\frac{2A^2}{N^2};
\\
\label{trX'}
&\iint_{\Sigma^2}tr(X')d\sigma d\sigma'
\sim-\frac{2A^2}{N}-\frac{2A^2}{N^2};
\\
\label{trY'Y}
&\iint_{\Sigma^2}tr(Y'Y)d\sigma d\sigma'
\sim\frac{3}{N}\left(A^2+3\mathcal{H}\right)+\frac{1}{N^2}(-2A^2-2\mathcal{I});
\\
\label{trX2}
&\iint_{\Sigma^2}tr(X^2)d\sigma d\sigma'
\sim\frac{8A^2}{N^2};
\\
\label{trX'2}
&\iint_{\Sigma^2}tr(X'^2)d\sigma d\sigma'
\sim\frac{8A^2}{N^2};
\\
\label{trXtrX'}
&\iint_{\Sigma^2}tr(X)tr(X')d\sigma d\sigma'
\sim\frac{1}{N^2}(6A^2+2\mathcal{I});
\\
\label{r^2trX}
&\iint_{\Sigma^2}r^2tr(X)d\sigma d\sigma'
\sim-\frac{2A^2}{N^2};
\\
\label{r^2trX'}
&\iint_{\Sigma^2}r^2tr(X')d\sigma d\sigma'
\sim-\frac{2A^2}{N^2};
\\
\label{r^2trY'Y}
&\iint_{\Sigma^2}r^2tr(Y'Y)d\sigma d\sigma'
\sim\frac{1}{N^2}\left(\frac{28}{5}A^2+\frac{16}{5}\mathcal{I}\right);
\\
\label{trYY'X}
&\iint_{\Sigma^2}tr(YY'X)d\sigma d\sigma'
\sim\frac{1}{N^2}\frac{3}{5}\left(-6A^2-2\mathcal{I}\right);
\\
\label{trY'YX'}
&\iint_{\Sigma^2}tr(Y'YX')d\sigma d\sigma'
\sim\frac{1}{N^2}\frac{3}{5}\left(-6A^2-2\mathcal{I}\right);
\\
\label{trXtrY'Y}
&\iint_{\Sigma^2}tr(X)tr(Y'Y)d\sigma d\sigma'
\sim\frac{1}{N^2}\frac{3}{5}\left(-12A^2-4\mathcal{I}\right);
\\
\label{trX'trY'Y}
&\iint_{\Sigma^2}tr(X')tr(Y'Y)d\sigma d\sigma'
\sim\frac{1}{N^2}\frac{3}{5}\left(-12A^2-4\mathcal{I}\right);
\\
\label{trY'YY'Y}
&\iint_{\Sigma^2}tr(Y'YY'Y)d\sigma d\sigma'
\sim\frac{1}{N^2}\frac{9}{25}\left(52 A^2+24 \mathcal{I}+12 \mathcal{I}_4\right);
\\
\label{trY'Y^2}
&\iint_{\Sigma^2}tr(Y'Y)^2d\sigma d\sigma'
\sim\frac{1}{N^2}\frac{9}{25}\left(60 A^2+40 \mathcal{I}+12 \mathcal{I}_4\right).
\end{align}
\end{lemma}
Before proving Lemma \ref{5.1}, we complete the proof of Proposition \ref{arith}.

\begin{proof}[Proof of Proposition \ref{arith}]
We insert the various estimates of Lemma \ref{5.1} into the approximate Kac-Rice Proposition \ref{KR}, recalling that for static surfaces we have $\mathcal{H}=A^2/9$.
\end{proof}

\begin{proof}[Proof of Lemma \ref{5.1}]
We will prove in detail a few of the estimates, the remaining being similar. In \cite[Proof of Lemma 5.1]{Maf18}, it was shown that
\begin{equation*}
\iint_{\Sigma^2}r^2d\sigma d\sigma'
=\frac{A^2}{N}+O\left(\frac{1}{N^2}\sum_{\mu\neq\mu'}\frac{1}{|\mu-\mu'|^2}\right)
\end{equation*}
which is \eqref{R2}.

To show \eqref{R4}, we proceed as in \cite[Section 6]{Maf18} and write
\begin{equation*}
\mathcal{R}_4(m)=\frac{|\mathcal{C}_4|}{N^4}+\mathfrak{S}_4.
\end{equation*}
We distinguish between the non-degenerate correlations $\mathcal{X}(4)$ and the complement, i.e. those that cancel out in pairs $\mathcal{D}'(4)$. For the former, we have  $|\mathcal{D}'(4)|=3N^2+O(N)$.
Thanks to \cite[Theorem 1.6]{benmaf}, we can bound
\begin{equation*}
|\mathcal{X}(4)|=O(N^{7/4}).
\end{equation*}
Collecting the estimates yields \eqref{R4}.

To show \eqref{trX}, we need to refine the analysis of \cite[Section 5.2]{Maf18}. We separate the domain of integration into the singular set $S$ of \cite[Definition 3.10]{Maf18} and its complement:
\begin{equation}
\label{s+ns}
\iint_{\Sigma^2}tr(X)d\sigma d\sigma'=\iint_{\Sigma^2\setminus S}tr(X)d\sigma d\sigma'+O(\text{meas}(S)),
\end{equation}
where we used the uniform boundedness of $X$ given by \cite[Lemma 3.5]{Maf18}. On $\Sigma^2\setminus S$ the covariance function $r$ is bounded away from $\pm 1$, so that
\begin{equation}
\label{trXa}
trX
=-\frac{1}{M}D\Omega D^T
-\frac{r^2}{M}D\Omega D^T
+O\left(\frac{r^4}{M}D\Omega D^T\right)
\end{equation}
where we also remarked that
\begin{equation*}
tr(QL^T\Omega D^TD\Omega LQ)=D\Omega D^T.
\end{equation*}
We then rewrite \eqref{s+ns} as
\begin{eqnarray}
\label{trXesti}
\iint_{\Sigma^2}tr(X)d\sigma d\sigma' &=& -\frac{1}{M}\iint_{\Sigma^2}D\Omega D^Td\sigma d\sigma'-\frac{1}{M}\iint_{\Sigma^2}r^2D\Omega D^Td\sigma d\sigma'\cr
&&+O\left(\frac{1}{M}\iint_{\Sigma^2}r^4D\Omega D^Td\sigma d\sigma'\right).
\end{eqnarray}
Due to \cite[Proof of Lemma 5.5]{Maf18}, the first summand on the RHS of \eqref{trXesti} gives a contribution of
\begin{equation*}
-\frac{2A^2}{N}
+O(\mathfrak{S}_2).
\end{equation*}
Let us now analyse the second summand in \eqref{trXesti}. We write
\begin{equation*}
r^2D\Omega D^T=\frac{1}{N^2}\left(-\frac{4\pi^2}{N^2}\right)\sum_{\mathcal{E}^4}e^{2\pi i\langle\sum\mu_i,\sigma\rangle}\mu_3\Omega\mu_4^T.
\end{equation*}
The terms corresponding to lattice points that cancel in pairs $\mathcal{D}'(4)$ survive only when $\mu_1=-\mu_2$ and $\mu_3=-\mu_4$ by \cite[Lemma 2.3]{brahpo}. These terms give a contribution of
\begin{equation*}
\iint_{\Sigma^2}-\frac{1}{M}\frac{-1}{N^2}\frac{4\pi^2}{N^2}(-N)\sum_{\mathcal{E}}\left[(1-n_1^2)(\mu^{(1)})^2+(1-n_2^2)(\mu^{(2)})^2+(1-n_3^2)(\mu^{(3)})^2\right]d\sigma d\sigma'=-\frac{2A^2}{N^2}.
\end{equation*}
We also carry the error terms
\begin{equation*}
\frac{1}{N^4}\sum_{\mathcal{X}(4)}\left|\int_{\Sigma}e^{2\pi i\langle\sum\mu_i,\sigma\rangle}d\sigma\right|^2=\frac{|\mathcal{X}(4)|}{N^4}\ll N^{-9/4+o(1)}
\end{equation*}
due to \eqref{fourcorr}, and, via \cite[section 7]{stein1} (cf. \cite[Proposition 5.4]{Maf18})
\begin{equation*}
\frac{1}{N^4}\sum_{\mathcal{E}^4\setminus\mathcal{C}(4)}\iint_{\Sigma^2}e^{2\pi i\langle\sum\mu_j,\sigma-\sigma'\rangle}d\sigma d\sigma'
\\\ll\frac{1}{N^4}\sum_{\mathcal{E}^4\setminus\mathcal{C}(4)}\frac{1}{|\mu_1+\mu_2+\mu_3+\mu_4|^2}=\mathfrak{S}_4.
\end{equation*}
Likewise, for the third summand in \eqref{trXesti}, we bound
\begin{equation*}
\frac{1}{N^6}\sum_{\mathcal{E}^6}\iint_{\Sigma^2}e^{2\pi i\langle\sum\mu_j,\sigma-\sigma'\rangle}d\sigma d\sigma'
\ll\frac{|\mathcal{C}(6)|}{N^6}+\mathfrak{S}_6.
\end{equation*}
According to \eqref{sixcorr}, the length six correlations are no more than $|C(6)|\ll N^{11/3+\epsilon}$. Consolidating all the estimates, we complete the proof of \eqref{trX}, \eqref{trX'} being similar. The proofs of the remaining estimates use the same ideas and are omitted here.
\end{proof}

\begin{proof}[Proof of Lemma \ref{err}]
For the first summand of $\mathfrak{E}$, we write
\begin{equation*}
\mathcal{R}_6(m)=\frac{|\mathcal{C}_6|}{N^6}+\mathfrak{S}_6
\end{equation*}
and use the bound \eqref{sixcorr} for $6$-correlations.

The rest of the proof uses the ideas of \cite[Proof of Lemma 3.9]{Maf18}. We will show the computations for the second summand of $\mathfrak{E}$, the remaining being analogous. We have
\begin{equation}
\label{trX3}
\iint_{\Sigma^2}tr(X^3)d\sigma d\sigma'=O\left(\iint_{\Sigma^2}\frac{1}{M^3}(D\Omega D^T)^3d\sigma d\sigma'\right)+O(\text{meas}(S))
\end{equation}
due to \eqref{trXa}. The terms coming from $6$-correlations carry a contribution of $|\mathcal{C}_6|/N^6$, while the off-diagonal are bounded by
\begin{equation*}
\frac{1}{N^4}\sum_{\mathcal{E}^6\setminus\mathcal{C}(6)}\left|\int_{\Sigma}
\left(\frac{\mu_5}{|\mu_5|}\Omega\frac{\mu_6^T}{|\mu_6|}(\sigma)\right)e^{2\pi i\langle\sum\mu_j,\sigma\rangle}d\sigma\right|^3
\ll_\Sigma\frac{1}{N^6}\sum_{\mathcal{E}^6\setminus\mathcal{C}(6)}\frac{1}{|\sum\mu_j|^2}=\mathfrak{S}_6.
\end{equation*}
Finally, we have the bound \eqref{measS} for $\text{meas}(S)$. All these error terms are admissible: the proof of Lemma \ref{err} is complete.
\end{proof}

We end this section by comparing upper bounds for the sums $\mathfrak{S}_4$ and $\mathfrak{S}_6$. By \cite[Lemma 4.7]{Maf18},
\begin{equation*}
\mathfrak{S}_4=O\left(\frac{N^{5/8+\epsilon}}{N^2}\right).
\end{equation*}

\begin{lemma}
\label{6mom6corr}
We have the bound
\begin{equation}
\label{eqn6corr}
\sum_{\mathcal{E}^6\setminus\mathcal{C}(6)}\frac{1}{|\mu_1+\mu_2+\mu_3+\mu_4+\mu_5+\mu_6|^2}\ll N^{4+4/9+\epsilon}.
\end{equation}
\end{lemma}
\begin{proof}
we separate the summation in \eqref{eqn6corr} over the ranges: $1\leq|v|\leq A$ and $|v|\geq A$, where $A=A(m)$ is a parameter.

\underline{First range: $1\leq|v|\leq A$.} Let us fix $v$ and count the number of solutions $(\mu_1,\dots,\mu_6)$ to
\begin{equation*}
v=\mu_1+\mu_2+\mu_3+\mu_4+\mu_5+\mu_6.
\end{equation*}
These are counted by the number of incidences between a set of points $\mathcal{P}$ and a set of spheres $\mathcal{S}$. The points are $\mathcal{P}=\mathcal{E}+\mathcal{E}$, $|\mathcal{P}|=N^2$, and the spheres are centred at the points $\{\mathcal{E}+\mathcal{E}+\mathcal{E}-v\}$ and have radius $\sqrt{m}$, $|\mathcal{S}|=N^3$. It is shown in \cite[Theorem 3.3]{benmaf} that these incidences number at most $\ll N^{11/3+\epsilon}$. Therefore,
\begin{equation}
\label{regi1}
\sum_{1\leq|v|\leq A}\frac{1}{|\mu_1+\mu_2+\mu_3+\mu_4+\mu_5+\mu_6|^2}\ll N^{11/3+\epsilon}\sum_{1\leq|v|\leq A}\frac{1}{|v|^2}\ll N^{11/3+\epsilon}A.
\end{equation}

\underline{Second range: $|v|\geq A$.} Here we have
\begin{equation}
\label{regi3}
\sum_{|v|\geq A}\frac{1}{|\mu_1+\mu_2+\mu_3+\mu_4+\mu_5+\mu_6|^2}\ll \frac{N^6}{A^2}.
\end{equation}
Collecting the estimates \eqref{regi1} and \eqref{regi3}, we obtain
\begin{equation*}
\sum_{\mathcal{E}^6\setminus\mathcal{C}(6)}\frac{1}{|\mu_1+\mu_2+\mu_3+\mu_4+\mu_5+\mu_6|^2}\ll N^{11/3+\epsilon} A+\frac{N^6}{A^2}.
\end{equation*}
The optimal choice for the parameter is $A=N^{7/9}$, so that one has the estimate \eqref{eqn6corr}.

\end{proof}

Therefore, we obtain the bound
\begin{equation*}
\mathfrak{S}_6=O\left(\frac{N^{4/9+\epsilon}}{N^2}\right).
\end{equation*}
We note that applying the incidence geometry argument from \cite{benmaf} yields a power improvement over the bound of
$\mathfrak{S}_6=O(N^{5/8+\epsilon}/N^2)$, given by the method of \cite[Lemma 4.7]{Maf18}.

\section{Fourth chaotic component: analytic formulas}\label{sec_4chaos}

\subsection{Proof of Proposition \ref{prop_var4}}

From (\ref{chaos_exp}) for $q=2$, recalling that the coefficients $\alpha_{n,m}$ in (\ref{alpha}) are symmetric, the fourth chaotic component of the nodal intersection length is 
\begin{eqnarray}\label{4chaos}
\mathcal L[4] &=& \sqrt M \left ( \frac{\beta_4\alpha_{0,0}}{4!}\int_{\Sigma} H_4(F(\sigma))\,d\sigma +\frac{\beta_2\alpha_{2,0}}{2! 2!}\int_{\Sigma} H_2(F(\sigma)) (H_2(Z_1(\sigma)) + H_2(Z_2(\sigma)))\,d\sigma \right.\cr
&&\left.+ \frac{\beta_0\alpha_{2,2}}{2! 2!} \int_{\Sigma} H_2(Z_1(
\sigma)) H_2(Z_2(
\sigma))\,d\sigma + \frac{\beta_0 \alpha_{4,0}}{4!}\int_{\Sigma} (H_4(Z_1(\sigma))+ H_4(Z_2(\sigma)))\,d\sigma \right ).
\end{eqnarray}
The following decomposition is a key result and will be proved in Appendix \ref{app4chaos}.
\begin{lemma}\label{4dec}
For every $m\in S$ we have the following 
\begin{equation}\label{4ab}
    \mathcal L[4] = \mathcal L^a[4] + \mathcal L^b[4],
\end{equation}
where 
\begin{eqnarray}\label{4a}
    \mathcal L^a[4] = \mathcal L_m^a[4] &:=&  \sqrt{\frac{4\pi^2m}{3}}   \frac{3}{16 \cdot 8} \frac{1}{N^2}  
\sum_{\mu,\mu'}  (|a_\mu |^2 - 1) (|a_{\mu'} |^2 - 1) \times 
\cr 
&&  \times \Bigg (\int_\Sigma \Big ( -3 -9\left \langle \frac{\mu}{|\mu|}, n(\sigma)\right \rangle^2 \left \langle \frac{\mu'}{|\mu'|}, n(\sigma)\right \rangle^2 +14\left \langle \frac{\mu}{|\mu|}, n(\sigma)\right \rangle^2\cr &&-6\left \langle \frac{\mu}{|\mu|}, \frac{\mu'}{|\mu'|}\right \rangle^2
+12\left \langle \frac{\mu}{|\mu|}, \frac{\mu'}{|\mu'|}\right \rangle \left \langle \frac{\mu}{|\mu|}, n\right \rangle \left \langle \frac{\mu'}{|\mu'|},n\right \rangle
\Big )d\sigma  \Bigg ) \cr
&&-\sqrt{\frac{4\pi^2m}{3}}   \frac{3}{16 \cdot 8} \frac{1}{N^2}  
\sum_{\mu}  |a_\mu |^4   \times\cr
&&\times \int_\Sigma \bigg ( -9 -9\left \langle \frac{\mu}{|\mu|}, n(\sigma)\right \rangle^4 +26\left \langle \frac{\mu}{|\mu|}, n(\sigma)\right \rangle^2
\bigg )d\sigma,
\end{eqnarray}
and $\mathcal L^b[4]=\mathcal L^b_m[4]$ is such that
\begin{equation}\label{4b}
  \Var(\mathcal L^b[4]) \ll m \cdot \max \left \lbrace \frac{|\mathcal X(4)|}{N^4}, \mathfrak{S}_4, \mathfrak{S}_2 \right \rbrace,
\end{equation}
where $\mathfrak{S}_\ell$ and $\mathcal X(\ell)$ are defined as in (\ref{frakS}) and (\ref{defxl}) respectively.
\end{lemma}

\begin{proof}[Proof of Proposition \ref{prop_var4} assuming Lemma \ref{4dec}]
Let us start computing the variance of the first summand on the r.h.s. of (\ref{4ab}): first we write 
\begin{eqnarray}
\mathcal L^a[4] = \mathcal L^{a,1}[4] - \mathcal L^{a,2}[4],
\end{eqnarray}
where 
\begin{eqnarray*}
\mathcal L^{a,1}[4] &:=& \sqrt{\frac{4\pi^2m}{3}}   \frac{3}{16 \cdot 8} \frac{4}{N^2}  
\sum_{\mu,\mu'\in \mathcal E^+}  (|a_\mu |^2 - 1) (|a_{\mu'} |^2 - 1) \times 
\cr 
&&  \times \Bigg (\int_\Sigma \Big ( -3 -9\left \langle \frac{\mu}{|\mu|}, n(\sigma)\right \rangle^2 \left \langle \frac{\mu'}{|\mu'|}, n(\sigma)\right \rangle^2 +14\left \langle \frac{\mu}{|\mu|}, n(\sigma)\right \rangle^2\cr 
&&-6\left \langle \frac{\mu}{|\mu|}, \frac{\mu'}{|\mu'|}\right \rangle^2
+12\left \langle \frac{\mu}{|\mu|}, \frac{\mu'}{|\mu'|}\right \rangle \left \langle \frac{\mu}{|\mu|}, n\right \rangle \left \langle \frac{\mu'}{|\mu'|},n\right \rangle
\Big )d\sigma  \Bigg )
\end{eqnarray*}
and
\begin{eqnarray}
\mathcal L^{a,2}[4] &:=& \sqrt{\frac{4\pi^2m}{3}}   \frac{3}{16 \cdot 8} \frac{1}{N^2}  
\sum_{\mu}  |a_\mu |^4   \times\cr
&&\times \int_\Sigma \bigg ( -9 -9\left \langle \frac{\mu}{|\mu|}, n(\sigma)\right \rangle^4 +26\left \langle \frac{\mu}{|\mu|}, n(\sigma)\right \rangle^2
\bigg )d\sigma.
\end{eqnarray}
We write 
\begin{eqnarray*}
\Var(\mathcal L^{a,1}[4])  &=&
\frac{4\pi^2m}{3}   \frac{3^2}{16^2 \cdot 8^2} \frac{4^2}{N^4}  
\sum_{\substack{\mu,\mu'\in \mathcal E^+\\\mu'',\mu'''\in \mathcal E^+}} \Cov( (|a_\mu |^2 - 1) (|a_{\mu'} |^2 - 1), (|a_{\mu''} |^2 - 1) (|a_{\mu'''} |^2 - 1)) \times 
\cr 
&&  \times \Bigg (\int_\Sigma \Big ( -3 -9\left \langle \frac{\mu}{|\mu|}, n(\sigma)\right \rangle^2 \left \langle \frac{\mu'}{|\mu'|}, n(\sigma)\right \rangle^2 +14\left \langle \frac{\mu}{|\mu|}, n(\sigma)\right \rangle^2\cr
&&-6\left \langle \frac{\mu}{|\mu|}, \frac{\mu'}{|\mu'|}\right \rangle^2
+12\left \langle \frac{\mu}{|\mu|}, \frac{\mu'}{|\mu'|}\right \rangle \left \langle \frac{\mu}{|\mu|}, n(\sigma)\right \rangle \left \langle \frac{\mu'}{|\mu'|},n(\sigma)\right \rangle
\Big )d\sigma  \Bigg ) \cr
&&  \times \Bigg (\int_\Sigma \Big ( -3 -9\left \langle \frac{\mu''}{|\mu''|}, n(\sigma')\right \rangle^2 \left \langle \frac{\mu'''}{|\mu'''|}, n(\sigma')\right \rangle^2 +14\left \langle \frac{\mu''}{|\mu''|}, n(\sigma')\right \rangle^2\cr
&&-6\left \langle \frac{\mu''}{|\mu''|}, \frac{\mu'''}{|\mu'''|}\right \rangle^2
+12\left \langle \frac{\mu''}{|\mu''|}, \frac{\mu'''}{|\mu'''|}\right \rangle \left \langle \frac{\mu''}{|\mu''|}, n(\sigma')\right \rangle \left \langle \frac{\mu'''}{|\mu'''|},n(\sigma')\right \rangle
\Big )d\sigma'  \Bigg ).
\end{eqnarray*}
Plainly, $\Cov((|a_\mu |^2 - 1) (|a_{\mu'} |^2 - 1),(|a_{\mu''} |^2 - 1) (|a_{\mu'''} |^2 - 1))= \mathbb E[(|a_\mu |^2 - 1) (|a_{\mu'} |^2 - 1)(|a_{\mu''} |^2 - 1) (|a_{\mu'''} |^2 - 1)] - \mathbb E[(|a_\mu |^2 - 1) (|a_{\mu'} |^2 - 1)] \mathbb E[(|a_{\mu''} |^2 - 1) (|a_{\mu'''} |^2 - 1)]$, and $\mathbb E[(|a_\mu |^2 - 1) (|a_{\mu'} |^2 - 1)(|a_{\mu''} |^2 - 1) (|a_{\mu'''} |^2 - 1) ]\ne 0$ if and only if $\mu,\mu',\mu'',\mu'''$ are pairwise equal. Hence 
\begin{itemize}
\item for $\mu\ne \mu', \mu''\ne \mu'''$, $\mathbb E[(|a_\mu |^2 - 1) (|a_{\mu'} |^2 - 1)] \mathbb E[(|a_{\mu''} |^2 - 1) (|a_{\mu'''} |^2 - 1)]=0$ and $\mathbb E[(|a_\mu |^2 - 1) (|a_{\mu'} |^2 - 1)(|a_{\mu''} |^2 - 1) (|a_{\mu'''} |^2 - 1) ]\ne 0$ if and only if $\mu=\mu''$, $\mu'=\mu'''$ or $\mu=\mu'''$, $\mu'=\mu''$ (note that $\mathbb E[(|a_\mu |^2 - 1)^2 (|a_{\mu'} |^2 - 1)^2 ]= 1$ for $\mu\ne \mu'$); 
\item for $\mu=\mu'$, $\mu''\ne \mu'''$, $\mathbb E[(|a_\mu |^2 - 1)^2] \mathbb E[(|a_{\mu''} |^2 - 1) (|a_{\mu'''} |^2 - 1)]=0$ and $\mathbb E[(|a_\mu |^2 - 1)^2 (|a_{\mu''} |^2 - 1) (|a_{\mu'''} |^2 - 1) ]= 0$ (the same holds true for $\mu\ne \mu', \mu'' = \mu'''$);
\item for $\mu=\mu', \mu''=\mu'''$, $\Cov((|a_\mu|^2-1)^2, (|a_{\mu''}|^2-1)^2)=0$ if $\mu\ne \mu''$, else $\Cov((|a_\mu|^2-1)^2, (|a_{\mu}|^2-1)^2)=8$.
\end{itemize} 
This ensures that 
\begin{eqnarray}\label{var4a1}
\Var(\mathcal L^{a,1}[4]) 
&= &\frac{4\pi^2m}{3}   \frac{3^2}{16^2 \cdot 8^2} \frac{4^2}{N^2} \frac{1}{N^2} 
\sum_{\mu\ne\mu'\in \mathcal E^+}  \Bigg |\int_\Sigma \Big ( -3 -9\left \langle \frac{\mu}{|\mu|}, n(\sigma)\right \rangle^2 \left \langle \frac{\mu'}{|\mu'|}, n(\sigma)\right \rangle^2 \cr
&&+14\left \langle \frac{\mu}{|\mu|}, n(\sigma)\right \rangle^2
-6\left \langle \frac{\mu}{|\mu|}, \frac{\mu'}{|\mu'|}\right \rangle^2+12\left \langle \frac{\mu}{|\mu|}, \frac{\mu'}{|\mu'|}\right \rangle \left \langle \frac{\mu}{|\mu|}, n\right \rangle \left \langle \frac{\mu'}{|\mu'|},n\right \rangle
\Big )d\sigma  \Bigg |^2 \cr
&&+\frac{4\pi^2m}{3}   \frac{3^2}{16^2 \cdot 8^2} \frac{4^2}{N^2} \frac{1}{N^2}  
\sum_{\mu\ne\mu'\in \mathcal E^+}   \Bigg (\int_\Sigma \Big ( -3 -9\left \langle \frac{\mu}{|\mu|}, n(\sigma)\right \rangle^2 \left \langle \frac{\mu'}{|\mu'|}, n(\sigma)\right \rangle^2 \cr
&&+14\left \langle \frac{\mu}{|\mu|}, n(\sigma)\right \rangle^2
-6\left \langle \frac{\mu}{|\mu|}, \frac{\mu'}{|\mu'|}\right \rangle^2
+12\left \langle \frac{\mu}{|\mu|}, \frac{\mu'}{|\mu'|}\right \rangle \left \langle \frac{\mu}{|\mu|}, n(\sigma)\right \rangle \left \langle \frac{\mu'}{|\mu'|},n(\sigma)\right \rangle
\Big )d\sigma  \Bigg ) \cr
&&  \times \Bigg (\int_\Sigma \Big ( -3 -9\left \langle \frac{\mu}{|\mu|}, n(\sigma')\right \rangle^2 \left \langle \frac{\mu'}{|\mu'|}, n(\sigma')\right \rangle^2 +14\left \langle \frac{\mu'}{|\mu'|}, n(\sigma')\right \rangle^2\cr
&&-6\left \langle \frac{\mu}{|\mu|}, \frac{\mu'}{|\mu'|}\right \rangle^2
+12\left \langle \frac{\mu}{|\mu|}, \frac{\mu'}{|\mu'|}\right \rangle \left \langle \frac{\mu}{|\mu|}, n(\sigma')\right \rangle \left \langle \frac{\mu'}{|\mu'|},n(\sigma')\right \rangle
\Big )d\sigma'  \Bigg ) \cr
&&+ \frac{4\pi^2m}{3}   \frac{3^2}{16^2 \cdot 8^2} \frac{4^2\cdot 8}{N^2} \frac{1}{N^2} \sum_{\mu \in \mathcal E^+} \Bigg |\int_\Sigma \Big ( -9 -9\left \langle \frac{\mu}{|\mu|}, n(\sigma)\right \rangle^4 \cr
&&+26\left \langle \frac{\mu}{|\mu|}, n(\sigma)\right \rangle^2
\Big )d\sigma  \Bigg |^2
. 
\end{eqnarray}
For $\Sigma$ a static surface, since $\mathcal H=\frac{A^2}{9}$,  
(\ref{var4a1}) becomes  
\begin{eqnarray}\label{var4a2}
\Var(\mathcal L^{a,1}[4]) 
&= &\frac{4\pi^2m}{3}   \frac{3^2}{16^2 \cdot 8^2} \frac{4^2}{N^2} \cdot 2\frac{1}{N^2} 
\sum_{\mu\ne\mu'\in \mathcal E^+}  \Bigg |\int_\Sigma \Big ( -3 -9\left \langle \frac{\mu}{|\mu|}, n(\sigma)\right \rangle^2 \left \langle \frac{\mu'}{|\mu'|}, n(\sigma)\right \rangle^2 \cr
&&+14\left \langle \frac{\mu}{|\mu|}, n(\sigma)\right \rangle^2
-6\left \langle \frac{\mu}{|\mu|}, \frac{\mu'}{|\mu'|}\right \rangle^2+12\left \langle \frac{\mu}{|\mu|}, \frac{\mu'}{|\mu'|}\right \rangle \left \langle \frac{\mu}{|\mu|}, n\right \rangle \left \langle \frac{\mu'}{|\mu'|},n\right \rangle
\Big )d\sigma  \Bigg |^2 \cr
&&+ \frac{4\pi^2m}{3}   \frac{3^2}{16^2 \cdot 8^2} \frac{4^2\cdot 8}{N^2} \frac{1}{N^2} \sum_{\mu \in \mathcal E^+} \Bigg |\int_\Sigma \Big ( -9 -9\left \langle \frac{\mu}{|\mu|}, n(\sigma)\right \rangle^4 \cr
&&+26\left \langle \frac{\mu}{|\mu|}, n(\sigma)\right \rangle^2
\Big )d\sigma  \Bigg |^2\cr 
&=& \frac{4\pi^2m}{3}   \frac{3^2}{16^2 \cdot 8^2} \frac{4\cdot 2}{N^2} \bigg [ \frac{1}{N^2} 
\sum_{\mu,\mu'\in \mathcal E}  \Bigg |\int_\Sigma \Big ( -3 -9\left \langle \frac{\mu}{|\mu|}, n(\sigma)\right \rangle^2 \left \langle \frac{\mu'}{|\mu'|}, n(\sigma)\right \rangle^2 \cr
&&+14\left \langle \frac{\mu}{|\mu|}, n(\sigma)\right \rangle^2
-6\left \langle \frac{\mu}{|\mu|}, \frac{\mu'}{|\mu'|}\right \rangle^2+12\left \langle \frac{\mu}{|\mu|}, \frac{\mu'}{|\mu'|}\right \rangle \left \langle \frac{\mu}{|\mu|}, n\right \rangle \left \langle \frac{\mu'}{|\mu'|},n\right \rangle
\Big )d\sigma  \Bigg |^2 \cr
&&+ \frac{4\pi^2m}{3}   \frac{3^2}{16^2 \cdot 8^2} \frac{4^2\cdot 8}{N^2} \frac{1}{N^2} \sum_{\mu \in \mathcal E^+} \Bigg |\int_\Sigma \Big ( -9 -9\left \langle \frac{\mu}{|\mu|}, n(\sigma)\right \rangle^4 \cr
&&+26\left \langle \frac{\mu}{|\mu|}, n(\sigma)\right \rangle^2
\Big )d\sigma  \Bigg |^2  + O\left (\frac{1}{N} \right )\bigg ].
\end{eqnarray}
Still for static surfaces, some tedious computations gives, as $m\to+\infty$ s.t. $m\ne 0,4,7 \ (\text{\rm mod\ } 8)$, 
\begin{equation}\label{var4a1static}
\Var(\mathcal L^{a,1}[4]) 
= \frac{\pi^2}{2^7\cdot 3\cdot 5^2}\frac{m}{N^2}  \left (35 A^2 +  81 \mathcal I_4\right )+ o\left (\frac{m}{N^2} \right ).
\end{equation}
It is easy to check that 
\begin{eqnarray*}
\Var(\mathcal L^{a,2}[4]) = o\left ( \frac{m}{N^2}\right ) 
\end{eqnarray*}
which together with (\ref{var4a1static}) gives 
\begin{equation}
    \Var(\mathcal L^a[4]) = \frac{\pi^2}{2^7\cdot 3\cdot 5^2}\frac{m}{N^2}  \left (35 A^2 +  81 \mathcal I_4\right )+ o\left (\frac{m}{N^2} \right )
\end{equation}
  Assumption \ref{theassu} ensures that for a well-separated sequence of eigenvalues $\mathfrak{S}_2$ and $\mathfrak{S}_4$ are $o(N^{-2})$, thus the contribution of $\mathcal L^b[4]$ is negligible (recall (\ref{fourcorr})) and 
\begin{equation*}
    \Var(\mathcal L[4]) = \frac{\pi^2}{9600}\frac{m}{N^2}  \left (35 A^2 +  81 \mathcal I_4\right )+ o\left (\frac{m}{N^2} \right ).
\end{equation*}
Thanks to the orthogonality property of different order Wiener chaoses, bearing in mind Theorem \ref{prop1}, we can conclude the proof.
\end{proof}

\section{Limit theorem for static surfaces: proof of Proposition \ref{prop_law4}}

\subsection{Preliminaries}

%
Let us rewrite the ``diagonal'' terms of the fourth chaotic projection $\mathcal L^a[4]= \mathcal L^{a,1}[4] - \mathcal L^{a,2}[4]$ as 
\begin{eqnarray*}
\mathcal L^{a,1}[4] &=& \sqrt{\frac{4\pi^2m}{3}} \frac{3\cdot 2}{16\cdot 8\cdot N}\bigg [ -3A \bigg ( \frac{1}{\sqrt{N/2}} \sum_{\mu\in \mathcal E^+} (|a_\mu|^2-1) \bigg )^2 \cr
&&- 9 \int_{\Sigma} \bigg (\frac{1}{\sqrt{N/2}} \sum_{\mu\in \mathcal E^+} (|a_\mu|^2-1) \left \langle \frac{\mu}{|\mu|}, n(\sigma)\right \rangle^2 \bigg )^2\,d\sigma \cr
&& +14 \bigg ( \frac{1}{\sqrt{N/2}} \sum_{\mu\in \mathcal E^+} (|a_\mu|^2-1) \bigg ) \int_{\Sigma} \bigg (\frac{1}{\sqrt{N/2}} \sum_{\mu\in \mathcal E^+} (|a_\mu|^2-1) \left \langle \frac{\mu}{|\mu|}, n(\sigma)\right \rangle^2 \bigg )\,d\sigma\cr
&& -6A \frac{1}{N/2} \sum_{\mu,\mu'\in \mathcal E^+} (|a_\mu|^2-1)(|a_{\mu'}|^2-1) \left \langle \frac{\mu}{|\mu|}, \frac{\mu'}{|\mu'|}\right \rangle^2 \cr
&&+12 \int_{\Sigma} \frac{1}{N/2}\sum_{\mu,\mu'\in \mathcal E^+}(|a_\mu|^2-1)(|a_{\mu'}|^2-1) \left \langle \frac{\mu}{|\mu|}, \frac{\mu'}{|\mu'|}\right \rangle \left \langle \frac{\mu}{|\mu|}, n(\sigma)\right \rangle\left \langle  \frac{\mu'}{|\mu'|}, n(\sigma)\right \rangle\,d\sigma \bigg ]
\end{eqnarray*} 
and
\begin{eqnarray*}
\mathcal L^{a,2}[4] &=& \sqrt{\frac{4\pi^2m}{3}}   \frac{3}{16 \cdot 8} \frac{1}{N}  \cdot 
\frac{1}{N/2}\sum_{\mu\in \mathcal E^+}  |a_\mu |^4   \times\cr
&&\times \int_\Sigma \bigg ( -9 -9\left \langle \frac{\mu}{|\mu|}, n(\sigma)\right \rangle^4 +26\left \langle \frac{\mu}{|\mu|}, n(\sigma)\right \rangle^2
\bigg )d\sigma.
\end{eqnarray*}
In what follows we investigate the asymptotic distribution of $\mathcal L^{a}[4]$. First of all note that 
\begin{equation}\label{media2}
 \mathbb E[\mathcal L^{a,2}[4]] = \sqrt{\frac{4\pi^2m}{3}}   \frac{3}{16 \cdot 8} \frac{1}{N}  \cdot 
\frac{2}{N}\sum_{\mu\in \mathcal E}   \int_\Sigma \bigg ( -9 -9\left \langle \frac{\mu}{|\mu|}, n(\sigma)\right \rangle^4 +26\left \langle \frac{\mu}{|\mu|}, n(\sigma)\right \rangle^2
\bigg )d\sigma.
\end{equation}
Recalling Lemma 4.6 in \cite{Maf18} to determine the asymptotic behaviour of the right hand side of (\ref{media2}), 
the Law of Large Numbers immediately gives the following result.
\begin{lemma}\label{lem_LLN} As $m\to +\infty$ s.t. $m\ne 0,4,7 \ (\text{\rm mod\ } 8)$,
\begin{equation*}
   \frac{ \mathcal L^{a,2}[4] }{ \sqrt{\frac{4\pi^2m}{3}}   \frac{3\cdot 2}{16 \cdot 8} \frac{1}{N}} \mathop{\longrightarrow}^{\mathbb P} -\frac{32}{15}A.
\end{equation*}
\end{lemma} 
In order to study $\mathcal L^{a,1}[4]$ we need to introduce some more notation: for $i,j\in \lbrace 1,2,3 \rbrace$, we define the following sequence of r.v.'s indexed by $m\in S$
\begin{equation*}
   W_{ij}= W_{ij}(m):= \frac{1}{m} \frac{1}{\sqrt{N/2}} \sum_{\mu\in \mathcal E^+} (|a_\mu|^2-1) \mu_i\mu_j.
\end{equation*}
(Note that $W_{ij}=W_{ji}$ for every $i,j$.)
Then we can rewrite the ``diagonal'' terms of the fourth chaotic component in a more compact way as a homogeneous polynomial of degree two in six variables  evaluated at $W:=(W_{11}, W_{12}, W_{13}, W_{22}, W_{23}, W_{33})^t$:
\begin{eqnarray}\label{eq4static}
\mathcal L^{a,1}[4] &=& \sqrt{\frac{4\pi^2m}{3}} \frac{3\cdot 2}{16\cdot 8\cdot N}\bigg [ -3A \bigg ( \sum_{i} W_{ii} \bigg )^2 - 9 \int_{\Sigma} \bigg (\sum_{ij} W_{ij}n_in_j \bigg )^2\,d\sigma\cr
&& +14 \bigg ( \sum_{i} W_{ii} \bigg ) \int_{\Sigma} \bigg (\sum_{ij} W_{ij}n_in_j \bigg )\,d\sigma -6A \sum_{ij} W_{ij}^2\cr
&&+12 \int_{\Sigma} \bigg ( \sum_i \bigg (\sum_j W_{ij} n_j \bigg )^2 \bigg )\,d\sigma \bigg ].
\end{eqnarray} 
\begin{example}\label{ex4sfera}\rm 
For $\Sigma$ the two-dimensional unit sphere, 
substituting 
\begin{eqnarray*}
    \int n_in_j\, d\sigma &=& \frac13 \delta_{ij} \cr
    \int n_i n_j n_\ell n_k\, d\sigma &=& \begin{cases}
    \frac15\qquad \text{if } i=j=\ell=k,\\
    \frac{1}{15}\qquad \text{if } i,j,\ell,k \text{ are pairwise equal},\\
    0\qquad \text{otherwise}. 
    \end{cases}
\end{eqnarray*}
into (\ref{eq4static}) we get 
\begin{equation*}
\mathcal L^{a,1}[4] 
= \sqrt{\frac{4\pi^2m}{3}} \frac{3\cdot 2}{16\cdot 8\cdot N}\frac{16}{15}A\bigg [ \bigg ( \sum_{i} W_{ii} \bigg )^2   -3 \sum_{ij} W_{ij}^2  \bigg ].
\end{equation*} 
\end{example}
Let us compute the covariance matrix $\Sigma_W := \mathbb E[W W^t]$ of $W$: for $i,j,k,l\in \lbrace {1,2,3}\rbrace$, $i\le j$, $k\le l$, 
\begin{eqnarray}\label{covW}
\mathbb E[W_{ij} W_{kl}] &=& \frac{1}{m^2} \frac{1}{N/2} \sum_{\mu,\mu'\in \mathcal E^+} \mathbb E[(|a_{\mu}|^2-1)(|a_{\mu'}|^2-1)] \mu_i\mu_j \mu'_k \mu'_l\cr
&=& \frac{1}{m^2} \frac{1}{N/2} \sum_{\mu\in \mathcal E^+} \mathbb E[(|a_{\mu}|^2-1)^2] \mu_i\mu_j \mu_k \mu_l\cr
&=& \frac{1}{m^2} \frac{1}{N/2} \sum_{\mu\in \mathcal E^+} \mu_i\mu_j \mu_k \mu_l = \frac{1}{m^2} \frac{1}{N} \sum_{\mu\in \mathcal E} \mu_i\mu_j \mu_k \mu_l.
\end{eqnarray}
Now let us set
\begin{eqnarray}\label{psi}
\psi=\psi(m) := \frac{1}{N} \sum_{\mu\in \mathcal E} \mu_1^4,
\end{eqnarray} 
then for $i,j\in \lbrace 1,2,3\rbrace$ such that $i\ne j$ 
\begin{equation*}
  \varphi = \varphi(m):=  \frac{1}{N} \sum_{\mu\in \mathcal E} \mu_i^2 \mu_j^2 = \frac{m^2}{6} - \frac{\psi}{2}.
\end{equation*}
It is immediate that the covariance matrix of the random vector $W$ is
\begin{equation*}
   \Sigma_{W} = \begin{pmatrix}
   \psi/m^2 &0 &0 &\varphi/m^2 &0 &\varphi/m^2\\
   0 &\varphi/m^2 &0 &0 &0 &0\\
   0 &0 &\varphi/m^2 &0 &0 &0\\
   \varphi/m^2 &0 &0 &\psi/m^2 &0 &\varphi/m^2\\
   0 &0 &0 &0 &\varphi/m^2 &0\\
   \varphi/m^2 &0 &0 &\varphi/m^2 &0 &\psi/m^2
   \end{pmatrix}.
\end{equation*}
\begin{lemma}\label{clt}
As $m\to +\infty$ s.t. $m\ne 0,4,7 \ (\text{\rm mod\ } 8)$, the vector $W$ converges in distribution to a Gaussian vector $Z=(Z_{11}, Z_{12}, Z_{13}, Z_{22}, Z_{23}, Z_{33})^t$ whose covariance matrix is 
\begin{equation*}
 \Sigma_{Z} := \mathbb E[Z Z^t] = \begin{pmatrix}
   1/5 &0 &0 &1/15 &0 &1/15\\
   0 &1/15 &0 &0 &0 &0\\
   0 &0 &1/15 &0 &0 &0\\
   1/15 &0 &0 &1/5 &0 &1/15\\
   0 &0 &0 &0 &1/15 &0\\
    1/15 &0 &0 &1/15 &0 &1/5
   \end{pmatrix}.
   \end{equation*}
\end{lemma}
\begin{proof}
Since the components of $W$ live in a fixed Wiener chaos (the second), it suffices to apply \cite{PT05} together with Lemma 4.6 in \cite{Maf18}, where in particular the exact asymptotic of $\psi$ in (\ref{psi}) has been found. 
\end{proof}
\begin{remark}\label{rem_diag}\rm
It is then natural to diagonalize the matrix $\Sigma_{Z}$ through an orthogonal matrix $O$ so that $\Sigma_{Z} = ODO^t$, where
\begin{equation}\label{o}
O^t=\begin{pmatrix} 1/\sqrt{3} &0 &0 &1/\sqrt{3} &0 &1/\sqrt{3}\\ -1/\sqrt{2} &0 &0 &0 &0 &1/\sqrt{2}\\ -1/\sqrt{6} &0 &0 &\sqrt{2/3} &0 &-1/\sqrt{6}\\ 0 &0 &0 &0 &1 &0\\ 0 &0 &1 &0 &0 &0\\ 0 &1 &0 &0 &0 &0\end{pmatrix}
\end{equation}
and
\begin{equation*}
D=\begin{pmatrix}
1/3 &0 &0 &0 &0 &0\\
0 &2/15 &0 &0 &0 &0\\
0 &0 &2/15 &0 &0 &0\\
0 &0 &0 &1/15 &0 &0\\
0 &0 &0 &0 &1/15 &0\\
0 &0 &0 &0 &0 &1/15
\end{pmatrix}.
\end{equation*}
It immediately follows that $\Delta^{-1} O^t  Z$ is a six dimensional standard Gaussian vector, where  
\begin{equation}\label{delta}
\Delta = \begin{pmatrix}
1/\sqrt 3 &0 &0 &0 &0 &0\\
0 &\sqrt 2/\sqrt{15} &0 &0 &0 &0\\
0 &0 &\sqrt 2/\sqrt{15} &0 &0 &0\\
0 &0 &0 &1/\sqrt{15} &0 &0\\
0 &0 &0 &0 &1/\sqrt{15} &0\\
0 &0 &0 &0 &0 &1/\sqrt{15}
\end{pmatrix}
\end{equation}
is a square root of $D$. 
\end{remark}

\subsection{Proof of Proposition \ref{prop_law4}} 

Bearing in mind Remark \ref{rem_diag}, let $X=(X_{11}, X_{12}, X_{13}, X_{22}, X_{23}, X_{33})^t$ be a six dimensional standard Gaussian vector, then $Y:=O\Delta X$ and $Z$ share the same distribution: 
\begin{eqnarray}\label{X}
Y:=O \Delta X &=& \begin{pmatrix}
1/\sqrt 3 &-1/\sqrt 2 &-1/\sqrt 6 &0 &0 &0\\
0 &0 &0 &0 &0 &1\\
0 &0 &0 &0 &1 &0\\
1/\sqrt{3} &0 &\sqrt 2/\sqrt 3 &0 &0 &0\\
0 &0 &0 &1 &0 &0\\
1/\sqrt 3 &1/\sqrt 2 &-1/\sqrt 6 &0 &0 &0
\end{pmatrix} 
\begin{pmatrix}
\frac{1}{\sqrt{3}}X_{11}\\
\frac{\sqrt 2}{\sqrt{15}}X_{12}\\
\frac{\sqrt 2}{\sqrt{15}}X_{13}\\
\frac{1}{\sqrt{15}}X_{22}\\
\frac{1}{\sqrt{15}}X_{23}\\
\frac{1}{\sqrt{15}}X_{33}
\end{pmatrix}\nonumber \\
&=&\begin{pmatrix}
1/3 X_{11} - 1/\sqrt{15} X_{12} - 1/\sqrt{45} X_{13} \\
1/\sqrt{15}X_{33}\\
1/\sqrt{15} X_{23}\\
1/3X_{11} + 2/\sqrt{45} X_{13}\\
1/\sqrt{15} X_{22}\\
1/3 X_{11} + 1/\sqrt{15} X_{12} - 1/\sqrt{45} X_{13}
\end{pmatrix} \mathop{=}^d Z.
\end{eqnarray}
 Note that 
$
  \sum_i Z_{ii} \mathop{=}^{d}  \sum_i Y_{ii} = X_{11}.
$
Let us define, for $i>j$, $Z_{ij}:=Z_{ji}$ and $X_{ij}:=X_{ji}$.

\begin{proof}[Proof of Proposition \ref{prop_law4}]
Thanks to (\ref{eq4static}), Lemma \ref{lem_LLN} and Lemma \ref{clt}, as $m\to +\infty$ s.t. $m\ne 0,4,7 \ (\text{\rm mod\ } 8)$,
\begin{eqnarray}\label{limitRV}
\frac{\mathcal L^{a}[4]}{\sqrt{\frac{4\pi^2m}{3}} \frac{3\cdot 2}{16\cdot 8} \frac{1}{N}} &\mathop{\to}^{d}&   -3A \bigg ( \sum_{i} Z_{ii} \bigg )^2 - 9 \int_{\Sigma} \bigg (\sum_{ij} Z_{ij}n_in_j \bigg )^2\,d\sigma\cr
&& + 14 \bigg ( \sum_{i} Z_{ii} \bigg )\int_{\Sigma} \bigg (\sum_{ij} Z_{ij}n_in_j \bigg )\,d\sigma -6A \sum_{ij} Z_{ij}^2\cr
&&+12 \int_{\Sigma} \bigg ( \sum_i \bigg (\sum_j Z_{ij} n_j \bigg )^2 \bigg )\,d\sigma + \frac{32}{15}A.
\end{eqnarray}
From (\ref{X}), the limiting r.v. on the right hand side of (\ref{limitRV}) is equal in distribution to the following polynomial of degree two in the variables $X_{ij}, i\le j$
\begin{equation}\label{defMtilde}
 \widetilde{\mathcal M} := \sum_{i\le j} c_{ij} (X_{ij}^2-1) + \sum_{\substack{i\le j, l\le k\\ (i,j)\ne (l,k)}} c_{ijlk} X_{ij}X_{lk},
\end{equation}
where 
\begin{eqnarray*}
&&c_{11}=0,\quad 
c_{12}=-\frac{1}{5}\int(3( n_1^2-n_3^2)^2+4n_2^2)d\sigma,\quad c_{13}=-\int(\frac{11}{15}-2n_2^2+\frac{9}{5}n_2^4)d\sigma,\cr
&&c_{22}=-\frac{4}{5}\int (n_1^2+3n_2^2n_3^2)d\sigma,\quad 
c_{23}=-\frac{4}{5}\int (n_2^2+3n_1^2n_3^2)d\sigma,\quad c_{33}=-\frac{4}{5}\int (n_3^2+3n_1^2n_2^2)d\sigma,\cr
&&c_{1112}=-\frac{16\sqrt{15}}{15}\int( n_1^2-n_3^2)d\sigma,\quad 
c_{1113}=-\frac{16\sqrt{5}}{15}\int(1-3n_2^2)d\sigma,\quad c_{1122}=\frac{32\sqrt{15}}{15}\int n_2n_3d\sigma,\cr
&&c_{1123}=\frac{32\sqrt{15}}{5}\int n_1n_3d\sigma, \quad 
c_{1133}=\frac{32\sqrt{15}}{5}\int n_1n_2d\sigma,\quad
c_{1213}=\frac{2\sqrt{3}}{15}\int (n_1^2-n_3^2)(1+9n_2^2)d\sigma,\cr
&&c_{1222}=\frac{4}{5}\int n_2n_3(2+3(n_1^2-n_3^2))d\sigma,\quad 
c_{1223}=\frac{12}{5}\int n_1n_3(n_1^2-n_3^2)d\sigma,\cr &&c_{1233}=\frac{4}{5}\int n_1n_2(-2+3(n_1^2-n_3^2))d\sigma,\quad c_{1322}=\frac{4\sqrt{3}}{15}\int n_2n_3(5-9n_2^2)d\sigma, \cr
&&c_{1323}=\frac{4\sqrt{3}}{15}\int n_1n_3(-1-9n_2^2)d\sigma, \quad 
c_{1333}=\frac{4\sqrt{3}}{15}\int n_1n_2(5-9n_2^2)d\sigma,\cr
&&c_{2223}=\frac{8}{5}\int n_1n_2(1-3n_3^2)d\sigma, \quad c_{2233}=\frac{8}{5}\int n_1n_3(1-3n_2^2)d\sigma,\quad c_{2333}=\frac{8}{5}\int n_2n_3(1-3n_1^2)d\sigma.
\end{eqnarray*}
Note that $\mathbb E[\widetilde{\mathcal M}]=0$, $\sum_{i\le j} c_{ij} = -32/15A$, the random variables $X_{ij}X_{lk}$ for $i\le j, l\le k$ are independent and
\begin{equation}\label{varMtilde}
    \Var(\widetilde{\mathcal M}) = 2 \sum_{i\le j} c_{ij}^2 + \sum_{\substack{i\le j, l\le k\\ (i,j)\ne (l,k)}} c_{ijlk}^2 = \frac{8}{225} (81\mathcal I_4 + 35A^2),
\end{equation}
where for the last equality we used the staticity of the surface and the fact that $n(\sigma)$ is a unit vector for each $\sigma\in \Sigma$. 
Since
\begin{equation*}
     \frac{4\pi^2m}{3}\cdot \frac{3^2\cdot 2^2}{16^2\cdot 8^2}\cdot \frac{1}{N^2}\cdot \Var(\widetilde{\mathcal M}) = \frac{\pi^2}{9600}\cdot \frac{m}{N^2}(81\mathcal I_4 + 35A^2) \sim \Var(\mathcal L[4])
\end{equation*}
we have 
\begin{equation}\label{defM}
  \frac{\mathcal L[4]}{\sqrt{\Var(\mathcal L[4])}} \mathop{\to}^d \frac{\widetilde{\mathcal M}}{\sqrt{\Var(\widetilde {\mathcal M})}} \mathop{=}^d \mathcal M,
\end{equation} 
where $\mathcal M$ is as in (\ref{M}). This concludes the proof of Proposition \ref{prop_law4}.
\end{proof}

\subsection{Proof of Equation \ref{eqD}}

Let us first recall the formula for the fourth chaotic component of the nodal area: as $m\to +\infty$ s.t. $m\ne 0,4,7 \ (\text{\rm mod\ } 8)$, \emph{in our notation}, \cite[(5.3)]{Cam19} reads as  
\begin{equation}\label{4caosarea}
    \mathcal A[4] = \frac{\sqrt m}{5\sqrt 3 N} \cdot 2\bigg [2+ \left (\sum_i W_{ii} \right )^2 - 3 \sum_{i,j} W_{i,j}^2 + o_{L^2(\mathbb P)}(1) \bigg ]
\end{equation}
where $o_{L^2(\mathbb P)}(1)$ denotes a sequence of random variables converging to zero in $L^2(\mathbb P)$. 

\begin{proof}[Proof of (\ref{eqD})]
From Lemma \ref{4dec}, Lemma \ref{lem_LLN} and Example \ref{ex4sfera} we have 
\begin{equation}
    \mathcal L[4] = \sqrt{\frac{4\pi^2m}{3}} \frac{3\cdot 2}{16\cdot 8\cdot N}\frac{16}{15}A\bigg [ 2+\bigg ( \sum_{i} W_{ii} \bigg )^2   -3 \sum_{ij} W_{ij}^2 + o_{L^2(\mathbb P)}(1) \bigg ],
\end{equation}
that together with (\ref{4caosarea}) immediately concludes the proof.
\end{proof}

Note that the dominant terms in $\mathcal A[4]$ and $\mathcal L[4]$ coincide, up to a factor depending on $m$.

\appendix 

\section{Proofs of technical lemmas}

\subsection{Proof of Lemma \ref{lem_2}}\label{sec_lem2}

\begin{proof}
 We start investigating the first summand on the r.h.s. of (\ref{exp_2}), we will isolate the contribution of the diagonal from that of off-diagonal terms: 
\begin{eqnarray}\label{ciao1}
\int_{\Sigma} H_2(F(\sigma))\,d\sigma &=& \int_{\Sigma} \left ( \frac{1}{N} \sum_{\mu, \mu'} a_\mu \overline{a_{\mu'}} \e_{\mu}(\sigma) \e_{-\mu'}(\sigma) - 1 \right )\, d\sigma \cr
&=&  |\Sigma | \frac{1}{N} \sum_{\mu} (|a_\mu |^2-1) + \frac{1}{N} \sum_{\mu\ne \mu'} a_\mu \overline{a_{\mu'}}\int_{\Sigma}   \e_{\mu}(\sigma) \e_{-\mu'}(\sigma)\, d\sigma.
\end{eqnarray}
For the other summand we have 
\begin{eqnarray}\label{term}
 \int_{\Sigma} \left ( H_2(Z_1(\sigma)) + H_2(Z_2(\sigma)) \right ) \,d\sigma &=& \int_{\Sigma} \left ( Z_1(\sigma)^2 + Z_2(\sigma)^2 - 2 \right ) \,d\sigma\cr
 &=& \int_{\Sigma} \left ( | (Z_1(\sigma), Z_2(\sigma))|^2 - 2 \right ) \,d\sigma
\end{eqnarray}
recalling that $H_2(t) = t^2 -1$. Since $M|(Z_1, Z_2)|^2 = |\nabla_\Sigma F|^2$, the random variable $\mathcal L[2]$ can be written as the integral of an explicit bivariate polynomial of degree two evaluated at $F$ and $|\nabla_\Sigma F|^2$ as anticipated in Remark \ref{remchaos}. 
From (\ref{Z_F}) we can write
\begin{equation}\label{sviluppo}
| (Z_1(\sigma), Z_2(\sigma))|^2=\frac{1}{n_3^2} \left ( (n_1^2 + n_3^2) (\widetilde \partial_1^\Sigma F)^2 + (n_2^2 + n_3^2)(\widetilde \partial_2^\Sigma F)^2  + 2n_1n_2 \widetilde \partial_1^\Sigma F \cdot \widetilde \partial_2^\Sigma F    \right ),
\end{equation}
and by the definition of surface gradient we have
\begin{eqnarray}\label{due}
\widetilde \partial_1^\Sigma F &=& (1-n_1^2) \widetilde \partial_1 F - n_1 n_2 \widetilde \partial_2 F- n_1 n_3 \widetilde \partial_3 F\cr
\widetilde \partial_2^\Sigma F &=& (1-n_2^2) \widetilde \partial_2 F - n_1 n_2 \widetilde \partial_1 F- n_2 n_3 \widetilde \partial_3 F.
\end{eqnarray}
Plugging \paref{due} into \paref{sviluppo} yields
\begin{eqnarray}\label{nice}
| (Z_1(\sigma), Z_2(\sigma))|^2 &=&\frac{1}{n_3^2} \Big (  c_1 (\widetilde \partial_1 F)^2 + c_2 (\widetilde \partial_2 F)^2 +  c_3 (\widetilde \partial_3 F)^2\cr
&& + c_{12} \widetilde \partial_1 F \widetilde \partial_2 F + c_{13} \widetilde \partial_1 F \widetilde \partial_3 F + c_{23} \widetilde\partial_2 F \widetilde \partial_3 F 
\Big ),
\end{eqnarray}
where 
\begin{eqnarray*}
&&c_1 = n_3^2(1 - n_1^2),\  c_2 = n_3^2(1 - n_2^2),\  c_3 = n_3^2(1 - n_3^2)\cr
&&c_{12} = -2n_3^2 n_1 n_2,\  c_{13}=-2n_3^2 n_1 n_3,\ c_{23}=-2n_3^2n_2 n_3.
\end{eqnarray*}
Substituting (\ref{nice}) into (\ref{term}) we get 
\begin{eqnarray}\label{hola2}
 \int_{\Sigma} \left ( H_2(Z_1(\sigma)) + H_2(Z_2(\sigma)) \right ) \,d\sigma &=& 
 \int_{\Sigma} \Big [\frac{1}{n_3^2} \Big (  c_1 (\widetilde \partial_1 F)^2 + c_2 (\widetilde \partial_2 F)^2 +  c_3 (\widetilde \partial_3 F)^2 + c_{12} \widetilde \partial_1 F \widetilde \partial_2 F \cr 
 &&+ c_{13} \widetilde \partial_1 F \widetilde \partial_3 F + c_{23} \widetilde\partial_2 F \widetilde \partial_3 F 
\Big ) - 2 \Big ]\,d\sigma.
\end{eqnarray}
For $i,j=1,2$ 
$$
\widetilde \partial_i F(\sigma) \widetilde \partial_j F(\sigma) = 4\pi^2 \frac{1}{M} \frac{1}{N} \sum_{\mu, \mu'} \mu_i \mu'_j a_\mu \overline{a_{\mu'}}\e_{\mu}(\sigma) \e_{-\mu'}(\sigma),
$$
hence we can rewrite \paref{hola2} isolating the contribution of the diagonal from that of off-diagonal terms as 
\begin{eqnarray}
\label{hola3}
&& \int_{\Sigma} \left ( H_2(Z_1(\sigma)) + H_2(Z_2(\sigma)) \right ) \,d\sigma \cr
&&= \frac{4\pi^2 }{M N}\left [ \sum_{i} \left (\sum_{\mu} |a_\mu |^2\mu_i^2  \int_{\Sigma} (1-n_i(\sigma)^2)\, d\sigma + \sum_{\mu\ne \mu'} a_\mu \overline{a_{\mu'}}\mu_i \mu_i' \int_{\Sigma} (1-n_i(\sigma)^2)e_\mu(\sigma) e_{-\mu'}(\sigma)\, d\sigma\right )\right.\cr
&&\left .-  \sum_{i<j}  \left ( \sum_{\mu} |a_\mu |^2\mu_i \mu_j \int_{\Sigma} 2n_i(\sigma) n_j(\sigma)\, d\sigma + \sum_{\mu\ne \mu'} a_\mu \overline{a_{\mu'}} \mu_i \mu'_j \int_{\Sigma} 2n_i(\sigma) n_j(\sigma) e_\mu(\sigma) e_{-\mu'}(\sigma)\, d\sigma \right )\right ]\cr
&&-2 \int_{\Sigma} d\sigma = \frac{ 4\pi^2}{MN} \sum_{\mu} |a_\mu |^2 \left (  \langle \mu, \mu\rangle|\Sigma | -   \int_\Sigma \langle \mu, n(\sigma)\rangle^2 \, d\sigma\right )  -2 \int_{\Sigma} d\sigma \cr
 &&  + \frac{4\pi^2}{MN}  \sum_{\mu\ne\mu'} a_\mu \overline{a_{\mu'}}  \left (\langle \mu, \mu'\rangle \int_{\Sigma} e_{\mu}(\sigma) e_{-\mu'}(\sigma)\,d\sigma-  \int_\Sigma \langle \mu, n(\sigma)\rangle \langle \mu', n(\sigma)\rangle e_{\mu}(\sigma) e_{-\mu'}(\sigma) \, d\sigma \right )\cr
&&= 3   \frac{1}{N} \sum_{\mu} |a_\mu |^2  \left ( |\Sigma | 
-    \int_\Sigma \left \langle \frac{\mu}{|\mu|}, n(\sigma)\right \rangle^2 \, d\sigma\right )  -2 |\Sigma|\cr
&& +  3\frac{1}{N} \sum_{\mu\ne\mu'} a_\mu \overline{a_{\mu'}}\int_{\Sigma} \left ( \left \langle \frac{\mu}{|\mu|}, \frac{\mu'}{|\mu'|} \right \rangle    -     \left \langle \frac{\mu}{|\mu|}, n(\sigma) \right \rangle \left \langle \frac{\mu'}{|\mu'|}, n(\sigma) \right \rangle\right ) e_{\mu}(\sigma) e_{-\mu'}(\sigma) \, d\sigma .
\end{eqnarray}
Now since 
$$
\E[|a_\mu |^2] = 1
$$
and 
$$
 \E \left [ \frac{1}{N} \sum_{\mu} |a_\mu |^2 \int_\Sigma \left \langle \frac{\mu}{|\mu|}, n(\sigma)\right \rangle^2 \, d\sigma  \right ] = \frac{1}{3} |\Sigma|,
$$
we can rewrite the r.h.s of \paref{hola3} as 
\begin{eqnarray}\label{hola4}
&& \int_{\Sigma} \left ( H_2(Z_1(\sigma)) + H_2(Z_2(\sigma)) \right ) \,d\sigma  = 3  \frac{1}{N} \sum_{\mu} (|a_\mu |^2  -1)\left (|\Sigma |  - \int_\Sigma \left \langle \frac{\mu}{|\mu|}, n(\sigma)\right \rangle^2 \, d\sigma \right )\cr
&&+  3\frac{1}{N} \sum_{\mu\ne\mu'} a_\mu \overline{a_{\mu'}}  \int_{\Sigma}\left (\left \langle \frac{\mu}{|\mu|}, \frac{\mu'}{|\mu'|} \right \rangle - \left \langle \frac{\mu}{|\mu|}, n(\sigma) \right \rangle \left \langle \frac{\mu'}{|\mu'|}, n(\sigma) \right \rangle \right ) e_{\mu}(\sigma) e_{-\mu'}(\sigma)\,d\sigma.
\end{eqnarray}
Let us set 
\begin{equation}\label{R0}
    R(0) := 3\frac{1}{N} \sum_{\mu\ne\mu'} a_\mu \overline{a_{\mu'}}  \int_{\Sigma}\left (\left \langle \frac{\mu}{|\mu|}, \frac{\mu'}{|\mu'|} \right \rangle - \left \langle \frac{\mu}{|\mu|}, n(\sigma) \right \rangle \left \langle \frac{\mu'}{|\mu'|}, n(\sigma) \right \rangle \right ) e_{\mu}(\sigma) e_{-\mu'}(\sigma)\,d\sigma.
\end{equation}
We are now in a position to write an explicit formula for the second chaotic component: from (\ref{ciao1}) and (\ref{hola4})
\begin{eqnarray}\label{hola5}
\mathcal L[2] &=& \sqrt M \left ( \frac{\beta_{2}\alpha_{0,0}}{2!} \int_{\Sigma} H_2(F(\sigma))\,d\sigma  +  \frac{\beta_{0}\alpha_{2,0}}{2!} \int_{\Sigma} (H_2(Z_1(\sigma)) + H_2(Z_2(\sigma)))\,d\sigma     \right )\cr
&=&  \sqrt M \Big[  \frac{\beta_{2}\alpha_{0,0}}{2!} |\Sigma | \frac{1}{N} \sum_{\mu} (|a_\mu |^2-1)\cr
&&+ \frac{\beta_{0}\alpha_{2,0}}{2!}  3  \frac{1}{N} \sum_{\mu} (|a_\mu |^2  -1)\left (|\Sigma |  - \int_\Sigma \left \langle \frac{\mu}{|\mu|}, n(\sigma)\right \rangle^2 \, d\sigma \right )   \cr
&& +\frac{\beta_{2}\alpha_{0,0}}{2!} \frac{1}{N} \sum_{\mu\ne \mu'} a_\mu \overline{a_{\mu'}}\int_{\Sigma}   \e_{\mu}(\sigma) \e_{-\mu'}(\sigma)\, d\sigma \cr
&&+ \frac{\beta_{0}\alpha_{2,0}}{2!}  3\frac{1}{N} \sum_{\mu\ne\mu'} a_\mu \overline{a_{\mu'}}  \int_{\Sigma}\left (\left \langle \frac{\mu}{|\mu|}, \frac{\mu'}{|\mu'|} \right \rangle - \left \langle \frac{\mu}{|\mu|}, n(\sigma) \right \rangle \left \langle \frac{\mu'}{|\mu'|}, n(\sigma) \right \rangle \right ) e_{\mu}(\sigma) e_{-\mu'}(\sigma)\,d\sigma \Big ]\cr 
& =&  \sqrt M \Big [  \frac{1}{N} \sum_{\mu} (|a_\mu |^2  -1) \left ( |\Sigma | \frac{\beta_{2}\alpha_{0,0}}{2!}  + 3 |\Sigma | \frac{\beta_{0}\alpha_{2,0}}{2!}         
-    3 \frac{\beta_{0}\alpha_{2,0}}{2!}\int_\Sigma \left \langle \frac{\mu}{|\mu|}, n(\sigma)\right \rangle^2 \, d\sigma  \right )  \cr
&&+ \frac{1}{N} \sum_{\mu\ne \mu'} a_\mu \overline{a_{\mu'}} \Big (  \frac{\beta_{2}\alpha_{0,0}}{2!} \int_{\Sigma}   \e_{\mu}(\sigma) \e_{-\mu'}(\sigma)\, d\sigma \cr
&&+3 \frac{\beta_{0}\alpha_{2,0}}{2!}  \int_{\Sigma}\left (\left \langle \frac{\mu}{|\mu|}, \frac{\mu'}{|\mu'|} \right \rangle - \left \langle \frac{\mu}{|\mu|}, n(\sigma) \right \rangle \left \langle \frac{\mu'}{|\mu'|}, n(\sigma) \right \rangle \right ) e_{\mu}(\sigma) e_{-\mu'}(\sigma)\,d\sigma \Big ) \Big ] \cr
& =& \sqrt M \Big [  \frac{1}{N} \sum_{\mu} (|a_\mu |^2  -1) \left ( |\Sigma | \frac{1}{8}
-    \frac{3}{8}\int_\Sigma \left \langle \frac{\mu}{|\mu|}, n(\sigma)\right \rangle^2 \, d\sigma  \right )     \cr
&&+ \frac{1}{N} \sum_{\mu\ne \mu'} a_\mu \overline{a_{\mu'}} \Big (  \int_{\Sigma}e_{\mu}(\sigma) e_{-\mu'}(\sigma) \Big ( -\frac14 + \frac38 \left (\left \langle \frac{\mu}{|\mu|}, \frac{\mu'}{|\mu'|} \right \rangle \right.\cr
&&\left.- \left \langle \frac{\mu}{|\mu|}, n(\sigma) \right \rangle \left \langle \frac{\mu'}{|\mu'|}, n(\sigma) \right \rangle \right )\Big )d\sigma \Big ) \Big ] \cr 
&=& \frac{\sqrt M}{8 \sqrt{\frac{N}{2} }} \Big [  \frac{1}{\sqrt{\frac{N}{2} }} \sum_{\mu\in \mathcal E^+} (|a_\mu |^2  -1) \left ( |\Sigma | 
-   3\int_\Sigma \left \langle \frac{\mu}{|\mu|}, n(\sigma)\right \rangle^2 \, d\sigma  \right )\Big ]  \cr
&&+ \frac{\sqrt{M}}{N} \sum_{\mu\ne \mu'} a_\mu \overline{a_{\mu'}} \Big (  \int_{\Sigma}e_{\mu}(\sigma) e_{-\mu'}(\sigma) \Big ( -\frac14 + \frac38  \left \langle \frac{\mu}{|\mu|}, \frac{\mu'}{|\mu'|} \right \rangle \cr
&&- \frac38 \left \langle \frac{\mu}{|\mu|}, n(\sigma) \right \rangle \left \langle \frac{\mu'}{|\mu'|}, n(\sigma) \right \rangle \Big )d\sigma \Big ) \Big ]\cr
& =& \mathcal L^a[2] + \mathcal L^b[2]
\end{eqnarray}
as defined in (\ref{2a}) and (\ref{2b}) thus concluding the proof.
\end{proof}

\subsection{Proof of Lemma \ref{4dec}}\label{app4chaos}

Let us first recall some properties of $\mathcal C(4)$ as defined in (\ref{l-corr}). We have 
\begin{equation}\label{c4}
    \mathcal C(4) = \mathcal D'(4) \cup \mathcal X(4),
\end{equation}
where $\mathcal X(4)$ is the set of non-degenerate $4$-length correlations, see (\ref{defxl}), and we denote by $\mathcal D'(4)=\mathcal D'_m(4)$ the complement of $\mathcal X(4)$ in $\mathcal C(4)$. It is known \cite[Section 2]{benmaf} that $\mathcal D'(4)$ coincide with the set of quadruples that \emph{cancel out in pairs}. 

\begin{proof}
Let us start investigating the first summand on the r.h.s. of \paref{4chaos}: 
\begin{eqnarray*}
\int_{\Sigma} H_4(F(\sigma))\,d\sigma 
&
=& \int_{\Sigma} (F(\sigma)^4 - 6F(\sigma)^2 + 3)\,d\sigma \cr 
& =&
\frac{1}{N^2} \sum_{\mu, \mu', \mu'', \mu'''} a_{\mu} \overline{a_{\mu'}} a_{\mu''} \overline{a_{\mu'''}} \int_{\Sigma} \e_{\mu - \mu' + \mu'' - \mu'''}(\sigma)\,d\sigma \cr
&& 
- 6 \frac{1}{N} \sum_{\mu,\mu'} a_\mu \overline{a_{\mu'}} \int_{\Sigma} \e_{\mu -\mu'}(\sigma)\,d\sigma + 3\int_{\Sigma} d\sigma\cr
&=&\frac{1}{N^2} \sum_{(\mu, \mu', \mu'', \mu''')\in \mathcal C(4)
} a_{\mu} \overline{a_{\mu'}} a_{\mu''} \overline{a_{\mu'''}} 
|\Sigma| 
- 6 \frac{1}{N} \sum_{\mu} |a_\mu |^2 |\Sigma| 
 + 3|\Sigma|\cr
 &&+ \frac{1}{N^2} \sum_{(\mu, \mu', \mu'', \mu''')\notin \mathcal C(4)
} a_{\mu} \overline{a_{\mu'}} a_{\mu''} \overline{a_{\mu'''}}\int_{\Sigma} \e_{\mu - \mu' + \mu'' - \mu'''}(\sigma)\,d\sigma \cr
&&- 6 \frac{1}{N} \sum_{\mu\ne \mu'} a_\mu \overline{a_{\mu'}} \int_{\Sigma} \e_{\mu -\mu'}(\sigma)\,d\sigma.
\end{eqnarray*}
Recalling the structure of $\mathcal C(4)$ in (\ref{c4}), and in particular that of $\mathcal D'(4)$, we can write  
\begin{eqnarray}\label{4chaos1}
\int_{\Sigma} H_4(F(\sigma))\,d\sigma 
&
=& \int_{\Sigma} (F(\sigma)^4 - 6F(\sigma)^2 + 3)\,d\sigma \cr 
&=&\frac{1}{N^2} \sum_{(\mu, \mu', \mu'', \mu''')\in \mathcal D'(4)
} a_{\mu} \overline{a_{\mu'}} a_{\mu''} \overline{a_{\mu'''}} 
|\Sigma| + \frac{1}{N^2} \sum_{(\mu, \mu', \mu'', \mu''')\in \mathcal X(4)
} a_{\mu} \overline{a_{\mu'}} a_{\mu''} \overline{a_{\mu'''}} 
|\Sigma| \cr
&&- 6 \frac{1}{N} \sum_{\mu} |a_\mu |^2 |\Sigma| 
 + 3|\Sigma|\cr
 &&+ \frac{1}{N^2} \sum_{(\mu, \mu', \mu'', \mu''')\notin \mathcal C(4)
} a_{\mu} \overline{a_{\mu'}} a_{\mu''} \overline{a_{\mu'''}}\int_{\Sigma} \e_{\mu - \mu' + \mu'' - \mu'''}(\sigma)\,d\sigma \cr
&&- 6 \frac{1}{N} \sum_{\mu\ne \mu'} a_\mu \overline{a_{\mu'}} \int_{\Sigma} \e_{\mu -\mu'}(\sigma)\,d\sigma\cr
&=&3
\frac{1}{N^2} \sum_{\mu, \mu'
} |a_\mu |^2 |a_{\mu'}|^2 
|\Sigma| 
- 6 \frac{1}{N} \sum_{\mu} |a_\mu |^2 |\Sigma| - 3 
\frac{1}{N^2} \sum_{\mu
} |a_\mu |^4
 + 3|\Sigma| \cr 
 &&+\frac{1}{N^2} \sum_{(\mu, \mu', \mu'', \mu''')\in \mathcal X(4)
} a_{\mu} \overline{a_{\mu'}} a_{\mu''} \overline{a_{\mu'''}} 
|\Sigma| \cr
&&+ \frac{1}{N^2} \sum_{(\mu, \mu', \mu'', \mu''')\notin \mathcal C(4)
} a_{\mu} \overline{a_{\mu'}} a_{\mu''} \overline{a_{\mu'''}}\int_{\Sigma} \e_{\mu - \mu' + \mu'' - \mu'''}(\sigma)\,d\sigma\cr 
&&- 6 \frac{1}{N} \sum_{\mu\ne \mu'} a_\mu \overline{a_{\mu'}} \int_{\Sigma} \e_{\mu -\mu'}(\sigma)\,d\sigma\cr 
 & =& 3\cdot 2 |\Sigma | \frac{1}{N} \left (\frac{1}{\sqrt{N/2}} \sum_{\mu\in \mathcal E^+}
 (|a_\mu |^2 - 1) \right )^2 - 3 
\frac{1}{N}\left ( \frac{1}{N}\sum_{\mu
} |a_\mu |^4 \right )\cr
&&  +\frac{1}{N^2} \sum_{(\mu, \mu', \mu'', \mu''')\in \mathcal X(4)
} a_{\mu} \overline{a_{\mu'}} a_{\mu''} \overline{a_{\mu'''}} 
|\Sigma|\cr
&&+ \frac{1}{N^2} \sum_{(\mu, \mu', \mu'', \mu''')\notin \mathcal C(4)
} a_{\mu} \overline{a_{\mu'}} a_{\mu''} \overline{a_{\mu'''}}\int_{\Sigma} \e_{\mu - \mu' + \mu'' - \mu'''}(\sigma)\,d\sigma\cr
&&- 6 \frac{1}{N} \sum_{\mu\ne \mu'} a_\mu \overline{a_{\mu'}} \int_{\Sigma} \e_{\mu -\mu'}(\sigma)\,d\sigma.
\end{eqnarray}
Let us set
\begin{eqnarray}\label{resto1}
    R(1) &:=&\frac{1}{N^2} \sum_{(\mu, \mu', \mu'', \mu''')\in \mathcal X(4)
} a_{\mu} \overline{a_{\mu'}} a_{\mu''} \overline{a_{\mu'''}} 
|\Sigma|\cr
&&+ \frac{1}{N^2} \sum_{(\mu, \mu', \mu'', \mu''')\notin \mathcal C(4)
} a_{\mu} \overline{a_{\mu'}} a_{\mu''} \overline{a_{\mu'''}}\int_{\Sigma} \e_{\mu - \mu' + \mu'' - \mu'''}(\sigma)\,d\sigma\cr 
&&- 6 \frac{1}{N} \sum_{\mu\ne \mu'} a_\mu \overline{a_{\mu'}} \int_{\Sigma} \e_{\mu -\mu'}(\sigma)\,d\sigma.
\end{eqnarray}
%
Let us now investigate the second term on the r.h.s. of \paref{4chaos}.
\begin{eqnarray}\label{4c2}
&&\int_{\Sigma} H_2(F(\sigma)) (
H_2(Z_1(\sigma)) +  H_2(Z_2(
\sigma)) )
\,d\sigma  =  \int_{\Sigma} (F(\sigma)^2 -1)
(
Z_1(\sigma)^2 +  Z_2(
\sigma)^2 - 2 )
\,d\sigma \cr
&&= \int_{\Sigma} (F(\sigma)^2 -1)
(
\| (Z_1(\sigma), Z_2(
\sigma))\|^2 - 2 )
\,d\sigma \cr
&&= \int_{\Sigma} (F(\sigma)^2 -1)\Big [\frac{1}{n_3^2} \Big (  c_1 (\widetilde \partial_1 F)^2 + c_2 (\widetilde \partial_2 F)^2 +  c_3 (\widetilde \partial_3 F)^2 + c_{12} \widetilde \partial_1 F \widetilde \partial_2 F + c_{13} \widetilde \partial_1 F \widetilde \partial_3 F\cr
&&+ c_{23} \widetilde\partial_2 F \widetilde \partial_3 F \Big ) - 2 \Big ]\,d\sigma \cr
&& = \frac{4\pi^2}{M} \frac{1}{N^2} \sum_{\mu, \mu', \mu'', \mu'''} a_\mu \overline{a_{\mu'}} a_{\mu''} \overline{a_{\mu'''}} \mu_1 \mu'_1 \int_{\Sigma} (1-n_1(\sigma)^2)\, \e_{\mu -\mu' + \mu'' - \mu'''}(\sigma)\, d\sigma \cr
 && +  \frac{4\pi^2}{M} \frac{1}{N^2} \sum_{\mu, \mu', \mu'', \mu'''} a_\mu \overline{a_{\mu'}} a_{\mu''} \overline{a_{\mu'''}} \mu_2 \mu'_2 \int_{\Sigma} (1-n_2(\sigma)^2)\, \e_{\mu -\mu' + \mu'' - \mu'''}(\sigma)\, d\sigma \cr
  && +  \frac{4\pi^2}{M} \frac{1}{N^2} \sum_{\mu, \mu', \mu'', \mu'''} a_\mu \overline{a_{\mu'}} a_{\mu''} \overline{a_{\mu'''}} \mu_3 \mu'_3 \int_{\Sigma} (1-n_3(\sigma)^2)\, \e_{\mu -\mu' + \mu'' - \mu'''}(\sigma)\, d\sigma \cr
  && -  \frac{4\pi^2}{M} \frac{1}{N^2} \sum_{\mu, \mu', \mu'', \mu'''} a_\mu \overline{a_{\mu'}} a_{\mu''} \overline{a_{\mu'''}} \mu_1 \mu'_2 \int_{\Sigma} 2n_1(\sigma)n_2(\sigma))\, \e_{\mu -\mu' + \mu'' - \mu'''}(\sigma)\, d\sigma \cr
   && -  \frac{4\pi^2}{M} \frac{1}{N^2} \sum_{\mu, \mu', \mu'', \mu'''} a_\mu \overline{a_{\mu'}} a_{\mu''} \overline{a_{\mu'''}} \mu_1 \mu'_3 \int_{\Sigma} 2n_1(\sigma)n_3(\sigma))\, \e_{\mu -\mu' + \mu'' - \mu'''}(\sigma)\, d\sigma \cr
   && -  \frac{4\pi^2}{M} \frac{1}{N^2} \sum_{\mu, \mu', \mu'', \mu'''} a_\mu \overline{a_{\mu'}} a_{\mu''} \overline{a_{\mu'''}} \mu_2 \mu'_3 \int_{\Sigma} 2n_2(\sigma)n_3(\sigma))\, \e_{\mu -\mu' + \mu'' - \mu'''}(\sigma)\, d\sigma \cr
   && -  \frac{2}{N} \sum_{\mu, \mu'} a_\mu \overline{a_{\mu'}} \int_{\Sigma}  \e_{\mu -\mu'}(\sigma)\, d\sigma \cr
 &&  - \int_{\Sigma} \Big [\frac{1}{n_3^2} \Big (  c_1 (\widetilde \partial_1 F)^2 + c_2 (\widetilde \partial_2 F)^2 +  c_3 (\widetilde \partial_3 F)^2 + c_{12} \widetilde \partial_1 F \widetilde \partial_2 F + c_{13} \widetilde \partial_1 F \widetilde \partial_3 F\cr
 &&+ c_{23} \widetilde\partial_2 F \widetilde \partial_3 F \Big ) - 2 \Big ]\,d\sigma. 
\end{eqnarray}
The last summand on the r.h.s. of (\ref{4c2}) has been already investigated in Section \ref{sec_lem2}, see (\ref{hola4}). For the other summands, repeating the same argument as for (\ref{4chaos1}), we have 
\begin{eqnarray}\label{argh1}
&&\int_{\Sigma} H_2(F(\sigma)) (
H_2(Z_1(\sigma)) +  H_2(Z_2(
\sigma)) )
\,d\sigma  =  \int_{\Sigma} (F(\sigma)^2 -1)
(
Z_1(\sigma)^2 +  Z_2(
\sigma)^2 - 2 )
\,d\sigma \cr
&& = \frac{4\pi^2}{M} \frac{1}{N^2} \sum_{\mu, \mu'} |a_\mu |^2 |a_{\mu'} |^2 \mu_1^2 \int_{\Sigma} (1-n_1(\sigma)^2)\,  d\sigma + 2 \cdot \frac{4\pi^2}{M} \underbrace{\frac{1}{N^2} \sum_{\mu, \mu'} |a_\mu |^2 |a_{\mu'} |^2 \mu_1 \mu'_1}_{=0}
\int_{\Sigma} (1-n_1(\sigma)^2)\,  d\sigma\cr
&& - \frac{4\pi^2}{M} \frac{1}{N^2} \sum_{\mu} |a_\mu |^4 \mu_1^2 \int_{\Sigma} (1-n_1(\sigma)^2)\,  d\sigma \cr
&&+ \frac{4\pi^2}{M} \frac{1}{N^2} \sum_{(\mu, \mu', \mu'', \mu''')\in \mathcal X(4)} a_\mu \overline{a_{\mu'}} a_{\mu''} \overline{a_{\mu'''}} \mu_1 \mu'_1 \int_{\Sigma} (1-n_1(\sigma)^2)\, d\sigma\cr
&&+ \frac{4\pi^2}{M} \frac{1}{N^2} \sum_{(\mu, \mu', \mu'', \mu''')\notin \mathcal C(4)} a_\mu \overline{a_{\mu'}} a_{\mu''} \overline{a_{\mu'''}} \mu_1 \mu'_1 \int_{\Sigma} (1-n_1(\sigma)^2)\, \e_{\mu -\mu' + \mu'' - \mu'''}(\sigma)\, d\sigma\cr
&& + \frac{4\pi^2}{M} \frac{1}{N^2} \sum_{\mu, \mu'} |a_\mu |^2 |a_{\mu'} |^2 \mu_2^2 \int_{\Sigma} (1-n_2(\sigma)^2)\,  d\sigma + 2 \cdot \frac{4\pi^2}{M} \underbrace{\frac{1}{N^2} \sum_{\mu, \mu'} |a_\mu |^2 |a_{\mu'} |^2 \mu_2 \mu'_2}_{=0} \int_{\Sigma} (1-n_2(\sigma)^2)\,  d\sigma\cr
&& - \frac{4\pi^2}{M} \frac{1}{N^2} \sum_{\mu} |a_\mu |^4 \mu_2^2 \int_{\Sigma} (1-n_2(\sigma)^2)\,  d\sigma \cr
&& +  \frac{4\pi^2}{M} \frac{1}{N^2} \sum_{(\mu, \mu', \mu'', \mu''')\in \mathcal X(4)} a_\mu \overline{a_{\mu'}} a_{\mu''} \overline{a_{\mu'''}} \mu_2 \mu'_2 \int_{\Sigma} (1-n_2(\sigma)^2)\, d\sigma \cr
&& +  \frac{4\pi^2}{M} \frac{1}{N^2} \sum_{(\mu, \mu', \mu'', \mu''')\notin \mathcal C(4)} a_\mu \overline{a_{\mu'}} a_{\mu''} \overline{a_{\mu'''}} \mu_2 \mu'_2 \int_{\Sigma} (1-n_2(\sigma)^2)\, \e_{\mu -\mu' + \mu'' - \mu'''}(\sigma)\, d\sigma \cr
&& + \frac{4\pi^2}{M} \frac{1}{N^2} \sum_{\mu, \mu'} |a_\mu |^2 |a_{\mu'} |^2 \mu_3^2 \int_{\Sigma} (1-n_3(\sigma)^2)\,  d\sigma + 2 \cdot \frac{4\pi^2}{M} \underbrace{\frac{1}{N^2} \sum_{\mu, \mu'} |a_\mu |^2 |a_{\mu'} |^2 \mu_3 \mu'_3}_{=0} \int_{\Sigma} (1-n_3(\sigma)^2)\,  d\sigma\cr
&& - \frac{4\pi^2}{M} \frac{1}{N^2} \sum_{\mu} |a_\mu |^4 \mu_3^2 \int_{\Sigma} (1-n_3(\sigma)^2)\,  d\sigma \cr
&& +  \frac{4\pi^2}{M} \frac{1}{N^2} \sum_{(\mu, \mu', \mu'', \mu''')\in \mathcal X(4)} a_\mu \overline{a_{\mu'}} a_{\mu''} \overline{a_{\mu'''}} \mu_3 \mu'_3 \int_{\Sigma} (1-n_3(\sigma)^2)\, d\sigma \cr
&& +  \frac{4\pi^2}{M} \frac{1}{N^2} \sum_{(\mu, \mu', \mu'', \mu''')\notin \mathcal C(4)} a_\mu \overline{a_{\mu'}} a_{\mu''} \overline{a_{\mu'''}} \mu_3 \mu'_3 \int_{\Sigma} (1-n_3(\sigma)^2)\, \e_{\mu -\mu' + \mu'' - \mu'''}(\sigma)\, d\sigma \cr
&& - \frac{4\pi^2}{M} \frac{1}{N^2} \sum_{\mu, \mu'} |a_\mu |^2 |a_{\mu'} |^2 \mu_1\mu_2 \int_{\Sigma} 2n_1(\sigma) n_2(\sigma)\,  d\sigma - 2 \cdot \frac{4\pi^2}{M} \underbrace{\frac{1}{N^2} \sum_{\mu, \mu'} |a_\mu |^2 |a_{\mu'} |^2 \mu_1 \mu'_2}_{=0} \int_{\Sigma}2n_1(\sigma) n_2(\sigma)\,  d\sigma\cr
&& + \frac{4\pi^2}{M} \frac{1}{N^2} \sum_{\mu} |a_\mu |^4 \mu_1\mu_2 \int_{\Sigma} 2n_1(\sigma) n_2(\sigma)\,  d\sigma \cr
&& -  \frac{4\pi^2}{M} \frac{1}{N^2} \sum_{(\mu, \mu', \mu'', \mu''')\in \mathcal X(4)} a_\mu \overline{a_{\mu'}} a_{\mu''} \overline{a_{\mu'''}} \mu_1 \mu'_2 \int_{\Sigma} 2n_1(\sigma)n_2(\sigma))\,d\sigma \cr
&& -  \frac{4\pi^2}{M} \frac{1}{N^2} \sum_{(\mu, \mu', \mu'', \mu''')\notin \mathcal C(4)} a_\mu \overline{a_{\mu'}} a_{\mu''} \overline{a_{\mu'''}} \mu_1 \mu'_2 \int_{\Sigma} 2n_1(\sigma)n_2(\sigma))\, \e_{\mu -\mu' + \mu'' - \mu'''}(\sigma)\, d\sigma \cr
&& - \frac{4\pi^2}{M} \frac{1}{N^2} \sum_{\mu, \mu'} |a_\mu |^2 |a_{\mu'} |^2 \mu_1\mu_3 \int_{\Sigma} 2n_1(\sigma) n_3(\sigma)\,  d\sigma - 2 \cdot \frac{4\pi^2}{M} \underbrace{\frac{1}{N^2} \sum_{\mu, \mu'} |a_\mu |^2 |a_{\mu'} |^2 \mu_1 \mu'_3}_{=0} \int_{\Sigma}2n_1(\sigma) n_3(\sigma)\,  d\sigma\cr
&& + \frac{4\pi^2}{M} \frac{1}{N^2} \sum_{\mu} |a_\mu |^4 \mu_1\mu_3 \int_{\Sigma} 2n_1(\sigma) n_3(\sigma)\,  d\sigma \cr
&& -  \frac{4\pi^2}{M} \frac{1}{N^2} \sum_{(\mu, \mu', \mu'', \mu''')\in \mathcal X(4)} a_\mu \overline{a_{\mu'}} a_{\mu''} \overline{a_{\mu'''}} \mu_1 \mu'_3 \int_{\Sigma} 2n_1(\sigma)n_3(\sigma))\,d\sigma \cr
&& -  \frac{4\pi^2}{M} \frac{1}{N^2} \sum_{(\mu, \mu', \mu'', \mu''')\notin \mathcal C(4)} a_\mu \overline{a_{\mu'}} a_{\mu''} \overline{a_{\mu'''}} \mu_1 \mu'_3 \int_{\Sigma} 2n_1(\sigma)n_3(\sigma))\, \e_{\mu -\mu' + \mu'' - \mu'''}(\sigma)\, d\sigma \cr
&& - \frac{4\pi^2}{M} \frac{1}{N^2} \sum_{\mu, \mu'} |a_\mu |^2 |a_{\mu'} |^2 \mu_2\mu_3 \int_{\Sigma} 2n_2(\sigma) n_3(\sigma)\,  d\sigma - 2 \cdot \frac{4\pi^2}{M} \underbrace{\frac{1}{N^2} \sum_{\mu, \mu'} |a_\mu |^2 |a_{\mu'} |^2 \mu_2 \mu'_3}_{=0} \int_{\Sigma}2n_2(\sigma) n_3(\sigma)\,  d\sigma\cr
&& + \frac{4\pi^2}{M} \frac{1}{N^2} \sum_{\mu} |a_\mu |^4 \mu_2\mu_3 \int_{\Sigma} 2n_2(\sigma) n_3(\sigma)\,  d\sigma \cr
&& -  \frac{4\pi^2}{M} \frac{1}{N^2} \sum_{(\mu, \mu', \mu'', \mu''')\in \mathcal X(4)} a_\mu \overline{a_{\mu'}} a_{\mu''} \overline{a_{\mu'''}} \mu_2 \mu'_3 \int_{\Sigma} 2n_2(\sigma)n_3(\sigma))\,d\sigma \cr
&& -  \frac{4\pi^2}{M} \frac{1}{N^2} \sum_{(\mu, \mu', \mu'', \mu''')\notin \mathcal C(4)} a_\mu \overline{a_{\mu'}} a_{\mu''} \overline{a_{\mu'''}} \mu_2 \mu'_3 \int_{\Sigma} 2n_2(\sigma)n_3(\sigma))\, \e_{\mu -\mu' + \mu'' - \mu'''}(\sigma)\, d\sigma \cr
&& -  \frac{2}{N} \sum_{\mu} |a_\mu |^2 |\Sigma | -  \frac{2}{N} \sum_{\mu\ne \mu'} a_\mu \overline{a_{\mu'}} \int_{\Sigma}  \e_{\mu -\mu'}(\sigma)\, d\sigma  \cr
&&- 3  \frac{1}{N} \sum_{\mu} (|a_\mu |^2  -1)\left (|\Sigma |  - \int_\Sigma \left \langle \frac{\mu}{|\mu|}, n(\sigma)\right \rangle^2 \, d\sigma \right )\cr
&&-  3\frac{1}{N} \sum_{\mu\ne\mu'} a_\mu \overline{a_{\mu'}}  \int_{\Sigma}\left (\left \langle \frac{\mu}{|\mu|}, \frac{\mu'}{|\mu'|} \right \rangle - \left \langle \frac{\mu}{|\mu|}, n(\sigma) \right \rangle \left \langle \frac{\mu'}{|\mu'|}, n(\sigma) \right \rangle \right ) e_{\mu}(\sigma) e_{-\mu'}(\sigma)\,d\sigma.
\end{eqnarray}
Let us rewrite the r.h.s. of (\ref{argh1}) in a more compact way:
\begin{eqnarray}
&&\int_{\Sigma} H_2(F(\sigma)) (
H_2(Z_1(\sigma)) +  H_2(Z_2(
\sigma)) )
\,d\sigma  =  \int_{\Sigma} (F(\sigma)^2 -1)
(
Z_1(\sigma)^2 +  Z_2(
\sigma)^2 - 2 )
\,d\sigma \cr
 && = 3 \frac{1}{N^2} \sum_{\mu, \mu'} (|a_\mu |^2-1)( |a_{\mu'}|^2 -1)\Big (|\Sigma | -   \int_\Sigma \left \langle \frac{\mu}{|\mu|}, n(\sigma)\right \rangle^2 \, d\sigma        \Big ) \cr
 && - 3 \frac{1}{N^2} \sum_{\mu} |a_\mu |^4\Big (|\Sigma | -   \int_\Sigma \left \langle \frac{\mu}{|\mu|}, n(\sigma)\right \rangle^2 \, d\sigma        \Big )\cr
 && +\frac{3}{N^2} \sum_{(\mu,\mu',\mu'',\mu''')\in \mathcal X(4)} a_\mu \overline{a_{\mu'}} a_{\mu''} \overline{a_{\mu'''}} \left ( |\Sigma| \left \langle \frac{\mu}{|\mu|},\frac{\mu'}{|\mu|} \right \rangle - \int_\Sigma \left \langle \frac{\mu}{|\mu|},n(\sigma) \right \rangle \left \langle \frac{\mu'}{|\mu'|},n(\sigma) \right \rangle\,d\sigma \right ) \cr
  && +\frac{3}{N^2} \sum_{(\mu,\mu',\mu'',\mu''')\notin \mathcal C(4)} a_\mu \overline{a_{\mu'}} a_{\mu''} \overline{a_{\mu'''}} \left (  \left \langle \frac{\mu}{|\mu|},\frac{\mu'}{|\mu|} \right \rangle \int_\Sigma \text{e}_{\mu-\mu'+\mu''-\mu'''}(\sigma) \,d\sigma \right.\cr
  &&
  \left.- \int_\Sigma \left \langle \frac{\mu}{|\mu|},n(\sigma) \right \rangle \left \langle \frac{\mu'}{|\mu'|},n(\sigma) \right \rangle\text{e}_{\mu-\mu'+\mu''-\mu'''}(\sigma)\,d\sigma \right )  -  \frac{2}{N} \sum_{\mu\ne \mu'} a_\mu \overline{a_{\mu'}} \int_{\Sigma}  \e_{\mu -\mu'}(\sigma)\, d\sigma \cr
  &&-  3\frac{1}{N} \sum_{\mu\ne\mu'} a_\mu \overline{a_{\mu'}}  \int_{\Sigma}\left (\left \langle \frac{\mu}{|\mu|}, \frac{\mu'}{|\mu'|} \right \rangle - \left \langle \frac{\mu}{|\mu|}, n(\sigma) \right \rangle \left \langle \frac{\mu'}{|\mu'|}, n(\sigma) \right \rangle \right ) e_{\mu}(\sigma) e_{-\mu'}(\sigma)\,d\sigma.
\end{eqnarray}
As before, we set
\begin{eqnarray}
R(2)&:=& \frac{3}{N^2} \sum_{(\mu,\mu',\mu'',\mu''')\in \mathcal X(4)} a_\mu \overline{a_{\mu'}} a_{\mu''} \overline{a_{\mu'''}} \left ( |\Sigma| \left \langle \frac{\mu}{|\mu|},\frac{\mu'}{|\mu|} \right \rangle - \int_\Sigma \left \langle \frac{\mu}{|\mu|},n(\sigma) \right \rangle \left \langle \frac{\mu'}{|\mu'|},n(\sigma) \right \rangle\,d\sigma \right ) \cr
  && +\frac{3}{N^2} \sum_{(\mu,\mu',\mu'',\mu''')\notin \mathcal C(4)} a_\mu \overline{a_{\mu'}} a_{\mu''} \overline{a_{\mu'''}} \left (  \left \langle \frac{\mu}{|\mu|},\frac{\mu'}{|\mu|} \right \rangle \int_\Sigma \text{e}_{\mu-\mu'+\mu''-\mu'''}(\sigma) \,d\sigma \right.\cr
  &&
  \left.- \int_\Sigma \left \langle \frac{\mu}{|\mu|},n(\sigma) \right \rangle \left \langle \frac{\mu'}{|\mu'|},n(\sigma) \right \rangle\text{e}_{\mu-\mu'+\mu''-\mu'''}(\sigma)\,d\sigma \right )  -  \frac{2}{N} \sum_{\mu\ne \mu'} a_\mu \overline{a_{\mu'}} \int_{\Sigma}  \e_{\mu -\mu'}(\sigma)\, d\sigma \cr
  &&-  3\frac{1}{N} \sum_{\mu\ne\mu'} a_\mu \overline{a_{\mu'}}  \int_{\Sigma}\left (\left \langle \frac{\mu}{|\mu|}, \frac{\mu'}{|\mu'|} \right \rangle - \left \langle \frac{\mu}{|\mu|}, n(\sigma) \right \rangle \left \langle \frac{\mu'}{|\mu'|}, n(\sigma) \right \rangle \right ) e_{\mu}(\sigma) e_{-\mu'}(\sigma)\,d\sigma.
\end{eqnarray}

Let us study the remaining terms on the r.h.s. of \paref{4chaos}: 
\begin{eqnarray}\label{ccc1}
  && \frac{\beta_0\alpha_{2,2}}{2! 2!} \int_{\Sigma} H_2(Z_1(
\sigma)) H_2(Z_2(
\sigma))\,d\sigma + \frac{\beta_0 \alpha_{4,0}}{4!}\int_{\Sigma} H_4(Z_1(\sigma))\,d\sigma + \frac{\beta_0 \alpha_{4,0}}{4!}\int_{\Sigma} H_4(Z_2(\sigma))\,d\sigma \cr
  && = -\frac{1}{16\cdot 8}\left (\int_{\Sigma} [ (Z_1(\sigma)^2 + Z_2(\sigma)^2)^2 - 8]\,d\sigma -8 \int_\Sigma  [Z_1(\sigma)^2 + Z_2(\sigma)^2 -2]\, d\sigma\right ).
\end{eqnarray}
It suffices to deal with the first summand on the r.h.s. of (\ref{ccc1}) since the second term has been already investigated in (\ref{hola4}). We have 
\begin{eqnarray*}
&&\int_{\Sigma} [ (Z_1(\sigma)^2 + Z_2(\sigma)^2)^2 - 8]\, d\sigma \cr
 &&= \int_{\Sigma} \big [ \Big (  b_1 (\widetilde \partial_1 F)^2 + b_2 (\widetilde \partial_2 F)^2 +  b_3 (\widetilde \partial_3 F)^2 + b_{12} \widetilde \partial_1 F \widetilde \partial_2 F + b_{13} \widetilde \partial_1 F \widetilde \partial_3 F + b_{23} \widetilde\partial_2 F \widetilde \partial_3 F 
\Big )^2 - 8 \Big ]\,d\sigma \cr
      && = \int_{\Sigma} \Big [ b_1^2 (\widetilde \partial_1 F)^4 +  b_2^2 (\widetilde \partial_2 F)^4 +  b_3^2 (\widetilde \partial_3 F)^4 \cr
      &&+ b_{12}^2 (\widetilde \partial_1 F )^2(\widetilde \partial_2 F)^2 + b_{13}^2 (\widetilde \partial_1 F )^2(\widetilde \partial_3 F)^2 + b_{23}^2 (\widetilde \partial_2 F )^2(\widetilde \partial_3 F)^2 \cr
&&+ 2 b_1 b_2 (\widetilde \partial_1 F)^2(\widetilde \partial_2 F)^2 + 2 b_1 b_3 (\widetilde \partial_1 F)^2(\widetilde \partial_3 F)^2 \cr
 && + 2b_1 b_{12} (\widetilde \partial_1 F)^3 \widetilde \partial_2 F + 2b_1 b_{13} (\widetilde \partial_1 F)^3 \widetilde \partial_3 F  + 2b_1 b_{23} (\widetilde \partial_1 F)^2 \widetilde \partial_2 F \widetilde \partial_3 F \cr
&&+  2 b_2 b_3 (\widetilde \partial_2 F)^2(\widetilde \partial_3 F)^2 \cr
 && + 2b_2 b_{12} (\widetilde \partial_2 F)^3 \widetilde \partial_1 F + 2b_2 b_{13} (\widetilde \partial_2 F)^2 \widetilde \partial_1 F\widetilde \partial_3 F  + 2b_2 b_{23} (\widetilde \partial_2 F)^3 \widetilde \partial_3 F \cr
  && + 2b_3 b_{12} (\widetilde \partial_3 F)^2 \widetilde \partial_1 F \widetilde \partial_2 F  \widetilde \partial_3 F + 2b_3 b_{13} (\widetilde \partial_3 F)^3 \widetilde \partial_1 F  + 2b_3 b_{23} (\widetilde \partial_3 F)^3 \widetilde \partial_2 F \cr
 && + 2 b_{12}b_{13} (\widetilde \partial_1 F)^2 \widetilde \partial_2 F \widetilde \partial_3 F + 2 b_{12}b_{23} (\widetilde \partial_2 F)^2 \widetilde \partial_1 F \widetilde \partial_3 F + 2 b_{13}b_{23} (\widetilde \partial_3 F)^2 \widetilde \partial_1 F \widetilde \partial_2 F - 8 \Big ]\, d\sigma. 
\end{eqnarray*}
We start with 
\begin{eqnarray}\label{crazy1}
\int_{\Sigma}  b_1^2 (\widetilde \partial_1 F)^4\, d\sigma &=& \frac{(4\pi^2)^2}{M^2} \sum_{\mu,\mu',\mu'',\mu'''} \mu_1 \mu'_1 \mu''_1 \mu'''_1 a_\mu \overline{a_{\mu'}} a_{\mu''} \overline{a_{\mu'''}} \int_{\Sigma} b_1(\sigma)^2 \e_{\mu-\mu'+\mu''-\mu'''}(\sigma)\,d\sigma \cr
 &=& 3 \frac{(4\pi^2)^2}{M^2} \sum_{\mu,\mu'} \mu_1^2 (\mu'_1)^2 |a_\mu |^2 |a_{\mu'}|^2 \int_{\Sigma} b_1(\sigma)^2 \,d\sigma\cr
 &&- 3 \frac{(4\pi^2)^2}{M^2} \sum_{\mu} \mu_1^4  |a_\mu |^4 \int_{\Sigma} b_1(\sigma)^2 \,d\sigma\cr
 &&+ \frac{(4\pi^2)^2}{M^2} \sum_{(\mu,\mu',\mu'',\mu''')\in \mathcal X(4)} \mu_1 \mu'_1 \mu''_1 \mu'''_1 a_\mu \overline{a_{\mu'}} a_{\mu''} \overline{a_{\mu'''}} \int_{\Sigma} b_1(\sigma)^2 \,d\sigma \cr
 &&+ \frac{(4\pi^2)^2}{M^2} \sum_{(\mu,\mu',\mu'',\mu''')\notin \mathcal C(4)} \mu_1 \mu'_1 \mu''_1 \mu'''_1 a_\mu \overline{a_{\mu'}} a_{\mu''} \overline{a_{\mu'''}}\cr
 &&\times \int_{\Sigma} b_1(\sigma)^2\e_{\mu-\mu'+\mu''-\mu'''}(\sigma) \,d\sigma
\end{eqnarray}
and we set
\begin{eqnarray}
R_{1111}(3)&:=& \frac{(4\pi^2)^2}{M^2} \sum_{(\mu,\mu',\mu'',\mu''')\in \mathcal X(4)} \mu_1 \mu'_1 \mu''_1 \mu'''_1 a_\mu \overline{a_{\mu'}} a_{\mu''} \overline{a_{\mu'''}} \int_{\Sigma} b_1(\sigma)^2 \,d\sigma \cr
 &&+ \frac{(4\pi^2)^2}{M^2} \sum_{(\mu,\mu',\mu'',\mu''')\notin \mathcal C(4)} \mu_1 \mu'_1 \mu''_1 \mu'''_1 a_\mu \overline{a_{\mu'}} a_{\mu''} \overline{a_{\mu'''}}\cr
 &&\times \int_{\Sigma} b_1(\sigma)^2\e_{\mu-\mu'+\mu''-\mu'''}(\sigma) \,d\sigma
\end{eqnarray}
Then we write 
\begin{eqnarray}\label{crazy2} 
 \int_{\Sigma}  b_{12}^2 (\widetilde \partial_1 F)^2(\widetilde \partial_2 F)^2\, d\sigma & = & \frac{(4\pi^2)^2}{M^2} \sum_{\mu,\mu',\mu'',\mu'''} \mu_1 \mu'_1 \mu''_2 \mu'''_2 a_\mu \overline{a_{\mu'}} a_{\mu''} \overline{a_{\mu'''}} \int_{\Sigma} b_{12}(\sigma)^2 \e_{\mu-\mu'+\mu''-\mu'''}(\sigma)\,d\sigma \cr
 & = & \frac{(4\pi^2)^2}{M^2} \sum_{\mu,\mu'} \mu_1^2 (\mu'_2)^2 |a_\mu |^2 |a_{\mu'}|^2 \int_{\Sigma} b_{12}(\sigma)^2 \,d\sigma \cr  
 && + 2 \cdot \frac{(4\pi^2)^2}{M^2} \sum_{\mu,\mu'} \mu_1 \mu_2 \mu'_1 \mu'_2 |a_\mu |^2 |a_{\mu'}|^2 \int_{\Sigma} b_{12}(\sigma)^2 \,d\sigma\cr  
 && - 3 \frac{(4\pi^2)^2}{M^2} \sum_{\mu} \mu_1^2 \mu_2^2 |a_\mu |^4\int_{\Sigma} b_{12}(\sigma)^2 \,d\sigma\cr
 && +\frac{(4\pi^2)^2}{M^2} \sum_{(\mu,\mu',\mu'',\mu''')\in \mathcal X(4)} \mu_1 \mu'_1 \mu''_2 \mu'''_2 a_\mu \overline{a_{\mu'}} a_{\mu''} \overline{a_{\mu'''}} \int_{\Sigma} b_{12}(\sigma)^2 \,d\sigma \cr
 &&+\frac{(4\pi^2)^2}{M^2} \sum_{(\mu,\mu',\mu'',\mu''')\notin \mathcal C(4)} \mu_1 \mu'_1 \mu''_2 \mu'''_2 a_\mu \overline{a_{\mu'}} a_{\mu''} \overline{a_{\mu'''}}\cr
&& \times \int_{\Sigma} b_{12}(\sigma)^2 \e_{\mu-\mu'+\mu''-\mu'''}(\sigma)\,d\sigma
\end{eqnarray}
and we define
\begin{eqnarray}
R_{1212}(3)&:=&\frac{(4\pi^2)^2}{M^2} \sum_{(\mu,\mu',\mu'',\mu''')\in \mathcal X(4)} \mu_1 \mu'_1 \mu''_2 \mu'''_2 a_\mu \overline{a_{\mu'}} a_{\mu''} \overline{a_{\mu'''}} \int_{\Sigma} b_{12}(\sigma)^2 \,d\sigma \cr
 &&+\frac{(4\pi^2)^2}{M^2} \sum_{(\mu,\mu',\mu'',\mu''')\notin \mathcal C(4)} \mu_1 \mu'_1 \mu''_2 \mu'''_2 a_\mu \overline{a_{\mu'}} a_{\mu''} \overline{a_{\mu'''}}\cr
&& \times \int_{\Sigma} b_{12}(\sigma)^2 \e_{\mu-\mu'+\mu''-\mu'''}(\sigma)\,d\sigma.
\end{eqnarray}
Moreover 
\begin{eqnarray}\label{crazy3}
\int_{\Sigma}  b_1 b_{12} (\widetilde \partial_1 F)^3 (\widetilde \partial_2 F)\, d\sigma & = & \frac{(4\pi^2)^2}{M^2} \sum_{\mu,\mu',\mu'',\mu'''} \mu_1 \mu'_1 \mu''_1 \mu'''_2 a_\mu \overline{a_{\mu'}} a_{\mu''} \overline{a_{\mu'''}}\cr
&&\times \int_{\Sigma} b_1(\sigma)  b_{12}(\sigma) \e_{\mu-\mu'+\mu''-\mu'''}(\sigma)\,d\sigma \cr
 & = & 3 \frac{(4\pi^2)^2}{M^2} \sum_{\mu,\mu'} \mu_1^2 \mu_1' \mu'_2 |a_\mu |^2 |a_{\mu'}|^2 \int_{\Sigma} b_1(\sigma)  b_{12}(\sigma) \,d\sigma \cr
 && -  3 \frac{(4\pi^2)^2}{M^2} \sum_{\mu} \mu_1^3 \mu_2 |a_\mu |^4 \int_{\Sigma} b_1(\sigma)  b_{12}(\sigma) \,d\sigma \cr
&& + \frac{(4\pi^2)^2}{M^2} \sum_{(\mu,\mu',\mu'',\mu''')\in \mathcal X(4)} \mu_1 \mu'_1 \mu''_1 \mu'''_2 a_\mu \overline{a_{\mu'}} a_{\mu''} \overline{a_{\mu'''}}\cr
&&\times \int_{\Sigma} b_1(\sigma)  b_{12}(\sigma) \,d\sigma \cr
&& + \frac{(4\pi^2)^2}{M^2} \sum_{(\mu,\mu',\mu'',\mu''')\notin \mathcal C(4)} \mu_1 \mu'_1 \mu''_1 \mu'''_2 a_\mu \overline{a_{\mu'}} a_{\mu''} \overline{a_{\mu'''}}\cr
&&\times \int_{\Sigma} b_1(\sigma)  b_{12}(\sigma) \e_{\mu-\mu'+\mu''-\mu'''}(\sigma)\,d\sigma
\end{eqnarray}
and as before we set
\begin{eqnarray}
R_{1112}(3)&:=& \frac{(4\pi^2)^2}{M^2} \sum_{(\mu,\mu',\mu'',\mu''')\in \mathcal X(4)} \mu_1 \mu'_1 \mu''_1 \mu'''_2 a_\mu \overline{a_{\mu'}} a_{\mu''} \overline{a_{\mu'''}}\cr
&&\times \int_{\Sigma} b_1(\sigma)  b_{12}(\sigma) \,d\sigma \cr
&& + \frac{(4\pi^2)^2}{M^2} \sum_{(\mu,\mu',\mu'',\mu''')\notin \mathcal C(4)} \mu_1 \mu'_1 \mu''_1 \mu'''_2 a_\mu \overline{a_{\mu'}} a_{\mu''} \overline{a_{\mu'''}}\cr
&&\times \int_{\Sigma} b_1(\sigma)  b_{12}(\sigma) \e_{\mu-\mu'+\mu''-\mu'''}(\sigma)\,d\sigma.
\end{eqnarray}
Finally we have 
\begin{eqnarray}\label{crazy4}
 \int_{\Sigma}  b_1 b_{23} (\widetilde \partial_1 F)^2 (\widetilde \partial_2 F)(\widetilde \partial_3 F)\, d\sigma & = & \frac{(4\pi^2)^2}{M^2} \sum_{\mu,\mu',\mu'',\mu'''} \mu_1 \mu'_1 \mu''_2 \mu'''_3 a_\mu \overline{a_{\mu'}} a_{\mu''} \overline{a_{\mu'''}} \cr
 &&\times \int_{\Sigma} b_1(\sigma)  b_{23}(\sigma) \e_{\mu-\mu'+\mu''-\mu'''}(\sigma)\,d\sigma \cr
 & =&   \frac{(4\pi^2)^2}{M^2} \sum_{\mu,\mu'} \mu_1^2 \mu_2' \mu'_3 |a_\mu |^2 |a_{\mu'}|^2 \int_{\Sigma} b_1(\sigma)  b_{23}(\sigma) \,d\sigma \cr  
 && +2 \frac{(4\pi^2)^2}{M^2} \sum_{\mu,\mu'} \mu_1 \mu_2 \mu_1' \mu'_3 |a_\mu |^2 |a_{\mu'}|^2 \int_{\Sigma} b_1(\sigma)  b_{23}(\sigma) \,d\sigma \cr
 && -3 \frac{(4\pi^2)^2}{M^2} \sum_{\mu} \mu_1^2 \mu_2  \mu_3 |a_\mu |^4 \int_{\Sigma} b_1(\sigma)  b_{23}(\sigma) \,d\sigma\cr
 && + \frac{(4\pi^2)^2}{M^2} \sum_{(\mu,\mu',\mu'',\mu''')\in \mathcal X(4)} \mu_1 \mu'_1 \mu''_2 \mu'''_3 a_\mu \overline{a_{\mu'}} a_{\mu''} \overline{a_{\mu'''}} \cr
 &&\times \int_{\Sigma} b_1(\sigma)  b_{23}(\sigma)\,d\sigma \cr
 && + \frac{(4\pi^2)^2}{M^2} \sum_{(\mu,\mu',\mu'',\mu''')\notin \mathcal C(4)} \mu_1 \mu'_1 \mu''_2 \mu'''_3 a_\mu \overline{a_{\mu'}} a_{\mu''} \overline{a_{\mu'''}} \cr
 &&\times \int_{\Sigma} b_1(\sigma)  b_{23}(\sigma) \e_{\mu-\mu'+\mu''-\mu'''}(\sigma)\,d\sigma
\end{eqnarray}
and 
\begin{eqnarray}
R_{1123}(3)&:=&\frac{(4\pi^2)^2}{M^2} \sum_{(\mu,\mu',\mu'',\mu''')\in \mathcal X(4)} \mu_1 \mu'_1 \mu''_2 \mu'''_3 a_\mu \overline{a_{\mu'}} a_{\mu''} \overline{a_{\mu'''}} \cr
 &&\times \int_{\Sigma} b_1(\sigma)  b_{23}(\sigma)\,d\sigma \cr
 && + \frac{(4\pi^2)^2}{M^2} \sum_{(\mu,\mu',\mu'',\mu''')\notin \mathcal C(4)} \mu_1 \mu'_1 \mu''_2 \mu'''_3 a_\mu \overline{a_{\mu'}} a_{\mu''} \overline{a_{\mu'''}} \cr
 &&\times \int_{\Sigma} b_1(\sigma)  b_{23}(\sigma) \e_{\mu-\mu'+\mu''-\mu'''}(\sigma)\,d\sigma.
\end{eqnarray}
Define $R_{ijkl}$ for $i\le j, k\le l$ all in $\lbrace 1,2,3 \rbrace$ analogously. 
Substituting (\ref{crazy1})-(\ref{crazy4}) and (\ref{hola4}) into (\ref{ccc1}) we get 
\begin{eqnarray}\label{2020}
&& \int_{\Sigma} [ (Z_1(\sigma)^2 + Z_2(\sigma)^2)^2 - 8 -8(Z_1(\sigma)^2 + Z_2(\sigma)^2) + 16]\, d\sigma  \cr
 && = 3 \frac{(4\pi^2)^2}{M^2} \sum_{\mu,\mu'} \mu_1^2 (\mu'_1)^2 |a_\mu |^2 |a_{\mu'}|^2 \int_{\Sigma} (1 - n_1(\sigma)^2)^2 \,d\sigma - 3 \frac{(4\pi^2)^2}{M^2} \sum_{\mu} \mu_1^4  |a_\mu |^4 \int_{\Sigma} (1 - n_1(\sigma)^2)^2 \,d\sigma \cr 
&& + 3 \frac{(4\pi^2)^2}{M^2} \sum_{\mu,\mu'} \mu_2^2 (\mu'_2)^2 |a_\mu |^2 |a_{\mu'}|^2 \int_{\Sigma} (1 - n_2(\sigma)^2)^2 \,d\sigma - 3 \frac{(4\pi^2)^2}{M^2} \sum_{\mu} \mu_2^4  |a_\mu |^4 \int_{\Sigma} (1 - n_2(\sigma)^2)^2 \,d\sigma \cr
 && + 3 \frac{(4\pi^2)^2}{M^2} \sum_{\mu,\mu'} \mu_3^2 (\mu'_3)^2 |a_\mu |^2 |a_{\mu'}|^2 \int_{\Sigma} (1 - n_3(\sigma)^2)^2 \,d\sigma - 3 \frac{(4\pi^2)^2}{M^2} \sum_{\mu} \mu_3^4  |a_\mu |^4 \int_{\Sigma} (1 - n_3(\sigma)^2)^2 \,d\sigma \cr
 && + \frac{(4\pi^2)^2}{M^2} \sum_{\mu,\mu'} \mu_1^2 (\mu'_2)^2 |a_\mu |^2 |a_{\mu'}|^2 \int_{\Sigma} ((-2n_1n_2)^2 + 2(1 - n_1(\sigma)^2) (1 - n_2(\sigma)^2)) \,d\sigma  \cr
 &&+  2 \cdot \frac{(4\pi^2)^2}{M^2} \sum_{\mu,\mu'} \mu_1 \mu_2 \mu'_1 \mu'_2 |a_\mu |^2 |a_{\mu'}|^2 \int_{\Sigma} ((-2n_1n_2)^2 + 2(1 - n_1(\sigma)^2) (1 - n_2(\sigma)^2))\,d\sigma\cr  
 && - 3 \frac{(4\pi^2)^2}{M^2} \sum_{\mu} \mu_1^2 \mu_2^2 |a_\mu |^4\int_{\Sigma} ((-2n_1n_2)^2 + 2(1 - n_1(\sigma)^2) (1 - n_2(\sigma)^2)) \,d\sigma \cr
 && + \frac{(4\pi^2)^2}{M^2} \sum_{\mu,\mu'} \mu_1^2 (\mu'_3)^2 |a_\mu |^2 |a_{\mu'}|^2 \int_{\Sigma} ((-2n_1n_3)(\sigma)^2 + 2(1 - n_1(\sigma)^2) (1-n_3^2)) \,d\sigma  \cr
 &&+  2 \cdot \frac{(4\pi^2)^2}{M^2} \sum_{\mu,\mu'} \mu_1 \mu_3 \mu'_1 \mu'_3 |a_\mu |^2 |a_{\mu'}|^2 \int_{\Sigma} ((-2n_1n_3)(\sigma)^2 + 2(1 - n_1(\sigma)^2)(1 - n_3(\sigma)^2))\,d\sigma\cr  
 && - 3 \frac{(4\pi^2)^2}{M^2} \sum_{\mu} \mu_1^2 \mu_3^2 |a_\mu |^4\int_{\Sigma} ((-2n_1n_3)(\sigma)^2 + 2(1 - n_1(\sigma)^2) (1 - n_3(\sigma)^2)) \,d\sigma \cr
 && + \frac{(4\pi^2)^2}{M^2} \sum_{\mu,\mu'} \mu_2^2 (\mu'_3)^2 |a_\mu |^2 |a_{\mu'}|^2 \int_{\Sigma} ((-2n_2n_3)(\sigma)^2 + 2(1 - n_2(\sigma)^2) (1 - n_3(\sigma)^2)) \,d\sigma  \cr
 &&+  2 \cdot \frac{(4\pi^2)^2}{M^2} \sum_{\mu,\mu'} \mu_2 \mu_3 \mu'_2\mu'_3 |a_\mu |^2 |a_{\mu'}|^2 \int_{\Sigma} ((-2n_2n_3)(\sigma)^2 + 2(1 - n_2(\sigma)^2)(1 - n_3(\sigma)^2))\,d\sigma\cr  
 && - 3 \frac{(4\pi^2)^2}{M^2} \sum_{\mu} \mu_2^2 \mu_3^2 |a_\mu |^4\int_{\Sigma} ((-2n_2n_3)(\sigma)^2 + 2(1 - n_2(\sigma)^2) (1 - n_3(\sigma)^2)) \,d\sigma \cr
 && + 3 \frac{(4\pi^2)^2}{M^2} \sum_{\mu,\mu'} \mu_1^2 \mu_1' \mu'_2 |a_\mu |^2 |a_{\mu'}|^2 \int_{\Sigma} 2(1 - n_1(\sigma)^2)  (-2n_1n_2) \,d\sigma \cr  
 && -  3 \frac{(4\pi^2)^2}{M^2} \sum_{\mu} \mu_1^3 \mu_2 |a_\mu |^4 \int_{\Sigma}2 (1 - n_1(\sigma)^2)  (-2n_1n_2) \,d\sigma \cr
  && + 3 \frac{(4\pi^2)^2}{M^2} \sum_{\mu,\mu'} \mu_1^2 \mu_1' \mu'_3 |a_\mu |^2 |a_{\mu'}|^2 \int_{\Sigma} 2(1 - n_1(\sigma)^2) (-2n_1n_3) \,d\sigma \cr  
 && -  3 \frac{(4\pi^2)^2}{M^2} \sum_{\mu} \mu_1^3 \mu_3 |a_\mu |^4 \int_{\Sigma} 2(1 - n_1(\sigma)^2)  (-2n_1n_3) \,d\sigma \cr
  && + 3 \frac{(4\pi^2)^2}{M^2} \sum_{\mu,\mu'} \mu_2^2 \mu_2' \mu'_3 |a_\mu |^2 |a_{\mu'}|^2 \int_{\Sigma} 2(1 - n_2(\sigma)^2)  (-2n_2n_3) \,d\sigma \cr 
  &&-  3 \frac{(4\pi^2)^2}{M^2} \sum_{\mu} \mu_2^3 \mu_3 |a_\mu |^4 \int_{\Sigma} 2(1 - n_2(\sigma)^2)  (-2n_2n_3) \,d\sigma \cr
  && + 3 \frac{(4\pi^2)^2}{M^2} \sum_{\mu,\mu'} \mu_2^2 \mu_2' \mu'_1 |a_\mu |^2 |a_{\mu'}|^2 \int_{\Sigma} 2(1 - n_2(\sigma)^2)  (-2n_1n_2) \,d\sigma  \cr
  && -  3 \frac{(4\pi^2)^2}{M^2} \sum_{\mu} \mu_2^3 \mu_1 |a_\mu |^4 \int_{\Sigma} 2(1 - n_2(\sigma)^2) (-2n_1n_2) \,d\sigma \cr
   && + 3 \frac{(4\pi^2)^2}{M^2} \sum_{\mu,\mu'} \mu_3^2 \mu_3' \mu'_1 |a_\mu |^2 |a_{\mu'}|^2 \int_{\Sigma} 2(1 - n_3(\sigma)^2)  b_{13}(\sigma) \,d\sigma \cr 
   &&-  3 \frac{(4\pi^2)^2}{M^2} \sum_{\mu} \mu_3^3 \mu_1 |a_\mu |^4 \int_{\Sigma} 2(1 - n_3(\sigma)^2) (-2n_1n_3) \,d\sigma \cr
   && + 3 \frac{(4\pi^2)^2}{M^2} \sum_{\mu,\mu'} \mu_3^2 \mu_3' \mu'_2 |a_\mu |^2 |a_{\mu'}|^2 \int_{\Sigma} 2(1 - n_3(\sigma)^2)  b_{23}(\sigma) \,d\sigma \cr 
   &&-  3 \frac{(4\pi^2)^2}{M^2} \sum_{\mu} \mu_3^3 \mu_2 |a_\mu |^4 \int_{\Sigma} 2(1 - n_3(\sigma)^2)  (-2n_2n_3) \,d\sigma \cr
    && +   \frac{(4\pi^2)^2}{M^2} \sum_{\mu,\mu'} \mu_1^2 \mu_2' \mu'_3 |a_\mu |^2 |a_{\mu'}|^2 \int_{\Sigma} 2(1 - n_1(\sigma)^2) (-2n_2n_3) + 2(-2n_1n_2)(-2n_1n_3) \,d\sigma \cr  
 && +2 \frac{(4\pi^2)^2}{M^2} \sum_{\mu,\mu'} \mu_1 \mu_2 \mu_1' \mu'_3 |a_\mu |^2 |a_{\mu'}|^2 \int_{\Sigma} 2(1 - n_1(\sigma)^2)  (-2n_2n_3) + 2(-2n_1n_2)b_{13} \,d\sigma \cr
 && -3 \frac{(4\pi^2)^2}{M^2} \sum_{\mu} \mu_1^2 \mu_2  \mu_3 |a_\mu |^4 \int_{\Sigma} 2(1 - n_1(\sigma)^2)  (-2n_2n_3) + 2(-2n_1n_2) (-2n_1n_3)  \,d\sigma \cr
 && +   \frac{(4\pi^2)^2}{M^2} \sum_{\mu,\mu'} \mu_2^2 \mu_1' \mu'_3 |a_\mu |^2 |a_{\mu'}|^2 \int_{\Sigma} 2(1 - n_2(\sigma)^2)  b_{13}(\sigma) + 2(-2n_1n_2)(-2n_2n_3) \,d\sigma \cr  
 && +2 \frac{(4\pi^2)^2}{M^2} \sum_{\mu,\mu'} \mu_2 \mu_1 \mu_2' \mu'_3 |a_\mu |^2 |a_{\mu'}|^2 \int_{\Sigma} 2(1 - n_2(\sigma)^2)  (-2n_1n_3)  + 2(-2n_1n_2)(-2n_2n_3) \,d\sigma \cr
 && -3 \frac{(4\pi^2)^2}{M^2} \sum_{\mu} \mu_2^2 \mu_1  \mu_3 |a_\mu |^4 \int_{\Sigma} 2(1 - n_2(\sigma)^2)  b_{13}(\sigma) + 2(-2n_1n_2)(-2n_2n_3) \,d\sigma \cr
 && +   \frac{(4\pi^2)^2}{M^2} \sum_{\mu,\mu'} \mu_3^2 \mu_1' \mu'_2 |a_\mu |^2 |a_{\mu'}|^2 \int_{\Sigma} 2(1 - n_3(\sigma)^2)  (-2n_1n_2) + 2(-2n_1n_3) (-2n_2n_3) \,d\sigma \cr  
 && +2 \frac{(4\pi^2)^2}{M^2} \sum_{\mu,\mu'} \mu_3 \mu_1 \mu_3' \mu'_2 |a_\mu |^2 |a_{\mu'}|^2 \int_{\Sigma} 2(1 - n_3(\sigma)^2)  (-2n_1n_2) + 2(-2n_1n_3) (-2n_2n_3) \,d\sigma \cr
 && -3 \frac{(4\pi^2)^2}{M^2} \sum_{\mu} \mu_3^2 \mu_1  \mu_2 |a_\mu |^4 \int_{\Sigma} 2(1 - n_3(\sigma)^2) (-2n_1n_2) + 2(-2n_1n_3) (-2n_2n_3) \,d\sigma - 8\int_{\Sigma} d\sigma \cr
&& - 8 \cdot 3  \frac{1}{N} \sum_{\mu} (|a_\mu |^2  -1)\left (|\Sigma |  - \int_\Sigma \left \langle \frac{\mu}{|\mu|}, n(\sigma)\right \rangle^2 \, d\sigma \right )\cr
&&- 8\cdot 3\frac{1}{N} \sum_{\mu\ne\mu'} a_\mu \overline{a_{\mu'}}  \int_{\Sigma}\left (\left \langle \frac{\mu}{|\mu|}, \frac{\mu'}{|\mu'|} \right \rangle - \left \langle \frac{\mu}{|\mu|}, n(\sigma) \right \rangle \left \langle \frac{\mu'}{|\mu'|}, n(\sigma) \right \rangle \right ) e_{\mu}(\sigma) e_{-\mu'}(\sigma)\,d\sigma\cr
&& + \sum_{i\le j} R_{ijij}(3) + 2 \sum_{\substack{i\le j, k\le l\\ (i,j) < (k,l)}} R_{ijkl}(3) - 8 R(0).
\end{eqnarray}
Equation (\ref{2020}) simplifies to 
\begin{eqnarray}\label{super1}
&& \int_{\Sigma} [ (Z_1(\sigma)^2 + Z_2(\sigma)^2)^2 - 8 -8(Z_1(\sigma)^2 + Z_2(\sigma)^2) + 16]\, d\sigma  \cr
 && =  \frac{9}{N^2}\sum_{\mu,\mu'}(|a_\mu |^2 - 1)(|a_{\mu'} |^2 - 1)\int_\Sigma \left ( \left \langle \frac{\mu}{|\mu|}, \frac{\mu'}{|\mu'|} \right \rangle - \left \langle \frac{\mu}{|\mu|}, n(\sigma) \right \rangle \left \langle \frac{\mu'}{|\mu'|}, n(\sigma) \right \rangle \right )^2\,d\sigma \cr
 &&-\frac{9}{N^2}\sum_{\mu}|a_\mu |^4 \int_\Sigma \left ( 1 - \left \langle \frac{\mu}{|\mu|}, n(\sigma) \right \rangle^2 \right )^2\,d\sigma\cr
 && -   \frac{8 \cdot 3}{N} \sum_{\mu} (|a_\mu |^2  -1)\left (|\Sigma |  - \int_\Sigma \left \langle \frac{\mu}{|\mu|}, n(\sigma)\right \rangle^2 \, d\sigma \right )\cr
&&- \frac{8\cdot 3}{N} \sum_{\mu\ne\mu'} a_\mu \overline{a_{\mu'}}  \int_{\Sigma}\left (\left \langle \frac{\mu}{|\mu|}, \frac{\mu'}{|\mu'|} \right \rangle - \left \langle \frac{\mu}{|\mu|}, n(\sigma) \right \rangle \left \langle \frac{\mu'}{|\mu'|}, n(\sigma) \right \rangle \right ) e_{\mu}(\sigma) e_{-\mu'}(\sigma)\,d\sigma\cr
&& + \sum_{i\le j} R_{ijij}(3) + 2 \sum_{\substack{i\le j, k\le l\\ (i,j) < (k,l)}} R_{ijkl}(3) - 8 R(0).
\end{eqnarray}
It follows that the fourth chaotic component can be rewritten as 
\begin{eqnarray}\label{4chaos2}
\mathcal L[4] &=& \sqrt M \Big ( \frac{\beta_4\alpha_{0,0}}{4!}\int_{\Sigma} H_4(F(\sigma))\,d\sigma\cr   
&&+\frac{\beta_2\alpha_{2,0}}{2! 2!}\int_{\Sigma} H_2(F(\sigma)) H_2(Z_1(\sigma))\,d\sigma +\frac{\beta_2\alpha_{0,2}}
{2! 2!}\int_{\Sigma} H_2(F(\sigma)) H_2(Z_2(
\sigma))\,d\sigma\cr
&&+ \frac{\beta_0\alpha_{2,2}}{2! 2!} \int_{\Sigma} H_2(Z_1(
\sigma)) H_2(Z_2(
\sigma))\,d\sigma\cr
&&+ \frac{\beta_0 \alpha_{4,0}}{4!}\int_{\Sigma} H_4(Z_1(\sigma))\,d\sigma + \frac{\beta_0 \alpha_{4,0}}{4!}\int_{\Sigma} H_4(Z_2(\sigma))\,d\sigma \Big )\cr 
& =& \sqrt{\frac{4\pi^2m}{3}} \Big [  \frac{3}{16} \cdot |\Sigma | \frac{1}{N^2}
\sum_{\mu,\mu'}  (|a_\mu |^2 - 1) (|a_{\mu'} |^2 - 1)
  -  \frac{1}{16} \cdot  3 
\frac{1}{N^2} \sum_{\mu
} |a_\mu |^4
\cr 
&& - \frac{1}{16} \cdot  3 \frac{1}{N^2} \sum_{\mu, \mu'} (|a_\mu |^2-1)( |a_{\mu'}|^2 -1)\Big (|\Sigma | -   \int_\Sigma \left \langle \frac{\mu}{|\mu|}, n(\sigma)\right \rangle^2 \, d\sigma        \Big ) \cr 
 && + \frac{1}{16} \cdot 3 \frac{1}{N^2} \sum_{\mu} |a_\mu |^4\Big (|\Sigma | -   \int_\Sigma \left \langle \frac{\mu}{|\mu|}, n(\sigma)\right \rangle^2 \, d\sigma        \Big ) \cr 
&& -\frac{1}{16\cdot 8} \frac{9}{N^2}\sum_{\mu,\mu'}(|a_\mu |^2 - 1)(|a_{\mu'} |^2 - 1)\int_\Sigma \left ( \left \langle \frac{\mu}{|\mu|}, \frac{\mu'}{|\mu'|} \right \rangle - \left \langle \frac{\mu}{|\mu|}, n(\sigma) \right \rangle \left \langle \frac{\mu'}{|\mu'|}, n(\sigma) \right \rangle \right )^2\,d\sigma
\cr 
&& +\frac{1}{16\cdot 8} \cdot \frac{9}{N^2}\sum_{\mu}|a_\mu |^4\int_\Sigma \left ( 1 - \left \langle \frac{\mu}{|\mu|}, n(\sigma) \right \rangle^2 \right )^2\,d\sigma
\cr
&& + a_0 R(0) + a_1 R(1) + a_2 R(2) + \sum_{i\le j}a_{ijij}R_{ijij}(3) + \sum_{\substack{i\le j, k\le l\\(i,j)< (k,l)}} a_{ijkl}R_{ijkl}(3)\Big]\cr
& = &\sqrt{\frac{4\pi^2m}{3}} \Big [\frac{1}{16} \cdot  3 \frac{1}{N^2} \sum_{\mu, \mu'} (|a_\mu |^2-1)( |a_{\mu'}|^2 -1)\Big (\int_\Sigma \left \langle \frac{\mu}{|\mu|}, n(\sigma)\right \rangle^2 \, d\sigma        \Big ) \cr 
 && - \frac{1}{16} \cdot 3 \frac{1}{N^2} \sum_{\mu} |a_\mu |^4\Big (\int_\Sigma \left \langle \frac{\mu}{|\mu|}, n(\sigma)\right \rangle^2 \, d\sigma        \Big ) \cr 
&& -\frac{1}{16\cdot 8} \frac{9}{N^2}\sum_{\mu,\mu'}(|a_\mu |^2 - 1)(|a_{\mu'} |^2 - 1)\int_{\Sigma} \left ( \left \langle \frac{\mu}{|\mu|}, \frac{\mu'}{|\mu'|}\right \rangle - 
 \left \langle \frac{\mu}{|\mu|}, n(\sigma)\right \rangle \left \langle \frac{\mu'}{|\mu'|}, n(\sigma)\right \rangle\right )^2\,d\sigma
\cr 
&& +\frac{1}{16\cdot 8} \cdot \frac{9}{N^2}\sum_{\mu}|a_\mu |^4\int_{\Sigma} \left (1 - \left \langle \frac{\mu}{|\mu|}, n(\sigma)\right \rangle^2  \right )^2\,d\sigma\cr 
&&a_0 R(0) + a_1 R(1) + a_2 R(2) + \sum_{i\le j}a_{ijij}R_{ijij}(3) + \sum_{\substack{i\le j, k\le l\\(i,j)< (k,l)}} a_{ijkl}R_{ijkl}(3)
\Big],
\end{eqnarray}
where the coefficients $a_0,a_1,a_2,a_{ijkl}$ are universal constants that can be made explicit. 
We can rewrite it as 
%
%
\begin{eqnarray*}
\mathcal L[4]  &=& \sqrt{\frac{4\pi^2m}{3}}   \frac{3}{16 \cdot 8} \frac{1}{N^2}  
\sum_{\mu,\mu'}  (|a_\mu |^2 - 1) (|a_{\mu'} |^2 - 1) \times 
\cr 
&&  \times \Big [\int_\Sigma \Big ( -3 -9\left \langle \frac{\mu}{|\mu|}, n(\sigma)\right \rangle^2 \left \langle \frac{\mu'}{|\mu'|}, n(\sigma)\right \rangle^2 +14\left \langle \frac{\mu}{|\mu|}, n(\sigma)\right \rangle^2-6\left \langle \frac{\mu}{|\mu|}, \frac{\mu'}{|\mu'|}\right \rangle^2
\cr
&&+12\left \langle \frac{\mu}{|\mu|}, \frac{\mu'}{|\mu'|}\right \rangle \left \langle \frac{\mu}{|\mu|}, n\right \rangle \left \langle \frac{\mu'}{|\mu'|},n\right \rangle
\Big )d\sigma  \Big ] \cr
&&-\sqrt{\frac{4\pi^2m}{3}}   \frac{3}{16 \cdot 8} \frac{1}{N^2}  
\sum_{\mu}  (|a_\mu |^4) \times 
\cr 
&&  \times \Big [\int_\Sigma \Big ( -3 -9\left \langle \frac{\mu}{|\mu|}, n(\sigma)\right \rangle^2 \left \langle \frac{\mu'}{|\mu'|}, n(\sigma)\right \rangle^2 +14\left \langle \frac{\mu}{|\mu|}, n(\sigma)\right \rangle^2-6\left \langle \frac{\mu}{|\mu|}, \frac{\mu'}{|\mu'|}\right \rangle^2
\cr
&&+12\left \langle \frac{\mu}{|\mu|}, \frac{\mu'}{|\mu'|}\right \rangle \left \langle \frac{\mu}{|\mu|}, n\right \rangle \left \langle \frac{\mu'}{|\mu'|},n\right \rangle
\Big )d\sigma  \Big ] \cr
&& + \sqrt{\frac{4\pi^2m}{3}}\Big(a_0 R(0) + a_1 R(1) + a_2 R(2) + \sum_{i\le j}a_{ijij}R_{ijij}(3) + \sum_{\substack{i\le j, k\le l\\(i,j)< (k,l)}} a_{ijkl}R_{ijkl}(3)\Big ).
\end{eqnarray*}
Let us set 
\begin{equation}
    \mathcal L^b[4] =\mathcal L_m^b[4]:=\sqrt{\frac{4\pi^2m}{3}}\Big(a_0 R(0) + a_1 R(1) + a_2 R(2) + \sum_{i\le j}a_{ijij}R_{ijij}(3) + \sum_{\substack{i\le j, k\le l\\(i,j)< (k,l)}} a_{ijkl}R_{ijkl}(3)\Big ),
\end{equation}
then it is easy to check that 
\begin{equation}
    \Var(\mathcal L^b[4]) \ll m \cdot \max \left \lbrace \frac{|\mathcal X(4)|}{N^4}, \mathfrak{S}_4, \mathfrak{S}_2 \right \rbrace 
\end{equation}
so the proof is concluded.
\end{proof}

\bibliographystyle{alpha}
\bibliography{bibfile}

\begin{thebibliography}{MPRW16}

\bibitem[AW09]{azawsc}
Jean-Marc Aza{\"{\i}}s and Mario Wschebor.
\newblock {\em Level sets and extrema of random processes and fields}.
\newblock John Wiley \& Sons, Inc., Hoboken, NJ, 2009.

\bibitem[Ber02]{berry2}
Michael~V Berry.
\newblock Statistics of nodal lines and points in chaotic quantum billiards:
  perimeter corrections, fluctuations, curvature.
\newblock {\em Journal of Physics A: Mathematical and General}, 35(13):3025,
  2002.

\bibitem[BM19]{benmaf}
Jacques Benatar and Riccardo~W. Maffucci.
\newblock Random waves on {$\Bbb T^3$}: nodal area variance and lattice point
  correlations.
\newblock {\em Int. Math. Res. Not. IMRN}, (10):3032--3075, 2019.

\bibitem[BR11a]{brahpo}
Jean Bourgain and Ze{\'e}v Rudnick.
\newblock On the geometry of the nodal lines of eigenfunctions of the
  two-dimensional torus.
\newblock {\em Ann. Henri Poincar\'e}, 12(6):1027--1053, 2011.

\bibitem[BR11b]{brnoda}
Jean Bourgain and Ze{\'e}v Rudnick.
\newblock On the nodal sets of toral eigenfunctions.
\newblock {\em Invent. Math.}, 185(1):199--237, 2011.

\bibitem[BR12]{brgafa}
Jean Bourgain and Ze{\'e}v Rudnick.
\newblock Restriction of toral eigenfunctions to hypersurfaces and nodal sets.
\newblock {\em Geom. Funct. Anal.}, 22(4):878--937, 2012.

\bibitem[BSR12]{bosaru}
Jean Bourgain, Peter Sarnak, and Ze{\'e}v Rudnick.
\newblock Local statistics of lattice points on the sphere.
\newblock {\em Modern Trends in Constructive Function Theory, Contemp. Math},
  661:269--282, 2012.

\bibitem[Cam19]{Cam19}
Valentina Cammarota.
\newblock Nodal area distribution for arithmetic random waves.
\newblock {\em Trans. Amer. Math. Soc.}, 372(5):3539--3564, 2019.

\bibitem[CM18]{CM18}
Valentina Cammarota and Domenico Marinucci.
\newblock A quantitative central limit theorem for the {E}uler-{P}oincar\'{e}
  characteristic of random spherical eigenfunctions.
\newblock {\em Ann. Probab.}, 46(6):3188--3228, 2018.

\bibitem[dC76]{docarm}
Manfredo~P. do~Carmo.
\newblock {\em Differential geometry of curves and surfaces}.
\newblock Prentice-Hall, Inc., Englewood Cliffs, N.J., 1976.
\newblock Translated from the Portuguese.

\bibitem[EL16]{EL16}
Anne Estrade and Jos\'{e}~R. Le\'{o}n.
\newblock A central limit theorem for the {E}uler characteristic of a
  {G}aussian excursion set.
\newblock {\em Ann. Probab.}, 44(6):3849--3878, 2016.

\bibitem[HW79]{harwri}
G.~H. Hardy and E.~M. Wright.
\newblock {\em An introduction to the theory of numbers}.
\newblock The Clarendon Press, Oxford University Press, New York, fifth
  edition, 1979.

\bibitem[KKW13]{krkuwi}
Manjunath Krishnapur, P{\"a}r Kurlberg, and Igor Wigman.
\newblock Nodal length fluctuations for arithmetic random waves.
\newblock {\em Ann. of Math. (2)}, 177(2):699--737, 2013.

\bibitem[KL01]{kraleo}
Marie~F. Kratz and Jos\'{e}~R. Le\'{o}n.
\newblock Central limit theorems for level functionals of stationary {G}aussian
  processes and fields.
\newblock {\em J. Theoret. Probab.}, 14(3):639--672, 2001.

\bibitem[Maf19]{Maf18}
Riccardo~W. Maffucci.
\newblock Nodal intersections for arithmetic random waves against a surface.
\newblock {\em Ann. Henri Poincar\'{e}}, 20(11):3651--3691, 2019.

\bibitem[Maf20]{Maf20}
Riccardo~W. Maffucci.
\newblock Restriction of 3{D} arithmetic {L}aplace eigenfunctions to a plane.
\newblock {\em Electron. J. Probab.}, 25:Paper No. 60, 17, 2020.

\bibitem[Mar]{primar}
Domenico Marinucci.
\newblock Private communication.

\bibitem[MPRW16]{MPRW}
Domenico Marinucci, Giovanni Peccati, Maurizia Rossi, and Igor Wigman.
\newblock Non-universality of nodal length distribution for arithmetic random
  waves.
\newblock {\em Geometric and Functional Analysis}, 26(3):926--960, 2016.

\bibitem[MW11]{marwig11}
Domenico Marinucci and Igor Wigman.
\newblock On the area of excursion sets of spherical {G}aussian eigenfunctions.
\newblock {\em J. Math. Phys.}, 52(9):093301, 21, 2011.

\bibitem[Not21a]{Not20}
Massimo Notarnicola.
\newblock Fluctuations of nodal sets on the $3$-torus and general cancellation
  phenomena.
\newblock {\em ALEA, Lat. Am. J. Probab. Math. Stat.}, 18:1127--1194, 2021.

\bibitem[Not21b]{Not21}
Massimo Notarnicola.
\newblock Matrix hermite polynomials, random determinants and the geometry of
  {G}aussian fields.
\newblock {\em Preprint arXiv:2109.13749}, 2021.

\bibitem[NP12]{nopebook}
Ivan Nourdin and Giovanni Peccati.
\newblock {\em Normal approximations with {M}alliavin calculus}, volume 192 of
  {\em Cambridge Tracts in Mathematics}.
\newblock Cambridge University Press, Cambridge, 2012.
\newblock From Stein's method to universality.

\bibitem[PSB97]{pascbo}
Andrzej Palczewski, Jacques Schneider, and Alexandre~V. Bobylev.
\newblock A consistency result for a discrete-velocity model of the {B}oltzmann
  equation.
\newblock {\em SIAM J. Numer. Anal.}, 34(5):1865--1883, 1997.

\bibitem[PT05]{PT05}
Giovanni Peccati and Ciprian~A. Tudor.
\newblock Gaussian limits for vector-valued multiple stochastic integrals.
\newblock In {\em S\'{e}minaire de {P}robabilit\'{e}s {XXXVIII}}, volume 1857
  of {\em Lecture Notes in Math.}, pages 247--262. Springer, Berlin, 2005.

\bibitem[Rud]{prirud}
Ze\'ev Rudnick.
\newblock Private communication.

\bibitem[RW08]{rudwi2}
Ze{\'e}v Rudnick and Igor Wigman.
\newblock On the volume of nodal sets for eigenfunctions of the {L}aplacian on
  the torus.
\newblock {\em Ann. Henri Poincar\'e}, 9(1):109--130, 2008.

\bibitem[RW16]{rudwig}
Ze\'ev Rudnick and Igor Wigman.
\newblock Nodal intersections for random eigenfunctions on the torus.
\newblock {\em Amer. J. Math.}, 138(6):1605--1644, 2016.

\bibitem[RW18]{roswig}
Maurizia Rossi and Igor Wigman.
\newblock Asymptotic distribution of nodal intersections for arithmetic random
  waves.
\newblock {\em Nonlinearity}, 31(10):4472--4516, 2018.

\bibitem[RWY16]{RWY16}
Ze\'{e}v Rudnick, Igor Wigman, and Nadav Yesha.
\newblock Nodal intersections for random waves on the 3-dimensional torus.
\newblock {\em Ann. Inst. Fourier (Grenoble)}, 66(6):2455--2484, 2016.

\bibitem[Ser94]{serne2}
Edoardo Sernesi.
\newblock {\em Geometria 2: Edoardo Sernesi}.
\newblock Bollati Boringheri, 1994.

\bibitem[Ste86]{stein1}
Elias~M Stein.
\newblock Oscillatory integrals in {F}ourier analysis.
\newblock {\em Beijing lectures in harmonic analysis (Beijing, 1984)},
  112:307--355, 1986.

\end{thebibliography}
\end{document}